\documentclass{article}

\usepackage[centertags]{amsmath}
\usepackage{hyperref}
\usepackage{amsfonts}
\usepackage{amssymb}
\usepackage{amsthm}
\usepackage{newlfont}
\usepackage{amscd}
\usepackage{amsmath,amscd}
\usepackage{graphicx}
\usepackage[all]{xy}
\usepackage{verbatim}

\usepackage{color}

\newcommand{\barr}{\overline}
\newcommand{\adj}{\rightleftarrows}

\newcommand{\NN}{\mathbb{N}}
\newcommand{\cA}{\mathcal{A}}
\newcommand{\cB}{\mathcal{B}}
\newcommand{\cC}{\mathcal{C}}
\newcommand{\cD}{\mathcal{D}}

\newcommand{\cF}{\mathcal{F}}

\newcommand{\cI}{\mathcal{I}}
\newcommand{\cJ}{\mathcal{J}}
\newcommand{\cK}{\mathcal{K}}

\newcommand{\cM}{\mathcal{M}}

\newcommand{\cR}{\mathcal{R}}
\newcommand{\cS}{\mathcal{S}}
\newcommand{\cT}{\mathcal{T}}

\newcommand{\cW}{\mathcal{W}}

\def\Del{{\Delta}}

\newtheorem{thm}{Theorem}[section]

\newtheorem{cor}[thm]{Corollary}

\newtheorem{lem}[thm]{Lemma}
\newtheorem{prop}[thm]{Proposition}

\newtheorem{example}[thm]{Example}
\theoremstyle{definition}
\newtheorem{define}[thm]{Definition}
\theoremstyle{remark}
\newtheorem{rem}[thm]{Remark}

\DeclareMathOperator{\SPS}{SPS}
\DeclareMathOperator{\PS}{PS}
\DeclareMathOperator{\SSh}{SSh}
\DeclareMathOperator{\Sh}{Sh}
\DeclareMathOperator{\id}{id}
\DeclareMathOperator{\op}{op}
\DeclareMathOperator{\Sp}{Sp}
\DeclareMathOperator{\Lw}{Lw}
\DeclareMathOperator{\R}{R}
\DeclareMathOperator{\h}{h}
\DeclareMathOperator{\Ob}{Ob}
\DeclareMathOperator{\Ex}{Ex}

\DeclareMathOperator{\Spec}{Spec}

\DeclareMathOperator{\Set}{Set}

\DeclareMathOperator{\Pro}{Pro}
\DeclareMathOperator{\Ind}{Ind}
\DeclareMathOperator{\Map}{Map}
\DeclareMathOperator{\Ho}{Ho}
\DeclareMathOperator{\Hom}{Hom}
\DeclareMathOperator{\Mor}{Mor}
\DeclareMathOperator{\Ab}{Ab}
\DeclareMathOperator{\prehocolim}{hocolim}
\DeclareMathOperator{\precolim}{colim}

\def\CM{\mathtt{CM}}
\def\Csep{\mathtt{SC^*}}

\def\colim{\mathop{\precolim}}
\def\hocolim{\mathop{\prehocolim}}

\def\lrar{\longrightarrow}

\def\x{\stackrel}

\def \mcal{\mathcal}

\DeclareFontEncoding{OT2}{}{} 
\DeclareTextFontCommand{\textcyr}{\fontencoding{OT2}\fontfamily{wncyr}\fontseries{m}\fontshape{n}\selectfont}

\begin{document}
\title{A Projective Model Structure on Pro-Simplicial Sheaves, and the Relative \'Etale Homotopy Type}

\author{
Ilan Barnea \footnote{The first author was supported by the Alexander von Humboldt Foundation (Humboldt Professorship of Michael Weiss).} \and
Tomer M. Schlank \footnote{The second author was supported by the Simons Fellowship in the Department of Mathematics of the Massachusetts Institute of Technology.}}

\maketitle

\begin{abstract}
In this work we shall introduce a new model structure on the category of pro-simplicial sheaves, which is very convenient for the study of \'etale homotopy. Using this model structure we define a pro-space associated to a topos, as a result of applying a derived functor. We show that our construction lifts Artin and Mazur's \'etale homotopy type \cite{AM} in the relevant special case. Our definition extends naturally to a relative notion, namely, a pro-object associated to a map of topoi. This relative notion lifts the relative \'etale homotopy type that was used in \cite{HaSc} for the study of obstructions to the existence of rational points. This relative notion enables to generalize these homotopical obstructions from fields to general base schemas and general maps of topoi.

Our model structure is constructed using a general theorem that we prove. Namely, we introduce a much weaker structure than a model category, which we call a ``weak fibration category". Our theorem says that a weak fibration category can be ``completed" into a full model category structure on its pro-category, provided it satisfies some additional technical requirements. Our model structure is obtained by applying this result to the weak fibration category of simplicial sheaves over a Grothendieck site, where the weak equivalences and the fibrations are local in the sense of Jardine \cite{Jar}.
\end{abstract}

\tableofcontents

\section{Introduction: Weak  Fibration Categories}

Model categories, introduced by Quillen in \cite{Qui}, provide a very general context in which it is possible to set up the basic machinery of homotopy theory. However, the structure of a model category is not always available. The structure of a model category is determined by the classes of weak equivalences and fibrations (since the class of cofibrations is then determined by a left lifting property). There are situations in which there is a natural definition of weak equivalences and fibrations; however, the resulting structure is not a model category. A notable example is the category of simplicial sheaves over a Grothendieck site, where the weak equivalences and the fibrations are local in the sense of Jardine \cite{Jar}.

In this paper we introduce a much weaker and easy to verify structure than a model category, which we call a ``weak fibration category". Our main theorem (Theorem \ref{t:model_big}) says that a weak fibration category can be
``completed" into a full model category structure on its pro-category, provided it satisfies a property which we call ``homotopically small", and the pro-category satisfies a certain two out of three property.
We also define the notion of a weak right Quillen functor between weak fibration categories and show that it induces a right Quillen functor between the model structures on the pro-categories.

The notion of a weak fibration category is closely related to K. S. Brown's notion of a ``category of fibrant objects" (\cite{Bro}) and Baues's notion of a ``fibration category" (\cite{Bau}). These notions were introduced as a more flexible structure than a model category in which to do abstract homotopy theory.

We now give the exact definition:

\begin{define}
Let $\mcal{C}$ be a category with finite limits, and let $\mcal{M}\subseteq\mcal{C}$ be a subcategory. We say
that $\mcal{M}$ is \emph{closed under base change} if whenever we have a pullback square
\[
\xymatrix{A\ar[d]^g\ar[r] & B\ar[d]^f\\
C\ar[r] & D}
\]
such that $f$ is in $\mcal{M}$, then $g$ is in $\mcal{M}$.
\end{define}

\begin{define}\label{d:weak_fib}
A \emph{weak fibration category} is a category $\mcal{C}$ with an additional
structure of two subcategories
$$\mcal{F}, \mcal{W} \subseteq \mcal{C}$$
that contain all the isomorphisms, such that the following conditions are satisfied:
\begin{enumerate}
\item $\mcal{C}$ has all finite limits.
\item $\mcal{W}$ has the two out of three property.
\item The subcategories $\mcal{F}$ and $\mcal{F}\cap \mcal{W}$ are closed under base change.
\item Every map $A\to B $ in $\mcal{C}$ can be factored as $A\xrightarrow{f} C\xrightarrow{g} B $,
where $f$ is in $\mcal{W}$ and $g$ is in $\mcal{F}$. We denote this property by $\Mor(\mcal{C}) = \mcal{F}\circ \mcal{W}$.
\end{enumerate}
The maps in $\mcal{F}$ are called \emph{fibrations}, and the maps in $\mcal{W}$ are called \emph{weak equivalences}.
\end{define}

\begin{rem}
Note that we \emph{do not} require the factorizations in Definition \ref{d:weak_fib} (3) to be functorial.
\end{rem}

\begin{example}\label{e:model_is_fib}
Let $(\mcal{M},\mcal{W},\mcal{C},\mcal{F})$ be a model category. Then $(\mcal{M},\mcal{W},\mcal{F})$ is a weak fibration category.
\end{example}

\begin{example}\label{r:COFO}
Let $(\mcal{C},\mcal{W},\mcal{F})$ be a category of fibrant objects, in the sense of \cite{Bro}. If $\cC$ has finite limits, then it follows from the results of \cite[Section I.1]{Bro} that $(\mcal{C},\mcal{W},\mcal{F})$ is a weak fibration category. Note that the existence of finite limits is not satisfied in many interesting examples (consider for example the category of Kan complexes).

Conversely, let $(\mcal{C},\mcal{W},\mcal{F})$ be a weak fibration category. An object $A\in \cC$ is called \emph{fibrant} if the unique map $A\to *$ is a fibration. Let $\cC_f$ denote the full subcategory of $\cC$ spanned by the fibrant objects. Then it is not hard to check that $(\cC_f,\cW\cap\cC_f,\cF\cap\cC_f)$ is a category of fibrant objects.
\end{example}

\begin{example}\label{e:profinite}
Let $\Gamma$ be a profinite group, and let $\cC$ be the category of simplical sets with a continuous (at each degree) $\Gamma$ action. Consider $\cC$ as a weak fibration category, where the weak equivalences and the fibrations are induced from those in simplicial sets. If $\Gamma$ is finite, then $\cC$ is a model category. However, if $\Gamma$ is infinite, it is not hard to check that if $\cC$ was a model category every cofibrant object would have a free action of $\Gamma$. But this is impossible since all the stabilizers of this object must be of finite index (since the action of $\Gamma$ is continuous).
\end{example}

\begin{example}\label{e:ssh}
More generally, take $\SSh(\cC)$ to be the category of simplicial sheaves on a Grothendieck site $\cC$, where both weak equivalences and fibrations are local as in \cite{Jar} (see  Section \ref{s:SPS}). This is the main example of a weak fibration category we will consider in this paper.
\end{example}

In order to describe our main result more explicitly, we need some preliminaries from the theory of pro-categories. This is explained in more detail in Section \ref{s:prelim}. Let $\cC$ be a category. Then there is a natural fully faithful functor $\cC\to \Pro(\cC)$. By abuse of notation we will consider objects and morphisms of $\cC$ as objects and morphisms of $\Pro(\cC)$ using this functor. If $M$ is any class of morphisms in $\cC$, there is a naturally corresponding class of morphisms in $\Pro(\cC)$ called $\Lw^{\cong}(M)$. These are maps in $\Pro(\cC)$ that are isomorphic to a natural transformation which  is levelwise in $M$.

A weak fibration category $(\cC,\cW,\cF)$ is called pro-admissible if $\Lw^{\cong}(\mcal{W})\subseteq  \Pro(\mcal{C})$ satisfies the two out of three property.
The notion of a homotopically small weak fibration category is a bit more involved and will be introduced in Definition \ref{d:sub_weak}.
Intuitively, this can be thought of as a weak fibration category whose homotopical data is controlled by its small full sub weak fibration categories.
We can now state our main result, which is shown in Theorem \ref{t:model_big} and Proposition \ref{p:ho_ff};

\begin{thm}\label{t:model_intro}
Let $(\mcal{C},\mcal{W},\mcal{F})$ be a homotopically small pro-admissible weak fibration category.
Then there exists a model structure on $\Pro(\mcal{C})$ such that the weak equivalences are $\Lw^{\cong}(\mcal{W})$, the cofibrations are ${}^{\perp} (\mcal{F}\cap \mcal{W})$ and the acyclic cofibrations are
${}^{\perp} \mcal{F}.$ (Recall that ${}^{\perp}M$ denotes the class of morphisms in $\Pro(\mcal{C})$ having the left lifting property with respect to all the morphisms in $M$.)

Furthermore, the natural functor $\cC\to\Pro(\cC)$ sends weak equivalences to weak equivalences, and the induced functor $\Ho(\cC)\to\Ho(\Pro(\cC))$ is fully faithful.
\end{thm}

\begin{rem}\
\begin{enumerate}
\item The notion of a model structure referred to in Theorem \ref{t:model_intro} just means that all the usual axioms for a model category are satisfied, except, maybe, completeness cocompleteness and functoriality of the factorizations.
\item The proof of \cite[Proposition 11.1]{IsaS}, shows that $\Pro(\cC)$ is complete (since $\cC$ has finite limits). This proof also shows that (for any cardinal $\kappa$) if $\cC$ has ($\kappa$-)small colimits, so does $\Pro(\cC)$.
\item In Theorem \ref{t:model_big} we will give a more explicit description of the fibrations in this model structure, but this requires some more definitions.
\end{enumerate}
\end{rem}

The idea behind Theorem \ref{t:model_intro} is that the main reason why $(\mcal{C},\mcal{W},\mcal{F})$ is not necessarily a model category is the absence of factorizations of maps $A\to B$ in $\mcal{C}$ into a cofibration followed by an acyclic fibration $A\to C\to B$. If $\mcal{C}$ was a model category, such a factorization would be a (homotopy) initial object in the category of all factorizations of $A\to B$ into a general map followed by an acyclic fibration. If $\cC$ is only a weak fibration category, such an initial object does not necessarily exist. In this case, we take $C$ to be the entire inverse system of all such factorizations, thus resulting in a pro-object. However, the category of factorizations is not necessarily small and not necessarily cofiltered. An important part of the proof is to replace it with a related category that is small and cofiltered (see Proposition \ref{p:factor_gen_0}).

Given a \emph{model category} $\cC$, model structures on $\Pro(\cC)$ were studied by Edwards and Hastings \cite{EH}, Isaksen \cite{Isa} and other authors. Here we obtain a model structure on $\Pro(\cC)$ while assuming a weaker structure on $\cC$ itself. In the case where $\cC$ is a model category, our model structure is identical to the one described in \cite{EH},  \cite{Isa}.

In Section \ref{s:SPS} we will show that the weak fibration category of simplicial sheaves considered in Example \ref{e:ssh} is homotopically small and pro-admissible. Applying Theorem \ref{t:model_intro} to this weak fibration category, we get a novel model structure on the category $\Pro(\SSh(\cC))$ of pro-simplicial sheaves. Since every local fibration (and in particular every levelwise fibration) is a fibration in this model structure, it can be considered a projective model structure on $\Pro(\SSh(\cC))$. We elaborate more on this model structure in Section \ref{s:new_big}.

\begin{thm}
Let $\cC$ be a small Grothendieck site, and let $\SSh(\cC)$ be the category of simplicial sheaves on $\cC$.
Then there exists a model structure on $\Pro(\SSh(\cC))$ such that the weak equivalences are $\Lw^{\cong}(\mcal{W})$, the cofibrations are ${}^{\perp} (\mcal{F}\cap \mcal{W})$ and the acyclic cofibrations are
${}^{\perp} \mcal{F}$, where $\cW$ and $\cF$ are the classes of local weak equivalences and local fibrations in $\SSh(\cC)$, respectively.

Furthermore, the natural functor $\SSh(\cC)\to\Pro(\SSh(\cC))$ sends local weak equivalences to weak equivalences, and the induced functor $$\Ho(\SSh(\cC))\to\Ho(\Pro(\SSh(\cC)))$$
is fully faithful.
\end{thm}

\begin{rem}
Since $\SSh(\cC)$ is complete and cocomplete it follows that the same is true for $\Pro(\SSh(\cC))$.
\end{rem}

In \cite{Jar2}, Jardine considers a different model structure on pro-simplicial sheaves, with the same class of weak equivalences. This model structure can be thought of as ``injective" (since every levelwise cofibration is a cofibration in this model structure). In Section \ref{s:compare} we will show that the identity functors constitute a Quillen equivalence between these two model structures.

Using our new model structure on $\Pro(\SSh(\cC))$, we obtain naturally a derived functor definition of the \'etale homotopy type defined by Artin and Mazur in \cite{AM}. We use Lemma \ref{l:l_adjoint} (2) and Proposition \ref{p:RQFunc_big} to show the following:

\begin{thm}\label{t:etale}
Let $X$ be a locally Noetherian scheme, and let $X_{\acute{e}t}$ be its \'etale topos.
Let $\pi_0:X_{\acute{e}t}\to \Set$ be the functor induced by the functor which sends a scheme to its set of connected scheme-theoretic components. Then prolongation by $\pi_0$,
$$\Pro(\pi_0):\Pro(X_{\acute{e}t}^{\Del^{\op}}) \to \Pro(\Set^{\Del^{\op}}),$$
is a left Quillen functor, relative to our projective model structures.
\end{thm}

For more details see Section \ref{s:etale_big}.

\begin{rem}
In the existing injective model structure of Jardine mentioned above, $\Pro(\pi_0):\Pro(X_{\acute{e}t}^{\Del^{\op}}) \to \Pro(\Set^{\Del^{\op}})$ is \emph{not} a left Quillen functor.
\end{rem}

Theorem \ref{t:etale} enables to make the following definition:
\begin{define}\label{d:etale0}
We define the \emph{\'etale topological realization of $X$} to be
$$|X_{\acute{e}t}| := \mathbb{L}\Pro(\pi_0)(*_{X_{\acute{e}t}})\in \Ho(\Pro(\Set^{\Del^{\op}}))=\Ho(\Pro(\cS)),$$
where $*_{X_{\acute{e}t}}$ is a terminal object of $X_{\acute{e}t}^{\Del^{\op}}$.
\end{define}

The above definition of the \'etale topological realization
is closely related to Artin and Mazur's \'etale homotopy type:
\begin{thm}[see Proposition \ref{p:AM}]\label{t:etale2}
Under the natural functor
$$\Ho : \Pro(\cS) \to \Pro(\Ho(\cS)),$$
$\Ho(|X_{\acute{e}t}|)$ is isomorphic to Artin and Mazur's \'etale homotopy type.
\end{thm}

In Definition \ref{d:etale0} we have used the fact that the left Quillen functor
$$\Pro(\pi_0):\Pro(X_{\acute{e}t}^{\Del^{\op}}) \to \Pro(\cS)$$
induces a derived functor
$$\mathbb{L}\Pro(\pi_0):\Ho(\Pro(X_{\acute{e}t}^{\Del^{\op}})) \to \Ho(\Pro(\cS)).$$
Since the natural functor
$\Pro(\cS)\to \Pro(\Ho(\cS))$
sends weak equivalences to isomorphisms, we obtain an induced functor
$$\Ho(\Pro(\cS))\to \Pro(\Ho(\cS)).$$
However, the category $\Ho(\Pro(\cS))$ is very different from $\Pro(\Ho(\cS))$ and holds much more information. An object in $\Ho(\Pro(\cS))$ cannot be recovered, even up to isomorphism, from its image in $\Pro(\Ho(\cS))$. In fact the natural functor above does not even reflect isomorphisms. Thus, the \'etale topological realization $|X_{\acute{e}t}|$, which is an object in $\Ho(\Pro(\cS)),$ is more refined then Artin and Mazur's \'etale homotopy type, which is an object in $\Pro(\Ho(\cS)).$
This is very important if one wants to use the tools of abstract homotopy theory to study the scheme. In \cite{Fri}, Friedlander also lifts the \'etale homotopy type of Artin and Mazur to an actual pro-space. He does so by replacing the classical notion of hypercovering (used by Artin and Mazur to define the \'etale homotopy type) by the more involved notion of rigid hypercovering. We achieve the same goal, but without appealing to rigid hypercoverings. Moreover, our definition is more conceptual and extends naturally to a general site (see Definition \ref{d:etale_big}). Our definition is closely related to the notion of the \emph{shape} of an infinity topos, considered by Lurie (\cite[Chapter 7]{Lur}) and To\"{e}n-Vezzosi (\cite{ToVe}). (The exact relation is explained in the following  paper by the first author Y. Harpaz and G. Horel \cite{BHH}.)

The definition of topological realization extends naturally to a relative notion. Namely, given a morphism of sites $f: \cC\to \cD$, we give a derived functor definition of the topological realization of $f$, which is an object $|\cC|_\cD\in\Pro(\SSh(\cD))$ (see Definition \ref{d:relative_big}). The non-relative notion is obtained by considering the site morphism  $ \cC\to *$.
We actually get a functor
$$|\bullet|_\cD: \cS ites/\cD \to \Ho\Pro(\SSh(\cD)).$$
It is easy to verify that for every site $\cD$ we have $|\cD|_\cD \simeq *$, so by the functoriality of $|\bullet|_\cD$ we have a map
$$h:\cC(\cD)\to [*_\cD,|\cC|_\cD]_{\Pro(\SSh(\cD))},$$
where $\cC(\cD)$ is the set of site morphisms $s:\cD \to \cC$ which are sections of the map $f:\cC\to \cD$.
The codomain of $h$ above has an obstruction theory and a Bousfield-Kan type spectral sequence, so the map $h$ can be used to study sections of maps of sites. For example, if the codomain of $h$ can be shown to be empty, then we know that $\cC(\cD)$ is empty, or in other words that $f$ has no section.

A case of special interest is when $f$ is the morphism of \'etale sites induced by a scheme morphism $X\to spec(K)$. Then, a section of $f$ is just a $K$-rational point of $X$. In this case the relative topological realization lifts the notion of the relative \'etale homotopy type  $\acute{E}t_{/K}(X)$ considered in \cite{HaSc} by Harpaz and the second author, in the context of studying rational points (in a similar way that the topological realization of the \'etale site of a variety $X$ lifts the \'etale homotopy type $\acute{E}t(X)$).

The work presented in this paper originated from the motivation of finding a suitable model structure in which the general machinery of abstract homotopy theory can be used to define and study obstructions to the existence of rational points. Such obstructions were studied without the framework of a model structure by Y. Harpaz and the second author in  \cite{HaSc} and by Ambrus P\'al in \cite{Pal}.
However, in the absence of a suitable model structure, homotopical notions and constructions were given in an ad-hoc fashion.
Furthermore, having a ``topological" object and a model structure allows one to use the general machinery of model categories in order to give simpler and more conceptual proofs to the results in \cite{HaSc}. This also enables to generalize the homotopical obstruction theory of \cite{HaSc}, from fields to arbitrary base schemas.
We elaborate more on this in sections \ref{s:relative} and \ref{s:rational}.

Although the case of simplicial sheaves was our original motivation, after writing this paper additional applications of our main theorem (Theorem \ref{t:model_intro}) were found. These applications also include new methods for verifying the property of pro-admissibility, and will be discussed in future papers. For example, in \cite{BaSc1} the authors use Theorem \ref{t:model_intro} to study the accessibility rank of weak equivalences in combinatorial model categories, and in a joint work with M. Joachim and S. Mahanta, the first author uses this theorem to construct a model structure on the pro-category of the category of separable-$C^*$-algebras \cite{BJM}. Given this, and the fact that all the proofs remain exactly the same, we decided to write the paper in this more general context.
\subsection{Organization of the paper}

We begin in Section \ref{s:prelim} with a brief account of the necessary background on pro-categories. In Section \ref{s:factor} we prove a factorization lemma (Proposition \ref{p:factor_gen_big}) which will be the main tool in proving the existence of our model structure. This section is the technical heart of the paper. Section \ref{s:model} contains our main result (Theorem \ref{t:model_big}), concerning the existence of a model structure on $\Pro(\cC)$ when $\cC$ is a homotopically small pro-admissible weak fibration category. In Section \ref{s:Quillen} we define the notion of a \emph{weak right Quillen functor} between weak fibration categories, and discuss when it induces a right Quillen functor between the corresponding model categories on the pro-categories. In Section \ref{s:cofib} we discuss the homotopy category of a weak fibration category. We show that the natural inclusion $\cC \to \Pro(\cC)$ induces a fully faithful functor $\Ho(\cC) \to \Ho(\Pro(\cC))$, when $\cC$ is homotopically small and pro-admissible.
In Section \ref{s:SPS} we consider our main examples, namely, the categories of simplicial sheaves and simplicial presheaves on a Grothendieck site. We show that with the notions of local weak equivalences and local fibrations, they both become homotopically small pro-admissible weak fibration categories. Using our main theorem we deduce the existence of induced model structures on their pro-categories. In Section \ref{s:etale_big} we apply the results of the previous two sections to give a derived functor definition of the \'etale homotopy type of \cite{AM}. We also generalize this to the topological realization of a general site and a morphism of sites, as explained above.

\subsection{Acknowledgments}
We would like to thank Yonatan Harpaz, Ambrus P\'{a}l, Bertrand To\"en  and Mark Behrens for some useful discussions.
We would like to thank our PhD advisors Emmanuel D. Farjoun, David Kazhdan, and Ehud De-Shalit, for their help and useful suggestions. The second author would also like to thank Pierre Deligne for a stimulating and useful conversation concerning the topics of this paper and their applications. Finally, we would like to thank the referee for his useful remarks.
\section{Preliminaries on Pro-Categories}\label{s:prelim}

In this section we bring a short review of the necessary background on pro-categories. Some of the definitions and lemmas given here are slightly non-standard.
Standard references on pro-categories include \cite{AM} and \cite{SGA4-I}. For the homotopical parts the reader is referred to \cite{BaSc}, \cite{EH} and \cite{Isa}.
Many of the ideas in this section (and paper) are influenced by Isaksen's work on pro-categories.

\begin{define}\label{d:cofiltered}
A category $\cI$ is called \emph{cofiltered} if the following conditions are satisfied:
\begin{enumerate}
\item The category $\cI$ is non-empty.
\item For every pair of objects $s,t \in \cI$, there exists an object $u\in \cI$, together with
morphisms $u\to s$ and $u\to t$.
\item For every pair of morphisms $f,g:s\to t$ in $\cI$ there exists a morphism $h:u\to s$ in $\cI$ such that $f\circ h=g\circ h$.
\end{enumerate}
A category satisfying only the first two properties listed above is called \emph{semi-cofiltered}.
\end{define}

If $\cT$ is a poset, then we view $\cT$ as a category which has a single morphism $u\to v$ iff $u\geq v$. Note that this convention is the opposite of that used by some authors. Thus, a poset $\cT$ is cofiltered iff $\cT$ is non-empty, and for every $a,b$ in $\cT$ there exists an element $c$ in $\cT$ such that $c\geq a,b$. A cofiltered poset will also be called \emph{directed}. Additionally, in the following, instead of saying ``a directed poset" we will just say ``a directed set".

\begin{define}\label{def CDS}
A cofinite poset is a poset $\cT$ such that for every element $x$ in $\cT$ the set $\cT_x:=\{z\in \cT| z \leq x\}$ is finite.
\end{define}

\begin{define}\label{def_deg}
Let $\cA$ be a cofinite poset. We define the degree function of $\cA$, $d=d_\cA:\cA\to\NN$, by
$$d(a):=max\{n\in \NN|\exists a_0<\cdots<a_n=a\}.$$
For every $n\geq -1$ we define $\cA^n:=\{a\in \cA|d(a)\leq n\}$ $(\cA^{-1}=\phi)$.
\end{define}
Thus $d:\cA\to\NN$ is a strictly increasing function. The degree function enables us to define or prove things concerning $\cA$ inductively, since clearly $\cA=\cup_{n\geq 0}\cA^n$.

\begin{define}\label{d:section}
Let $\cT$ be a poset, and let $\cA$ be a subset of $\cT$. We shall say that $\cA$ is a (lower) section of $\cT$, if for every $x$ in $\cA$ and $y$ in $\cT$ such that $y <x$, we have that $y$ is also in $\cA$.
\end{define}

\begin{define}\label{d:triangle}
Let $\cC$ be a category.
\begin{enumerate}
\item The category $\cC^{\lhd}$ has as objects $\Ob(\cC)\coprod{\infty}$, and the morphisms are the morphisms in $\cC$, together with a unique morphism $\infty\to c$, for every object $c$ in $\cC$.
\item The category $\cC^{\rhd}$ has as objects $\Ob(\cC)\coprod (-\infty)$, and the morphisms are the morphisms in $\cC$, together with a unique morphism $c\to (-\infty)$, for every object $c$ in $\cC$.
\end{enumerate}
\end{define}

In particular, if $\cC=\phi$ then $\cC^{\lhd}=\{\infty\}$.
Note that if $\cA$ is a cofinite poset and $a$ is an element in $\cA$ of degree $n$, then $\cA_a$ is naturally isomorphic to $(\cA_a^{n-1})^{\lhd}$ (where $\cA_a^{n-1}$ is just $(\cA_a)^{n-1}$, see Definition \ref{def CDS}).

The following lemma is clear, but we include it for later
reference:
\begin{lem}\label{l:eqiv_cofiltered}\
\begin{enumerate}
\item A cofinite poset $\cA$ is cofiltered iff for every finite section $R$ of $\cA$ (see Definition \ref{d:section}), there exists an element $c$ in $\cA$ such that $c\geq r$ for every $r$ in $R$.
\item A category $\cC$ is cofiltered iff for every finite poset $R$ and every functor $F:R\to \cC$, there exists an object $c$ in $\cC$, together with compatible morphisms $c\to F(r)$ for every $r$ in $R$ (that is, a morphism $Diag(c)\to F$ in $\cC^R$, or equivalently, we can extend the functor $F:R\to \cC$ to a functor $R^{\lhd}\to \cC$).
\end{enumerate}
\end{lem}

A category is called \emph{small} if it has a small set of objects and a small set of morphisms

\begin{define}\label{def_pro}
Let $\mcal{C}$ be a category. The category $\Pro(\mcal{C})$ has as objects all diagrams in $\cC$ of the form $\cI\to \cC$ such that $\cI$ is small and cofiltered (see Definition \ref{d:cofiltered}). The morphisms are defined by the formula
$$\Hom_{\Pro(\mcal{C})}(X,Y):=\lim_s \colim_t \Hom_{\mcal{C}}(X_t,Y_s).$$
Composition of morphisms is defined in the obvious way.
\end{define}

Thus, if $X:\cI\to \mcal{C}$ and $Y:\cJ\to \mcal{C}$ are objects in $\Pro(\mcal{C})$, providing a morphism $X\to Y$ means specifying for every $s$ in $\cJ$ an object $t$ in $\cI$ and a morphism $X_t\to Y_s$ in $\mcal{C}$. These morphisms should of course satisfy some compatibility condition. In particular, if $p:\cJ\to \cI$ is a functor, and $\phi:p^*X:=X\circ p\to Y$ is a natural transformation, then the pair $(p,\phi)$ determines a morphism $\nu_{p,\phi}:X\to Y$ in $\Pro(\cC)$ (for every $s$ in $\cJ$ we take the morphism $\phi_s:X_{p(s)}\to Y_s$). In particular, taking $Y=p^*X$ and $\phi$ to be the identity natural transformation, we see that $p$ determines a morphism $\nu_{p,X}:X\to p^*X$ in $\Pro(\cC)$.

Let $f:X\to Y$ be a morphism in $\Pro(\cC)$. A morphism in $\cC$ of the form $X_r\to Y_s$ that represents the $s$ coordinate of $f$ in $\colim_{t\in \cI} \Hom_{\mcal{C}}(X_t,Y_s)$ is called ``representing $f$".

The word pro-object refers to objects of pro-categories. A \emph{simple} pro-object
is one indexed by the category with one object and one (identity) map. Note that for any category $\mcal{C}$, $\Pro(\mcal{C})$ contains $\mcal{C}$ as the full subcategory spanned by the simple objects.

\begin{define}\label{d:cofinal}
Let $p:\cJ\to \cI$ be a functor between categories. The functor $p$ is said to be (left) cofinal if for every object $i$ in $\cI$ the following conditions are satisfied:
\begin{enumerate}
\item The over category ${p}_{/i}$ is nonempty.
\item The over category ${p}_{/i}$ is connected.
\end{enumerate}
A functor satisfying only the first property above is called \emph{semi-cofinal}.
\end{define}

Cofinal functors play an important role in the theory of pro-categories mainly because of the following well-known lemma:

\begin{lem}\label{l:cofinal}
Let $p:\cJ\to \cI$ be a cofinal functor between small cofiltered categories, and let $X:\cI\to \cC$ be an object in $\Pro(\cC)$. Then $\nu_{p,X}:X\to p^*X$ is an isomorphism in $\Pro(\cC)$.
\end{lem}

We denote by $[1]$ the category with object set $\{0,1\}$ and one non-identity morphism $0\to 1$. Thus, if $\cC$ is any category, the functor category $\cC^{[1]}$ is just the category of morphisms in $\cC$.

\begin{lem}[{\cite[Corollary 3.26]{BaSc}}]\label{every map natural}
One can construct a natural inverse equivalences of categories
$$\Pro(\cC^{[1]}) \rightleftarrows \Pro(\cC)^{[1]},$$
where the left adjoint is $\phi\mapsto\nu_{\id,\phi}$.

In particular, for every morphism in $\Pro(\cC)$ one can choose an isomorphic morphism (in the category of morphisms in $\Pro(\cC)$) that comes from a natural transformation and this choice can be done functorially.
\end{lem}

\begin{define}\label{def natural}
Let $\mcal{C}$ be a category with finite limits, $M$ a class of morphisms in $\mcal{C}$, $\cI$ a small category and $F:X\to Y$ a morphism in $\mcal{C}^\cI$. Then:
\begin{enumerate}
\item The map  $F$ is called a \emph{levelwise} $M$-\emph{map}, if for every $i$ in $\cI$ the morphism $X_i\to Y_i$ is in $M$. We will denote this by $F\in \Lw(M)$.
\item The map  $F$ is called a \emph{special} $M$-\emph{map}, if the following hold:
    \begin{enumerate}
    \item The indexing category $\cI$ is a cofinite poset (see Definition \ref{def CDS}).
    \item The natural map $X_t \to Y_t \times_{\lim_{s<t} Y_s} \lim_{s<t} X_s $ is in $M$, for every $t$ in $ \cI$.
    \end{enumerate}
    We will denote this by $F\in \Sp(M)$.
\item The diagram $X$ is called a \emph{special} $\mcal{M}$-\emph{diagram}, if the natural transformation $X\to *$ is a special $\mcal{M}$-map.
\end{enumerate}
\end{define}

\begin{define}\label{def mor}
Let $\mcal{C}$ be a category and $M,N$ classes of morphisms in $\mcal{C}$.
\begin{enumerate}
\item If $f$ and $g$ are morphisms in $\cC$, we denote by $f \perp g$ the fact that $f$ has the left lifting property with respect to $g$. We denote by $M \perp N$ the fact that $f\perp g$ for every $f$ in $M$ and $g$ in $N$.
\item We denote by $\R(M)$ the class of morphisms in $\mcal{C}$ that are retracts of morphisms in $M$. Note that $\R(\R(M))=\R(M)$.
\item We denote by $M^{\perp}$ (resp. ${}^{\perp}M$) the class of morphisms in $\mcal{C}$ having the right (resp. left) lifting property with respect to all the morphisms in $M$.
\item We denote by $\Lw^{\cong}(M)$ the class of morphisms in $\Pro(\mcal{C})$ that are \textbf{isomorphic} to a morphism that comes from a natural transformation which is a levelwise $M$-map.
\item If $\cC$ has finite limits, we denote by $\Sp^{\cong}(M)$ the class of morphisms in $\Pro(\mcal{C})$ that are \textbf{isomorphic} to a morphism that comes from a natural transformation which is a special $M$-map.
\end{enumerate}
\end{define}

\begin{lem}\label{l_lift}
Assume $\Mor(\cC)=M\circ N$. Then
$$N^{\perp}\subseteq \R(M),\:\:^{\perp}M\subseteq \R(N).$$
\end{lem}

\begin{lem}[{\cite[Proposition 2.2]{Isa}}]\label{l:ret_lw}
Let $M$ be any class of morphisms in $\mcal{C}$. Then $$\R(\Lw^{\cong}(M)) = \Lw^{\cong}(M).$$
\end{lem}

\begin{lem}\label{c:ret_lift}
Let $M$ be any class of morphisms in $\mcal{C}$. Then
$$(\R(M))^{\perp} = M^{\perp},\:\: {}^{\perp}(\R(M)) = {}^{\perp}M,$$
$$\R(M^{\perp}) = M^{\perp},\:\: \R({}^{\perp}M) = {}^{\perp}M.$$
\end{lem}

\begin{lem}[{\cite[Lemma 2.14]{BaSc}}]\label{l:SpMo_is_Mo}
Let $M$ be any class of morphisms in $\mcal{C}$. Then $${}^{\perp}\Sp^{\cong}(M) = {}^{\perp}M.$$
\end{lem}

\begin{prop}\label{p:section}
Let $\mcal{C}$ be a category with finite limits, and $\mcal{M} \subseteq \mcal{C}$ a subcategory that is closed under base change and contains all the isomorphisms. Let $\cT$ be a cofinite poset, $X:\cT \to \mcal{C}$ a special $\mcal{M}$-diagram, and $\cA  \subseteq \cB \subseteq \cT$ any two finite sections of $\cT$ (see Definition \ref{d:section}).
Then the map
$$\lim\limits_{s\in \cB} X_t \to  \lim\limits_{s\in \cA} X_t $$
is in $\mcal{M}$.
\end{prop}
\begin{proof}
We prove the lemma by induction on the size of $\cB$. The base of the induction ($\cB=\phi$) is clear.
Now assume that the lemma holds for $|\cB| <n$ ($n\geq 1$).
Let us prove the lemma for $|\cB|=n$. If $\cA=\cB$ the statement is clear. Otherwise, choose a maximal element
$x\in \cB \backslash \cA$. We can decompose the map
$$\lim\limits_{s\in \cB} X_s \to  \lim\limits_{s\in \cA} X_s $$
into
$$\lim\limits_{s\in \cB} X_s \to \lim\limits_{s\in \cB\backslash \{x\}} X_s \to  \lim\limits_{s\in \cA} X_s .$$
The first map belongs to $\mcal{M}$ by considering the pullback square
$$\xymatrix{
\lim\limits_{s\in \cB} X_s \ar[r]\ar[d] & \lim\limits_{s\in \cB \backslash \{x\}} X_s\ar[d] \\
X_t \ar[r]^{\mcal{M}} & \lim\limits_{s < t} X_s,
}$$
and the second map belongs to $\mcal{M}$ by the induction hypothesis. Since $\mcal{M}$ is closed under composition, we have the desired result.
\end{proof}

\begin{cor}\label{c:M_object}
Let $\mcal{C}$ be a category with finite limits, and $\mcal{M} \subseteq \mcal{C}$ a subcategory that is closed under base change, and contains all the isomorphisms. Let $\cT$ be a cofinite poset, and $X:\cT \to \mcal{C}$ a special $\mcal{M}$-diagram.
Then for every $t\in \cT$, $X_t$ is an $\mcal{M}$-object (that is, the morphism $X\to *$ is in $\mcal{M}$).
\end{cor}
\begin{proof}
Apply Proposition \ref{p:section} with $\cB=\cT_{\leq t}, \cA = \emptyset.$
\end{proof}

\begin{prop}\label{forF_sp_is_lw}
Let $\mcal{C}$ be a category with finite limits, and $\mcal{M} \subseteq \mcal{C}$ a subcategory that is closed under base change, and contains all the isomorphisms. Let $F:X\to Y$ be a natural transformation between diagrams in $\mcal{C}$, which is a special $\mcal{M}$-map. Then $F$ is a levelwise $\mcal{M}$-map.
\end{prop}

Proposition \ref{forF_sp_is_lw} appears in \cite{FaIs} (see Lemmas 2.3 and 5.14), but without a full proof. We thus explain briefly how it follows easily from Corollary \ref{c:M_object} by a simple trick.

\begin{proof}
Let $Ar(\mcal{C})$ denote the category of arrows in $\mcal{C}$.
Define $\mcal{M}^{ar}\subseteq \Mor(Ar(\mcal{C}))$ to be the class of morphisms represented by squares
$$\xymatrix{
A \ar[r]\ar[d]& B\ar[d] \\
C\ar[r] & D,
}$$
such that the natural map $A\to B\times_D C$ is in $\mcal{M}$.
It is a standard verification that $\mcal{M}^{ar}\subseteq Ar(\mcal{C})$ is a subcategory that is closed under base change and contains all the isomorphisms.
It is also easy to see that an object $A\to B \in Ar(\mcal{C})$ is an $\mcal{M}^{ar}$-object iff it is a morphism in $\mcal{M}$.

Let $\cT$ be a cofinite poset, and let $F:X\to Y$ be a morphism in $\mcal{C}^\cT$, which is a special $\mcal{M}$-map. Let $\cI$ denote the category with two objects $0,1$, and a unique morphism $0\to 1$. Then $F$ can be regarded as a functor $F:\cI\to \mcal{C}^\cT$, or equivalently, as a functor $F:\cT\to \mcal{C}^\cI=Ar(\cC)$. It is straightforward to check that $F$ (in the first picture) is a special $\mcal{M}$-map iff $F$ (in the second picture) is a special $\mcal{M}^{ar}$-diagram. It follows from Corollary \ref{c:M_object} that for every $t\in \cT$, $F_t\in Ar(\cC)$ is an $\mcal{M}^{ar}$-object. It thus follows that for every $t\in \cT$, $F_t:X_t\to Y_t$ is in $\mcal{M}$.
\end{proof}

\begin{cor}\label{l:forF_sp_is_lw}
Let $\mcal{C}$ be a category with finite limits, and $\mcal{M} \subseteq \mcal{C}$ a subcategory that is closed under base change, and contains all the isomorphisms. Then $\Sp^{\cong}(\mcal{M})\subseteq \Lw^{\cong}(\mcal{M})$.
\end{cor}

\section{Factorization of Maps}\label{s:factor}
In this section we prove a proposition about factorization of maps (Proposition \ref{p:factor_gen_0}) which will be our main tool in proving later the existence of the desired model structure.

We begin with some definitions and lemmas that we will need for proving this proposition and in later sections.

\begin{define}\label{d:pre_cofinal_big}
Let $F:\cJ\to \cI$ be a functor, and let $\barr{\cI}$ be  a category, that contains $\cI$ as a subcategory.
Then $F$ is called \emph{pre-cofinal} relative to $\barr{\cI}$, if for every morphism in $\cI$ of the form $f:i\to F(j)$, there exists a morphism $g:j'\to j$ in $\cJ$ such that $F(g)$ factors through $f$ in $\barr{\cI}$:
$$\xymatrix{
F(j') \ar[dr]_{\barr{\cI}}\ar[rr]^{F(g)}  & \empty & F(j)  \\
\empty &  i \ar[ur]_f & \empty .}$$
$F$ is called simply pre-cofinal, if it is pre-cofinal relative to $\cI$.
\end{define}

\begin{lem}\label{l:compose_onto_big}
The composition of pre-cofinal functors is pre-cofinal.

More generally, let $G:\cK\to \cJ$ be a pre-cofinal functor relative to $\barr{\cJ}$, and let $F:\barr{\cJ}\to\barr{\cI}$ be a functor such that $F|_\cJ:\cJ\to \cI$ is pre-cofinal relative to $\barr{\cI}$.
Then $F\circ G:\cK\to \cI$ is pre-cofinal relative to $\barr{\cI}$.
\end{lem}

\begin{proof}
Let $f:i\to F(G(k))$ be a morphism in $\cI$. The restriction $F|_\cJ$ is pre-cofinal relative to $\barr{\cI}$, so there exists a morphism $g:j'\to G(k)$ in $\cJ$ such that $f\circ t=F(g)$ ($t\in \barr{\cI}$). The functor $G$ is pre-cofinal  relative to $\barr{\cJ}$, so there exists a morphism $h:k'\to k$ in $\cK$ such that $g\circ l=G(h)$ ($l\in \barr{\cJ}$). It follows that
$$f\circ t\circ F(l)=F(g)\circ F(l)=F(g\circ l)=F(G(h)).$$
\end{proof}

The following lemma connects the notions of pre-cofinal and cofinal functors:

\begin{lem}\label{l:onto is cofiltered}
Let $F:\cJ\to \cI$ be a pre-cofinal functor relative to $\barr{\cI}$.
\begin{enumerate}
\item
Suppose that $\Ob(\cI)=\Ob(\barr{\cI})$, $\cI$ is semi-cofiltered (see Definition \ref{d:cofiltered}) and $\cJ\neq\phi$.
Then $F:\cJ\to\barr{\cI}$ is semi-cofinal (see Definition \ref{d:cofinal}).
\item
If in addition $\cI=\barr{\cI}$, $\cI$ is cofiltered and $\cJ$ is semi-cofiltered (see Definition \ref{d:cofiltered}), then $F:\cJ\to\barr{\cI}$ is cofinal.
\end{enumerate}
\end{lem}

\begin{proof}
We need to show that for every $i\in \cI$, the over-category $F_{/i}$ is nonempty. Let $i\in \cI$. The category
$\cJ$ is nonempty, so we can choose $j\in \cJ$. We have that $F(j),i\in \cI$ and $\cI$ is semi-cofiltered, so there exist morphisms in $\cI$ of the form $f:i'\to i, g:i'\to F(j)$. The functor $F$ is pre-cofinal relative to $\barr{\cI}$, so there exists $h:j'\to j$ in $\cJ$ such that $F(h)$ factors through $g$ in $\barr{\cI}$:
$$\xymatrix{
F(j') \ar[dr]^k_{\barr{\cI}}\ar[rr]^{F(h)}  & \empty & F(j).  \\
\empty &  i' \ar[ur]_g & \empty }$$
In particular, $f\circ k:F(j')\to i$ is an object in $F_{/i}$.

Now suppose that $\cI=\barr{\cI}$, $\cI$ is cofiltered and $\cJ$ is semi-cofiltered. We need to show that $F_{/i}$ is also connected. Let $f_1:F(j_1)\to i,f_2:F(j_2)\to i$ be two objects in $F_{/i}$. Since $\cJ$ is semi-cofiltered, there exist morphisms in $\cJ$ of the form $g_1:j_3\to j_1,g_2:j_3\to j_2$. Then $f_1 F(g_1),f_2 F(g_2):F(j_3)\to i$ are two parallel morphisms in $\cI$. Since $\cI$ is cofiltered, there exists a morphism $h:i'\to F(j_3)$ in $\cI$ such that $f_1 F(g_1) h=f_2 F(g_2) h$. Since $F$ is pre-cofinal, there exists a morphism $k:j_4\to j_3$ in $\cJ$ such that $F(k)=h l$. It follows that
$$f_1 F(g_1 k)=f_1 F(g_1) F(k)=f_1 F(g_1) h l=f_2 F(g_2) h l=f_2 F(g_2) F(k)=f_2 F(g_2 k).$$
Thus, we have morphisms in $F_{/i}$ as follows:
$$\xymatrix{
F(j_1) \ar[dr]_{f_1}  & F(j_4) \ar[l]_{F(g_1 k)} \ar[r]^{F(g_2 k)} \ar[d] & F(j_2)  \ar[dl]^{f_2}.\\
\empty &  i & \empty }$$
\end{proof}

\begin{define}
A factorization category is a triple $(\cC,N,\cM)$ such that:
\begin{enumerate}
\item $\mcal{C}$ is a category that has finite limits.
\item $\mcal{M} \subseteq \mcal{C}$ is a subcategory that is closed under base change
\item $N \subseteq \Mor(\mcal{C})$ is an arbitrary class of morphisms such that $\mcal{M}\circ N =\Mor(\mcal{C})$.
\end{enumerate}
\end{define}

\begin{define}\label{d:factorization_cat}
Let $(\cC,N,\cM)$ be a factorization category and let
$f:X\to Y$ be a morphism in $\cC^\cT$, where $\cT$ is a small cofiltered category.

We define the category $\barr{\cF_f}$, whose objects are all pairs $(t, X_t \xrightarrow{g} H \xrightarrow{h} Y_t)$ such that $t\in \cT$, $h\circ g = f_t$, $g\in N$ and $h\in \mcal{M}$. A morphism
$$(t, X_t \to H \to Y_t)\to (t', X_{t'} \to H' \to Y_{t'})$$
in $\barr{\cF_f}$ is given by a morphism $t\to t'$ in $\cT$ together with a commutative diagram of the form
$$\xymatrix{
X_t \ar[d] \ar[r]^N & H \ar[d] \ar[r]^{\mcal{M}} & Y_t \ar[d]\\
X_{t'} \ar[r]^N &  H' \ar[r]^{\mcal{M}} & Y_{t'}    ,}$$
such that the left and right vertical maps are induced by the given morphism $t\to t'$.

We also consider the subcategory $\cF_f\subseteq\barr{\cF_f}$ containing all the objects, but containing only morphisms as above such that the induced map $H\to H'\times_{Y_{t'}}Y_t$ is in $\cM$.

There are natural functors $\barr{p}:\barr{\cF_f}\to \cT$ and $p:{\cF_f}\to \cT$. We define the functors $\barr{H_f}:\barr{\cF_f}\to \cC$ and $H_f:{\cF_f}\to \cC$ to be those sending $(t,X_t\to H \to Y_t)$ to $H$. There are obvious factorizations
$$\barr{p}^*X\xrightarrow{} \barr{H_f}\xrightarrow{} \barr{p}^*Y,$$
$$p^*X\xrightarrow{} H_f\xrightarrow{} p^*Y$$
of $p^*f:p^*X\to p^*Y$ and $\barr{p}^*f:\barr{p}^*X\to \barr{p}^*Y$, in the functor categories $\mcal{C}^{{\cF_f}}$ and $\mcal{C}^{\barr{\cF_f}}$ respectively.
\end{define}

\begin{rem}
The categories $\cF_f,\barr{\cF_f}$ are semi-cofiltered (see Definition \ref{d:cofiltered}), that is, for every pair of objects there is an object that dominates both. However, not every pair of parallel morphisms can be equalized, so they are not necessarily cofiltered.
\end{rem}

These categories are universal among categories that produce certain types of factorizations after pullbacks, as expressed by the following lemma:
\begin{lem}\label{l:factor}
Let $(\cC,N,\cM)$ be a factorization category and let
$f:X\to Y$ be a morphism in $\cC^\cT$, where $\cT$ is a small cofiltered category.
\begin{enumerate}
\item Let $\cA$ be any category, and let $r:\cA\to \cT$ be a functor. Then a factorization $r^*X\xrightarrow{g} H \xrightarrow{h} r^*Y$ of  $r^*f:r^*X\to r^*Y$, in the category $\cC^{\cA}$, such that $h$ is a levelwise $\mcal{M}$ map and $g$ is a levelwise $N$ map, gives rise in a natural way to a functor $q:\cA \to \barr{\cF_f}$ such that $r=\barr{p}q$.

\item Let $\cA$ is a cofinite poset, and let $r:\cA\to \cT$ be a functor. Then a factorization $r^*X\xrightarrow{g} H \xrightarrow{h} r^*Y$ of  $r^*f:r^*X\to r^*Y$, in the category $\cC^{\cA}$, such that $h$ is a special $\mcal{M}$ map and $g$ is a levelwise $N$ map, gives rise in a natural way to a functor $q:\cA \to \cF_f$ such that $r=pq$.
\end{enumerate}
\end{lem}

\begin{proof}
(1) is trivial so let us show (2). By (1) we have an induced functor $q:\cA \to \barr{\cF_f}$ such that $r={p}q$. It remains to show that $q$ is actually a functor to $\cF_f$.

Let $a<b$ be elements of $\cA$. Then the morphism $q(b)\to q(a)$ in $\barr{\cF_f}$ is given by the commutative diagram
$$\xymatrix{
X_{r(b)} \ar[d] \ar[r]^{g_b} & H(b) \ar[d] \ar[r]^{h_b} & Y_{r(b)} \ar[d]\\
X_{r(a)} \ar[r]^{g_a} &  H(a) \ar[r]^{h_a} & Y_{r(a)}    .}$$
We need to show that the induced map $H(b)\to H(a)\times_{ Y_{r(a)}} Y_{r(b)}$ is in $\cM$.
Since the map $h:H \to r^*Y$ in $\cC^{\cA}$ is a special $\cM$ map, it follows from Proposition \ref{p:section}, when applied to the sections $R_{b}$ and $R_{a}$ of $\cA$, that the map
$$\lim\limits_{R_{b}} H\ \to \lim\limits_{R_{a}} H\times_{\lim\limits_{R_{a}} Y\circ r}\lim\limits_{R_{b}} Y\circ r$$
is in $\cM$, or in other words, that the map $H(b) \to H(a)\times_{Y_{r(a)}}Y_{ r(b)}$ is in $\cM$.
\end{proof}

\begin{lem}\label{l:lifting_fact}
Let $(\cC,N,\cM)$ be a factorization category.
Suppose we have a commutative diagram in $\mcal{C}$ of the form
$$\xymatrix{
X \ar[d] \ar[r] & C \ar[d]^{\mcal{M}} \\
Y \ar[r] &  D          .}$$
Then we can embed this diagram in a bigger commutative diagram of the form
$$\xymatrix{X\ar[rr]\ar[dd]\ar[dr]^N & & C\ar[dd]^\cM\\
              & Y' \ar[dl]_\cM\ar[ur]   & \\
Y  \ar[rr]  & &  D ,  \\}$$
such that the induced map $Y'\to Y\times_D C$ is in $\cM$.
\end{lem}

\begin{rem}
Notice the resemblance of Lemma \ref{l:lifting_fact}, to \cite{Bro} I 2, Lemma 1.
\end{rem}

\begin{proof}
Consider the diagram
$$\xymatrix{
X       \ar[d] \ar[r] & C \ar@{=}[d] \\
Y\times_D C \ar[r]\ar[d]^{\mcal{M}}  &   C \ar[d]^{\mcal{M}} \\
Y          \ar[r] &  D .}$$
Since $\mcal{M}\circ N = \Mor(\mcal{C})$, we can factor the map $X \to Y\times_D C$ and obtain
$$\xymatrix{
X       \ar[d]^{N} \ar[r] & C \ar@{=}[d] \\
Y'       \ar[d]^{\mcal{M}} \ar[r] & C \ar@{=}[d] \\
Y\times_D C \ar[r] \ar[d]^{\mcal{M}}       &  C \ar[d]^{\mcal{M}}  \\
Y          \ar[r] &  D  ,}$$
and since $\mcal{M}$ is closed under composition, we get the desired result.
\end{proof}

\begin{lem}\label{l:extend_factorization}
Let $(\cC,N,\cM)$ be a factorization category.
Let $R$ be a finite poset, and let $f:X\to Y$ be a map in $\mcal{C}^{R^{\lhd}}$. Let $X|_R \xrightarrow{g} H \xrightarrow{h} Y|_R$ be a factorization of $f|_R$ such that $g$ is levelwise $N$ and $h$ is special $\cM$.
Then all the factorizations of   $f$  of the form $X \xrightarrow{g'} H' \xrightarrow{h'} Y$, such that $g'$ is levelwise $N$, $h'$ is special $\mcal{M}$ and $H'|_R=H,g'|_R = g,h'|_R = h$, are in natural 1-1 correspondence with all factorizations of the map   $X(\infty) \to \lim_R H \times_{\lim_R Y} Y(\infty)$
of the form $X(\infty) \xrightarrow{g''} H'(\infty) \xrightarrow{h''} \lim_R H \times_{\lim_R Y} Y(\infty)$, such that $g''\in N$ and $h'' \in \mcal{M}$ (in particular there always exists one, since $\cM\circ N=\Mor(\cC)$).
\end{lem}

\begin{proof}
This is quite straightforward. For more details see \cite[Lemma 4.2]{BaSc}.
\end{proof}

\begin{lem}\label{l:F_T_onto}
Let $(\cC,N,\cM)$ be a factorization category and let
$f:X\to Y$ be a morphism in $\cC^\cT$, where $\cT$ is a small cofiltered category.
Then the functor $p:\cF_f\to \cT$ is pre-cofinal.
\end{lem}

\begin{proof}
Let $(t, X_t \xrightarrow{g} H \xrightarrow{h} Y_t)$ be an object of $\cF_f$, and let $t'\to t$ be a morphism in $\cT$. It is enough to show that there exists a morphism in $\cF_f$ of the form
$$\xymatrix{
X_{t'} \ar[d] \ar[r]^N & H' \ar[d] \ar[r]^{\mcal{M}} & Y_{t'}  \ar[d]\\
X_t \ar[r]^N &  H \ar[r]^{\mcal{M}} &  Y_t  }$$
such that the left and right vertical maps are induced by the given morphism $t'\to t$.

We have a commutative diagram of the form
$$\xymatrix{
X_{t'} \ar[d] \ar[rr]^f & \empty & Y_{t'}  \ar[d]\\
X_t \ar[r]^N &  H \ar[r]^{\mcal{M}} &  Y_t  .}$$
Thus, we can apply Lemma \ref{l:lifting_fact} with $X:=X_{t'}, Y:=Y_{t'}, C:=H, D:=Y_t$, and get the desired result.
\end{proof}

The following proposition is our main motivation for introducing the concept of a pre-cofinal functor:

\begin{prop}\label{p:lw_factor_gen}
Let $(\cC,N,\cM)$ be a factorization category and let
$f:X\to Y$ be a morphism in $\cC^\cT$, where $\cT$ is a small cofiltered category.

Let $\cI$ be a small cofiltered category, and let $q:\cI\to \cF_f$ be a pre-cofinal functor, relative to $\barr{\cF_f}$. Then we have a cofinal functor $pq:\cI\to \cT$, and an induced factorization
$$(pq)^*X\xrightarrow{g_q} q^*H_f \xrightarrow{h_q} (pq)^*Y$$
of $(pq)^*f:(pq)^*X\to (pq)^*Y$, in $\cC^{\cI}$, such that $h_q$ is a levelwise $\mcal{M}$ map, and $g_q$ is a levelwise $N$ map that belongs to ${}^{\perp}\mcal{M}$ as a map in $\Pro(\cC)$.
(See Definition \ref{d:factorization_cat} for the notations $\cF_f$, $\barr{\cF_f}$, $p$ and $H_f$.)

In particular, it follows that $X\cong (pq)^*X$ and $Y\cong (pq)^*Y$ in $\Pro(\cC)$, and we get a factorization
$$X\xrightarrow{g} H_f^\cI \xrightarrow{h} Y$$
of $f:X\to Y$ in $\Pro(\mcal{C})$, such that $h$ is in $\Lw^{\cong}(\mcal{M})$ and $g$ is in $\Lw^{\cong}(N)\cap  {}^{\perp}\mcal{M}$.
\end{prop}

\begin{proof}
Clearly we have a factorization
$$(pq)^*X\xrightarrow{g_q} q^*H_f \xrightarrow{h_q} (pq)^*Y$$
of $(pq)^*f:(pq)^*X\to (pq)^*Y$, in $\cC^{\cI}$, such that $h_q$ is a levelwise $\mcal{M}$ map and $g_q$ is a levelwise $N$ map.

The functor $q:\cI\to \cF_f$ is pre-cofinal, relative to $\barr{\cF_f}$, and the functor $p:\cF_f\to \cT$ is pre-cofinal by Lemma \ref{l:F_T_onto}. From Lemma \ref{l:compose_onto_big} we get that the functor $pq:\cI\to \cT$ is pre-cofinal, so by Lemma  \ref{l:onto is cofiltered} it is also cofinal.

It thus remains to show that $g_q\in {}^{\perp}\mcal{M}$, as a map in $\Pro(\cC)$.
Consider the following diagram:
$$
\xymatrix{
\{X_{p(q(a))}\}_{a\in \cI} \ar[d]^{g_q} \ar[r] & C \ar[d]^{\mcal{M}} \\
\{H_{q(a)}\}_{a\in \cI}          \ar[r] & D.
}$$
We need to show the existence of a lift in the above square. It follows from the definition of morphisms in $\Pro(\cC)$ that there exists some $a_0 \in \cI$, such that the above square factors as
$$
\xymatrix{
\{X_{p(q(a))}\}_{a\in \cI} \ar[d]^{g_q} \ar[r] & X_{p(q(a_0))} \ar[d]^{g_{q(a_0)}} \ar[r] & C \ar[d]^{\mcal{M}} \\
\{H_{q(a)}\}_{a\in \cI}          \ar[r] & H_{q(a_0)}        \ar[r] & D.}
$$
In order to finish the proof, it is enough to find a morphism $a'_0\to a_0$ in $\cI$, such that in the following diagram we can add a dotted line:
$$\xymatrix{
X_{p(q(a'_0))} \ar[d]^{g_{q(a'_0)}} \ar[d] \ar[r] & X_{p(q(a_0))} \ar[d]^{g_{q(a_0)}} \ar[r] & C \ar[d]^{\mcal{M}} \\
H_{q(a'_0)}  \ar@{..>}[urr]          \ar[r] & H_{q(a_0)}          \ar[r] & D
.}
$$
By Lemma \ref{l:lifting_fact} we have a commutative diagram in $\mcal{C}$ of the form
$$\xymatrix{X_{p(q(a_0))}\ar[rr]\ar[dd]\ar[dr]^N & & C\ar[dd]^\cM\\
              & Z \ar[dl]_\cM\ar[ur]   & \\
 H_{q(a_0)}  \ar[rr]  & &  D .  \\}$$

We thus have a morphism in $\cF_f$ of the form
$$\xymatrix{
X_{p(q(a_0))}\ar[d]^= \ar[r]^N &  Z        \ar[d]^{\cM} \ar[r]^{\cM} & Y_{p(q(a_0))} \ar[d]^{=} \\
X_{p(q(a_0))} \ar[r]_{g_{q(a_0)}}  & H_{q(a_0)}  \ar[r]_{h_{q(a_0)}} &  Y_{p(q(a_0))} . \\}$$
The functor $q:\cI\to \cF_f$ is pre-cofinal, relative to $\barr{\cF_f}$, so there exists a morphism $a'_0\to a_0$ in $\cI$, such that the induced morphism
$$\xymatrix{
X_{p(q(a'_0))}\ar[d] \ar[r]^{g_{q(a'_0)}} &  H_{q(a'_0)} \ar[d] \ar[r]^{h_{q(a'_0)}} & Y_{p(q(a'_0))} \ar[d] \\
X_{p(q(a_0))} \ar[r]_{g_{q(a_0)}}  & H_{q(a_0)}  \ar[r]_{h_{q(a_0)}} &  Y_{p(q(a_0))}  \\}$$
factors as
$$\xymatrix{
X_{p(q(a'_0))}\ar[d] \ar[r]^{g_{q(a'_0)}} &  H_{q(a'_0)} \ar[d] \ar[r]^{h_{q(a'_0)}} & Y_{p(q(a'_0))} \ar[d] \\
X_{p(q(a_0))}\ar[d]^= \ar[r]^N &  Z        \ar[d]^{\cM} \ar[r]^{\cM} & Y_{p(q(a_0))} \ar[d]^{=} \\
X_{p(q(a_0))} \ar[r]_{g_{q(a_0)}}  & H_{q(a_0)}  \ar[r]_{h_{q(a_0)}} &  Y_{p(q(a_0))} .\\}$$
Composing the morphisms $H_{q(a'_0)}\to Z$ and $Z\to C$, we get the desired lift.
\end{proof}

\begin{define}\label{d:sub_factor}
Let $(\cC,N,\cM)$ be a factorization category.

A full sub factorization category is a full subcategory $\cC_s\subseteq\cC$ that is closed under finite limits in $\cC$, such that $(\mcal{C}_s,\cC_s\cap N,\cC_s\cap\mcal{M})$ is a factorization category.

A full sub factorization category $(\cC_s,N_s,\cM_s)$ is called \emph{dense} if for every diagram in $\cC$ of the form
$$A\to H \xrightarrow{\cM} B$$
such that $A,B\in \cC_s$, there exists a diagram in $\cC$ of the form
$$\xymatrix{
A \ar[dr]\ar[r]  & H'  \ar[r]^{\cM} \ar[d] & B  \\
\empty &  H\ar[ur]_{\cM} & \empty }$$
such that $H'\in \cC_s$.
\end{define}

Our main motivation for introducing the concept of a dense subcategory is the following:
\begin{lem}\label{l:dense}
Let $(\cC,N,\cM)$ be a factorization category and let $\cC_s\subseteq\mcal{C}$ be a dense full sub factorization category. Let
$f:X\to Y$ be a morphism in $\cC_s^\cT$, where $\cT$ is a small cofiltered category.

Consider the categories $\cF_f,\barr{\cF_f}$ defined for the factorization category
$(\cC,N,\cM)$ and the natural transformation $f$ as in Definition \ref{d:factorization_cat}. We define the categories $\cF_f^s$ and $\barr{\cF_f^s}$ to be those given by Definition \ref{d:factorization_cat} for the factorization category
$(\cC_s,N_s,\cM_s)$ and the natural transformation $f$.

Then the natural inclusion $\cF_f^s\hookrightarrow \cF_f$ is pre-cofinal, relative to $\barr{\cF_f}$.
\end{lem}

\begin{proof}
Note that $\cF_f^s$ and $\barr{\cF_f^s}$ are just the full subcategories of ${\cF_f}$ and $\barr{\cF_f}$, spanned by objects of the form
$$(t, X_t \xrightarrow{g} H \xrightarrow{h} Y_t),$$
such that $H\in \cC_s$.

Let
$$\xymatrix{
X_t \ar[d] \ar[r]^N & H \ar[d] \ar[r]^{\mcal{M}} & Y_t \ar[d]\\
X_{t'} \ar[r]^N &  H' \ar[r]^{\mcal{M}} & Y_{t'}    }$$
be a morphism in $\cF_f$, such that $H'\in \cC_s$. We need to find a morphism
$$\xymatrix{
X_{t''} \ar[d] \ar[r]^N & H'' \ar[d] \ar[r]^{\mcal{M}} & Y_{t''} \ar[d]\\
X_{t'} \ar[r]^N &  H' \ar[r]^{\mcal{M}} & Y_{t'}    ,}$$
in $\cF_f^s$, that factors through the first morphism, in $\barr{\cF_f}$, as in the following diagram:
$$\xymatrix{
X_{t''} \ar[d] \ar[r]^N & H'' \ar[d] \ar[r]^{\mcal{M}} & Y_{t''} \ar[d]\\
X_t \ar[d] \ar[r]^N & H \ar[d] \ar[r]^{\mcal{M}} & Y_t \ar[d]\\
X_{t'} \ar[r]^N &  H' \ar[r]^{\mcal{M}} & Y_{t'}    .}$$
Thus, it is enough to find a diagram in $\cC$ of the form
$$\xymatrix{
\empty   & H''\ar[dr]^{\cM}  \ar[d] &  \empty \\
X_t\ar[r]^N\ar[ur]^N &  H  \ar[r]^{\cM}& Y_t, }$$
such that $H''\in \cC_s$, and the composition $H''\to H\to H'\times_{Y_{t'}}Y_t$ belongs to $\cM$ (because then we can take $t''=t$, and the morphism $t''\to t$ to be the identity).

Consider the following diagram in $\cC$:
$$X_t\xrightarrow{N} H \xrightarrow{\cM} H'\times_{Y_{t'}}Y_t.$$
$X_t,H'\times_{Y_{t'}}Y_t\in \cC_s$, so by the condition in the theorem, there exists a diagram in $\cC$ of the form
$$\xymatrix{
X_t \ar[dr]_{N}\ar[r]  & H''' \ar[r]^{\cM} \ar[d] & H'\times_{Y_{t'}}Y_t,  \\
\empty &  H\ar[ur]_{\cM} & \empty }$$
such that $H'''\in \cC_s$.
Now, factor the map $X_t \to   H'''$, in $\cC_s$, as
$$\xymatrix{
X_t \ar[drr]_{N}\ar[r]^N & H''\ar[r]^{\cM} & H''' \ar[r]^{\cM} \ar[d] & H'\times_{Y_{t'}}Y_t,  \\
\empty & \empty &  H\ar[ur]_{\cM} & \empty }$$
to get the desired diagram. (Note that we have a pullback square
$$\xymatrix{
 H'\times_{Y_{t'}}Y_t \ar[d] \ar[r]^{\mcal{M}} & Y_t \ar[d]\\
 H' \ar[r]^{\mcal{M}} & Y_{t'}    ,}$$
so that $H'\times_{Y_{t'}}Y_t \xrightarrow{\cM}  Y_t$.)
\end{proof}

We now come to our main proposition about factorization of maps:

\begin{prop}\label{p:factor_gen_0}
Let $(\cC,N,\cM)$ be a factorization category and let $\cC_s\subseteq\mcal{C}$ be a small dense full sub factorization category.

Then every morphism $f:X\to Y$ in $\Pro(\cC_s)$ can be functorially factored in $\Pro(\cC)$ as $X\xrightarrow{g} H_f \xrightarrow{h} Y$, where $h$ is in $\Sp^{\cong}(\mcal{M})$ and $g$ is in $\Lw^{\cong}(N)\cap  {}^{\perp}\mcal{M}$ (see Definition \ref{def mor}).
\end{prop}

\begin{proof}
Without loss of generality we may assume that $f:X\to Y$ is a map in $\cC_s^\cT$ for some (small) cofiltered category $\cT$ (see Lemma \ref{every map natural}).

The proof will consist of finding a cofinite directed set $\cA_f$, together with a cofinal functor $r:\cA_f\to \cT$, and  a factorization $r^*X\xrightarrow{g} H_f \xrightarrow{h} r^*Y$, of $r^*f:r^*X\to r^*Y$, in the category $\cC^{\cA_f}$,  such that $h$ is a special $\mcal{M}$ map, and $g$ is a levelwise $N$ map that belongs to $ {}^{\perp}\mcal{M}$ as a map in $\Pro(\cC)$ (see Definition \ref{def natural}). Then $X\cong r^*X$ and $Y\cong r^*Y$ in $\Pro(\cC)$, and we get the desired factorization.

Let the categories $\cF_f,\barr{\cF_f},\cF_f^s$ and $\barr{\cF_f^s}$ be defined as in Lemma \ref{l:dense}.

Using Proposition \ref{p:lw_factor_gen} we see that to prove  Proposition \ref{p:factor_gen_0} it is enough to find a small cofinite directed set $\cA_f$, together with a functor $r:\cA_f\to \cT$, and  a factorization $r^*X\xrightarrow{g} H \xrightarrow{h} r^*Y$, of $r^*f:r^*X\to r^*Y$, in the category $\cC_s^{\cA_f}$,  such that $h$ is a special $\mcal{M}$ map, $g$ is a levelwise $N$ map,  and the induced composition $\cA_f \to \cF_f^s\hookrightarrow \cF_f$ is pre-cofinal relative to $\barr{\cF_f}$, where $\cA_f \to \cF_f^s$ is the map given by Lemma \ref{l:factor} (2).

We shall define $\cA_f$, $r:\cA_f\to \cT$ and the factorization $r^*X\xrightarrow{g} H \xrightarrow{h} r^*Y$ inductively (see Definition \ref{def_deg}).

We start by defining $\cA^{-1}_f := \emptyset$ and $r:\cA^{-1}_f \to \cT$, $r^*X\xrightarrow{g} H \xrightarrow{h} r^*Y$ in the only possible way.

Now suppose we have defined a small cofinite poset $\cA^n_f$ (with objects of degree up to $n$), a functor $r^n:\cA^n_f \to \cT$ and a factorization of $r^{n*}f$, denoted $r^{n*}X\xrightarrow{g} H \xrightarrow{h} r^{n*}Y$, such that $g$ is levelwise $N$ and $h$ is special $\mcal{M}$.

We define $\cB^{n+1}_f$ to be the set of all tuples  $(R,r:R^{\lhd}\to \cT, r^*X \xrightarrow{g} H \xrightarrow{h}  r^*Y)$, such that
$R$ is a finite section in $\cA^n_f$ (see Definition \ref{d:section}), $r:R^{\lhd}\to \cT$ is a functor such that $r|_R = r^n|_R$ and
$r^*X \xrightarrow{g} H \xrightarrow{h}  r^*Y$ is a factorization of $r^*f$ in $\cC_s^{R^{\lhd}}$, such that  $g$ is levelwise $N$, $h$ is special $\mcal{M}$ and, when  restricted to $\cC_s^{R}$, this factorization is the same as $r^{n*}X\xrightarrow{g} H \xrightarrow{h} r^{n*}Y$ restricted to $\cC_s^{R}$.

As a set, we define $\cA^{n+1}_f := \cA^n_f \coprod \cB_f^{n+1}$. For $c\in \cA^{n}_f$ we set $c < (R,r:R^{\lhd}\to \cT, r^*X \xrightarrow{g} H \xrightarrow{h}  r^*Y)$ iff $c \in R$. Thus we have defined an $n+1$-level cofinite poset $\cA^{n+1}_f$.
We now define $r^{n+1}:\cA^{n+1}_f \to \cT$ by $r^{n+1}|_{\cA^n_f} = r^{n} $ and $r^{n+1}(R,r:R^{\lhd}\to \cT, r^*X \xrightarrow{g} H \xrightarrow{h}  r^*Y) = r(\infty)$, where $\infty \in R^{\lhd}$ is the initial object. It is clear the factorization  $r^{n*}X\xrightarrow{g} H \xrightarrow{h} r^{n*}Y$ extends naturally to a factorization
$r^{n+1*}X\xrightarrow{g} H \xrightarrow{h} r^{n+1*}Y$ such that $g$ is levelwise $N$ and $h$ is special $\mcal{M}$.

Notice that at level zero we have $\cA^0_f=\cB^0_f=\Ob(\cF_f^s)$, and the map $r^{0}:\cA^{0}_f=\Ob(\cF_f^s) \to \cT$ is the natural projection.

Now we define $\cA_f = \cup \cA^n_f$.
Clearly $\cA_f$ is a small cofinite poset.
It is clear that by taking the limit on all the $r^n$ we obtain a functor $r:\cA_f \to \cT$, and a factorization $r^{*}X\xrightarrow{g} H \xrightarrow{h} r^{*}Y$ of $r^{*}X\xrightarrow{r^*f} r^{*}Y$, such that $g$ is levelwise $N$ and $h$ is special $\mcal{M}$.

Using Lemmas \ref{l:dense} and \ref{l:compose_onto_big} we see that we need only prove the following two things:
\begin{enumerate}
\item $\cA_f$ is cofiltered.
\item The induced functor $q:\cA_f \to \cF_f^s$ is pre-cofinal.
\end{enumerate}

To prove that $\cA_f$ is cofiltered we need to  show that for every finite section $R \subset \cA_f$,
there exists an element $c \in \cA_f$ such that $c\geq a$ for every $a\in R$ (see Lemma \ref{l:eqiv_cofiltered}).
Indeed let $R \subset \cA_f$ be a finite section. Since $R$ is finite there exists some $n \in \NN$ such that
$R \subset \cA_f^{n}$. We can take $c$ to be any element in $\cB^{n+1}_f$ of the form $(R,r:R^{\lhd}\to \cT, r^*X \xrightarrow{g} H \xrightarrow{h}  r^*Y)$. To show that such an element exists, note that since $\cT$ is cofiltered we can extend the functor $r:R\to \cT$ to a functor $r:R^{\lhd}\to \cT$ (see Lemma \ref{l:eqiv_cofiltered}). Now, the existence of the suitable factorization $r^*X \xrightarrow{g} H \xrightarrow{h}  r^*Y$ follows  from Lemma \ref{l:extend_factorization}.

We are now left to show that the functor $q:\cA_f\to \cF_f^s$ is pre-cofinal. Let $c\in \cA_f$, and let $d:i \to q(c)$ be a map in $\cF_f^s$. There exists a unique $n\geq 1$ such that $c\in \cA^n_f\setminus \cA^{n-1}_f=\cB^n_f$.
We can write $c$ as $c = (R,r:R^{\lhd} \to \cT,r^*X \xrightarrow{g} H \xrightarrow{h}  r^*Y)$, where $R$ is a finite section in $\cA^{n-1}_f$.
Then we have $q(c) = (r(\infty),X_{r(\infty)} \xrightarrow{g} H(\infty) \xrightarrow{h} Y_{r(\infty)})$, where $\infty\in R^{\lhd}$ is the initial object.

Note that $R_c:=\{a\in \cA^n_f|c\geq a\}\subseteq \cA^n_f $ is naturally isomorphic to $R^{\lhd}$.
Let
$$\xymatrix{
X_t  \ar[r]^{g} \ar[d] & H\ar[r]^{h}\ar[d] & Y_t \ar[d]\\
X_{r(c)} \ar[r]^{g_{c}}  & H_{c}\ar[r]^{h_{c}} & Y_{r(c)}\\
}$$
be the morphism $d$ in $\cF_f^s$.

Now it is enough to find $c' \in \cB^{n+1}_f$ such that $c'>c$, $q(c') = (t,X_t \xrightarrow{g} H \xrightarrow{h} Y_t)$, and the induced map $q(c') \to q(c)$ is exactly $d$.

We shall take $c' := (R_c,r':R_c^{\lhd}\to \cT, r'^*X \xrightarrow{g} H' \xrightarrow{h}  r'^*Y)$, where
$$r'|_{R_c} = r'|_{R^{\lhd}} = r|_{R^{\lhd}},\:\:r'(\infty') = t,$$
and $\infty'\in R_c^{\lhd}$ is the initial object.

To define the functor $r':R_c^{\lhd}\to \cT$, it remains to define a morphism $r'(\infty') = t\to r(c)$, and we take this morphism to be the one given by $d:i \to q(c)$. We extend the factorization $r^*X \xrightarrow{g} H \xrightarrow{h}  r^*Y$ to a factorization $r'^*X \xrightarrow{g} H' \xrightarrow{h}  r'^*Y$, using the morphism $d:i \to q(c)$. To show that $c' \in \cB^{n+1}_f$, it remains to check that $H' \xrightarrow{h}  r'^*Y$ is a special $\cM$ map in $\cC_s^{R_c^{\lhd}}$. We only need to check the special condition on $\infty'\in R_c^{\lhd}$. This just says that the induced map $k:H\to  H_{c}\times_{Y_{r(c)}}Y_t$ belongs to $\cM$. But this follows from the fact that $d$ is a morphism in $\cF_f^s$.

Now it is clear that $c'>c$, $q(c') = (t,X_t \xrightarrow{g} H \xrightarrow{h} Y_t)$, and the induced map $q(c') \to q(c)$ is exactly $d$.
\end{proof}

\begin{rem}
Proposition \ref{p:factor_gen_0} holds verbatim also if we assume that the dense full sub factorization category $\cC_s\subseteq\mcal{C}$ is essentially small instead of small. This is because we can then find a small equivalent full subcategory $\cC'_s\subseteq\cC_s$. Then $\cC'_s$ is a dense full sub fibration category of $\cC$ and we can use it in Proposition \ref{p:factor_gen_0} instead of $\cC_s$.
\end{rem}

Proposition \ref{p:factor_gen_0} has the following two immediate corollaries:
\begin{prop}\label{p:factor_gen_big}
Let $(\cC,N,\cM)$ be a factorization category.
Suppose that for every morphism $f:X\to Y$ in $\Pro(\cC)$, we can find an essentially small dense full sub factorization category $\cC_s\subseteq\cC$ (see Definition \ref{d:sub_factor}) such that $X,Y\in \Pro(\cC_s)$.

Then every morphism $f:X\to Y$ in $\Pro(\cC)$ can be factored \emph{(not necessarily functorially)} in $\Pro(\cC)$ as $X\xrightarrow{g} H_f \xrightarrow{h} Y$, where $h$ is in $\Sp^{\cong}(\mcal{M})$ and $g$ is in $\Lw^{\cong}(N)\cap  {}^{\perp}\mcal{M}$ (see Definition \ref{def mor}).
\end{prop}

\begin{prop}\label{p:factor_gen}
Let $(\cC,N,\cM)$ be an essentially small factorization category.
Then every morphism $f:X\to Y$ in $\Pro(\cC)$ can be functorially factored as $X\xrightarrow{g} H_f \xrightarrow{h} Y$, where $h$ is in $\Sp^{\cong}(\mcal{M})$ and $g$ is in $\Lw^{\cong}(N)\cap  {}^{\perp}\mcal{M}$ (see Definition \ref{def mor}).
\end{prop}

\begin{rem}\
\begin{enumerate}
\item Notice that by assuming that our factorization category is essentially small we gain \emph{functorial} factorizations.
\item
There is a strong connection between Proposition \ref{p:factor_gen} and a dual version of Quillen's small object argument. This, as well as other aspects of the ``small" case, will be discussed in a future paper.
\end{enumerate}
\end{rem}

\section{The Model Structure on $\Pro(\mcal{C})$}\label{s:model}
In this section we show how to construct a model structure out of a weak fibration category. Namely, given a weak fibration category $(\cC,\cW,\cF)$ that satisfies two extra conditions (which we call ``homotopically small" and ``pro-admissible"), we shall construct a model structure on $\Pro(\mcal{C})$. The weak equivalences in this model structure will be maps that are isomorphic to a natural transformation that is a levelwise $\cW$ map. The cofibrations will be maps that have the left lifting property with respect to $\cF\cap\cW$ (considered as maps in $\Pro(\cC)$). The fibrations will be maps that are retracts of natural transformations that are special $\cF$ maps. However, we will not get a model category in the modern sense of the word since it will not necessarily have arbitrary colimits and functorial factorizations. Thus, in order to make our terminology clear, we make the following definition:

\begin{define}\label{d:model_structure}
A \emph{model structure} on a category $\cC$ consists of three subcategories $\cW,\cF,\cC$, each containing all the objects of $\cC$, called weak equivalences, fibrations, and cofibrations, satisfying the following properties:
\begin{enumerate}
\item $\cW$ satisfies the two out of three property.
\item $\cW,\cF,\cC$ are closed under retracts.
\item Any morphism in $\cC$ can be factored (not necessarily functorially) into a cofibration followed by a trivial
fibration, and into a trivial cofibration followed by a fibration.
\item The trivial cofibrations have the left lifting property with respect to fibrations, and the cofibrations have the
left lifting property with respect to trivial fibrations.
\end{enumerate}
\end{define}

Thus, a model structure on a category just means that all the usual axioms for a model category are satisfied, except maybe completeness, cocompleteness and functoriality of the factorizations.

\begin{define}\label{d:rel}
A \emph{relative category} is a pair $(\cC,\cW)$ consisting of a category $\cC$, and a subcategory $\cW\subseteq \cC$ that contains all the isomorphisms and satisfies the two out of three property. The maps in $\cW$ are called \emph{weak equivalences.}
\end{define}

\begin{rem}
Any weak fibration category is naturally a relative category, when ignoring the fibrations.
\end{rem}

\begin{define}\label{d:admiss}
A relative category $(\cC,\cW)$ is called \emph{pro-admissible} (respectively \emph{ind admissible}) if $\Lw^{\cong}(\mcal{W})\subseteq  \Pro(\cC)^{[1]}$ (respectively $\Lw^{\cong}(\mcal{W})\subseteq  \Ind(\mcal{C})^{[1]}$) satisfies the two out of three property.
\end{define}

\begin{example}[Isaksen]\label{e:proper}
Let $\cM$ be a proper model category and $\cW$ the class of weak equivalences in $\cM$. Then it follows from Lemma 3.5 and Lemma 3.6 of~\cite{Isa} that $(\cM,\cW)$ is pro-admissible.
\end{example}

\begin{rem}
Example~\ref{e:proper} can be generalized to a wider class of relative categories via the notion of \textbf{proper factorizations}, see~\cite[Proposition 3.7]{BaSc1}.
\end{rem}

\begin{example}\label{e:finite SS}
Let $\cS_f$ denote the category of finite simplicial sets, that is, simplicial sets having a finite number of non-degenerate simplicies. Let $\cW$ denote the class of weak equivalences between objects in $\cS_f$ (in the standard sense of Quillen). Then by \cite[Theorem 4.6]{BaSc1} the relative category $(\cS_f,\cW)$ is ind-admissible (or in other words, $(\cS_f^{\op},\cW^{\op})$ is pro-admissible).
\end{example}

\begin{rem}
Example~\ref{e:finite SS} was used in \cite{BaSc1} to show that the category of weak equivalences in the standard model structure on simplicial sets is finitely accessible.
\end{rem}

\begin{example}
Let $\CM$ denote the category compact metrizable spaces and let $\cW$ denote the class of homotopy equivalences between them. Then it is shown in \cite{Bar} that the relative category $(\CM,\cW)$ is ind-admissible (or in other words, that $(\CM^{\op},\cW^{\op})$ is pro-admissible).
\end{example}

\begin{example}
Let $\Csep$ denote the category separable $C^*$-algebras and let $\cW$ denote the class of homotopy equivalences between them. Then it is shown in \cite{BJM} that the relative category $(\Csep,\cW)$ is pro-admissible.
\end{example}

\begin{rem}
We do not know of any example of a relative category which is not pro-admissible.
\end{rem}

\begin{define}\label{d:sub_weak}
Let $(\mcal{C},\mcal{W},\mcal{F})$ be a weak fibration category. We define:
\begin{enumerate}
\item A full sub weak fibration category is a full subcategory $\cC_s\subseteq\cC$ that is closed under finite limits in $\cC$, such that $(\mcal{C}_s,\cC_s\cap\mcal{W},\cC_s\cap\mcal{F})$ is a weak fibration category.

\item A full sub weak fibration category $\cC_s\subseteq\cC$ is called \emph{dense} if the following conditions hold:
\begin{enumerate}
\item The full sub factorization category $(\cC_s,\cW_s,\cF_s)$ of $(\cC,\cW,\cF)$ is dense (see Definition \ref{d:sub_factor}).
\item The full sub factorization category $(\cC_s,\cC_s,\cF_s\cap\cW_s)$ of $(\cC,\cC,\cF\cap\cW)$ is dense (see Definition \ref{d:sub_factor}).
\end{enumerate}
\item The weak fibration category $\cC$ is called \emph{homotopically small} if for every object $X$ in $\Pro(\cC)$, we can find an essentially small dense full sub weak fibration category $\cC_s\subseteq\mcal{C}$ such that $X\in \Pro(\cC_s)$.
\end{enumerate}
\end{define}

\begin{rem}
Note that any essentially small weak fibration category is clearly homotopically small.
\end{rem}

\begin{lem}\label{l:hom small mor}
Let $\cC$ be a homotopically small weak fibration category and let $f:X\to Y$ be a morphism in $\Pro(\cC)$. Then there exists an essentially small dense full sub weak fibration category $\cC_s\subseteq\mcal{C}$ such that $X,Y\in \Pro(\cC_s)$.
\end{lem}

\begin{proof}
Suppose that $X=\{X_i\}_{i\in\cI}$ and $Y=\{Y_j\}_{j\in\cJ}$. Clearly the categories $\cI^{\rhd}$ and $\cJ^{\rhd}$ are also cofiltered categories (see Definition \ref{d:triangle}) and we can extend $X$ and $Y$ to diagrams $X^{\rhd}:\cI^{\rhd}\to\cC$ and $Y^{\rhd}:\cJ^{\rhd}\to\cC$ by defining $X(-\infty)=Y(-\infty)=*$, where $*\in \cC$ is the terminal object. Now we can apply the definition of homotopically small to the object
$$\{X^{\rhd}_i\times Y^{\rhd}_j\}_{(i,j)\in\cI^{\rhd}\times \cJ^{\rhd}}$$
in $\Pro(\cC)$.
\end{proof}

\begin{lem}\label{cardinal_condition}
Let $(\cC,\cW,\cF)$ be a weak fibration category and $\gamma$ an infinite cardinal. Suppose that for every cardinal $\lambda\geq\gamma$ we are given a full sub weak fibration category $\cC_{\lambda}\subseteq\cC$ such that the following hold:
\begin{enumerate}
\item For every $\lambda\geq\gamma$ we have that $\cC_{\lambda}\subseteq\cC$ is essentially small and dense.
\item For every $\mu\geq\lambda\geq\gamma$ we have $\cC_{\lambda}\subseteq\cC_\mu$.
\item We have $\bigcup_{\lambda\geq\gamma}\cC_\lambda=\cC$.
\end{enumerate}
Then $(\cC,\cW,\cF)$ is homotopicaly small.
\end{lem}

\begin{proof}
Let $\{X\}_{i\in\cI}$ be an object in $Pro(\cC)$. Using the fact that $\cI$ is small, and conditions (2) and (3), we can choose some cardinal $\lambda \geq \gamma$ such that  $X_i\in \cC_\lambda$ for every $i\in\cI$. By condition (1) the  full sub weak fibration category
$\cC_\lambda\subseteq \cC$ is dense and essentially small and we have $X\in Pro(\cC_\lambda)$.
\end{proof}

\begin{example}\label{e:simplicial}
Let $\cS$ be the category of simplicial sets, endowed with the Kan-Quillen model structure. Then $\cS$ is a weak fibration category.

Let $\lambda$ be a cardinal. We call a simplicial set $X$ $\lambda$-\emph{bounded} if $|X_n| \leq \lambda$  for every $n\in \mathbb{N}$. We denote by $\cS_{\lambda}$ the full subcategory of  $\cS$ spanned by the $\lambda$-bounded simplicial sets.

Then it follows from Proposition \ref{p:SPS_wfc_l} and Proposition \ref{p:SPS_dense}
that if $\lambda \geq \aleph_0$ then $\cS_{\lambda}$ is an essentially small dense full sub weak fibration category of $\cS$. In particular, it follows from Lemma \ref{cardinal_condition} that the weak fibration category $\cS$ is homotopically small.
\end{example}

\begin{prop}\label{p:acyclic cofibration}
Let $(\cC,\cW,\cF)$ be a weak fibration category. Then in $\Pro(\cC)$ we have
$${}^{\perp}(\cF \cap \cW)\cap \Lw^{\cong}(\cW)={}^{\perp}\cF.$$
\end{prop}

\begin{proof}
Since $\cC$ is a weak fibration category we know that $\cC$ has finite limits and that $\Mor(\cC)=\cF\circ\cW$. By~\cite[Proposition 4.1]{BaSc} we know that $\Mor(\Pro(\cC))=\Sp^{\cong}(\cF)\circ \Lw^{\cong}(\cW)$. Now, by Lemmas \ref{l:SpMo_is_Mo}, \ref{l_lift} and \ref{l:ret_lw} (in this order) we have that
\[{}^{\perp} \cF={}^{\perp}\Sp^{\cong}(\cF)\subseteq \R(\Lw^{\cong}(\cW))=\Lw^{\cong}(\cW).\]
Thus we clearly have
$${}^{\perp}\cF\subseteq{}^{\perp}(\cF \cap \cW)\cap \Lw^{\cong}(\cW).$$

We are left to show that
$${}^{\perp}(\cF \cap \cW)\cap \Lw^{\cong}(\cW) \subseteq {}^\perp \mcal{F}.$$
Let $X\to Y$ be a map in ${}^{\perp}(\cF \cap \cW) \cap \Lw^{\cong}(\cW)$. We need to show that there exists a lift in every square of the form
$$\xymatrix{
 X \ar[d]\ar[r]& A \ar[d]^{\mcal{F}}\\
 Y \ar[r]     & B.}$$
Without loss of generality, we may assume that $X\to Y$ is a natural transformation, which is a levelwise $\mcal{W}$-map. Thus we have a diagram of the form
$$\xymatrix{
\{X_t\}_{t\in \cT} \ar[d]_{{}^{\perp}(\cF \cap \cW)}\ar[r]& A \ar[d]^{{\mcal{F}}}\\
 \{Y_t\}_{t\in \cT} \ar[r]    &  B.}$$
By the definition of morphisms in $\Pro(\cC)$, there exists $t \in \cT$ such that
the above square factors as
$$
\xymatrix{
\{X_t\}_{t\in \cT} \ar[d]_{{}^{\perp}(\cF \cap \cW)}\ar[r]& X_t \ar[d]^{\mcal{W}}\ar[r]& A \ar[d]^{{\mcal{F}}}\\
\{Y_t\}_{t\in \cT}      \ar[r]     & Y_t \ar[r]     &  B.}$$
By taking the fiber product we get the following diagram:
$$\xymatrix{
\{X_t\}_{t\in \cT} \ar[dd]_{{}^{\perp}(\cF \cap \cW)}\ar[r]& X_t \ar[d] \ar@/_3pc/[dd]_{\mcal{W}} \ar[r]& A \ar@{=}[d]\\
 \empty               &Y_t\times_B A \ar[d]^{\mcal{F}}\ar[r]& A \ar[d]^{{\mcal{F}}} \\
\{Y_t\}_{t\in \cT} \ar[r]     & Y_t \ar[r]     &  B.}$$
Factoring the map $X_t \to Y_t\times_B A$ as $X_t \xrightarrow{\cF}H\xrightarrow{\cW}  Y_t\times_B A$ and composing, we obtain
$$\xymatrix{
\{X_t\}_{t\in \cT}  \ar@{=}[d]\ar[r]& X_t \ar[d]^{\mcal{W}}  \ar[r]&  A \ar@{=}[d] \\
\{X_t\}_{t\in \cT}  \ar[d]_{{}^{\perp}(\cF \cap \cW)} \ar[r]&     H\ar[r] \ar[d]^{\mcal{F} \cap \mcal{W}}& A \ar[d]^{{\mcal{F}}}  \\
\{Y_t\}_{t\in \cT}  \ar[r]     & Y_t \ar[r]     &  B,}$$
where the map $H \to Y_t$ belongs to $\mcal{W}$, because $\mcal{W}$ has the two out of three property.
But now we clearly have a lift in the left bottom square.
\end{proof}

\begin{thm}\label{t:model_big}
Let $(\mcal{C},\mcal{W},\mcal{F})$ be a homotopically small (see Definition \ref{d:sub_weak}) pro-admissible (see Definition \ref{d:admiss}) weak fibration category.
Then there exists a model structure on $\Pro(\mcal{C})$ (see Definition \ref{d:model_structure}) such that:
\begin{enumerate}
\item The weak equivalences are $\mathbf{W} := \Lw^{\cong}(\mcal{W})$.
\item The fibrations are $\mathbf{F} := \R(\Sp^{\cong}(\mcal{F}))$.
\item The cofibrations are $\mathbf{C} := {}^{\perp} (\mcal{F}\cap \mcal{W})$.
\end{enumerate}
Moreover, the acyclic fibrations in this model structure are precisely
$$\mathbf{F}\cap\mathbf{W} = \R(\Sp^{\cong}(\mcal{F}\cap\cW)),$$
and the acyclic cofibrations are precisely
$$\mathbf{C}\cap\mathbf{W} = {}^{\perp} \mcal{F}.$$
\end{thm}

\begin{rem}\label{r:complete}
The proof of \cite{IsaS}, Proposition 11.1, shows that $\Pro(\cC)$ is complete (since $\cC$ has finite limits). This proof also shows that (for any cardinal $\kappa$) if $\cC$ has ($\kappa$-)small colimits, so does $\Pro(\cC)$.
\end{rem}

\begin{example}
Combining Example \ref{e:proper} with Example \ref{e:simplicial} we see that Theorem \ref{t:model_big} applies for the weak fibration structure on simplicial sets, given by the standard model structure.
\end{example}

\begin{proof}
We verify the different axioms of a model structure appearing in Definition \ref{d:model_structure}.

Clearly, the classes $\mathbf{W},\mathbf{F}$ and $\mathbf{C}$ contain all isomorphisms and $\mathbf{C}$ is a  subcategory.
Since $\cC$ is pro-admissible, the class $\mathbf{W}$ satisfies the two out of three property (and is in particular a subcategory), so Condition (1) in Definition \ref{d:model_structure} is satisfied.

The class $\mathbf{F}$ is trivially closed under retracts. From Lemma \ref{l:ret_lw} if follows that $\mathbf{W}$ is closed under retracts. From Lemma \ref{c:ret_lift} we get that $\mathbf{C}$ is closed under retracts, so Condition (2) in Definition \ref{d:model_structure} is satisfied.

By Corollary \ref{l:forF_sp_is_lw} we have
$$\Sp^{\cong}(\mcal{F} \cap \mcal{W}) \subseteq \Lw^{\cong}(\mcal{F} \cap \mcal{W}) \subseteq  \Lw^{\cong}(\mcal{W}) = \mathbf{W}.$$
We also have
$$\Sp^{\cong}(\mcal{F} \cap \mcal{W}) \subseteq \Sp^{\cong}(\mcal{F}) \subseteq \R(\Sp^{\cong}(\mcal{F}))=\mathbf{F}.$$
Since $\mathbf{F}$ and $\mathbf{W}$ are closed under retracts, we have
$$\R(\Sp^{\cong}(\mcal{F} \cap \mcal{W})) \subseteq \mathbf{F}\cap \mathbf{W} .$$

Let us now verify Condition (3) in Definition \ref{d:model_structure}.
Let $X\to Y$ be a map in $\Pro(\cC)$. By Lemma \ref{l:hom small mor}, we can apply Proposition \ref{p:factor_gen_big} for the factorization category $(\cC,\mcal{C},\mcal{F} \cap \mcal{W})$, and get a factorization of $X\to Y$ of the form
$$ X \xrightarrow{{}^\perp (\mcal{F} \cap \mcal{W})} Z \xrightarrow{\Sp^{\cong}(\mcal{F}\cap \mcal{W})} Y. $$
But $\mathbf{C} = {}^\perp (\mcal{F} \cap \mcal{W})$, and we have shown that
$\Sp^{\cong}(\mcal{F}\cap \mcal{W}) \subseteq \mathbf{F} \cap \mathbf{W}$, so we have factored $X\to Y$ as
$$ X \xrightarrow{\mathbf{C}} Z \xrightarrow{\mathbf{F}\cap \mathbf{W}} Y.$$
By Lemma \ref{l:hom small mor}, we can apply Proposition \ref{p:factor_gen_big} also for the factorization category $(\cC,\mcal{W},\mcal{F})$, and get a factorization of $X\to Y$ of the form
$$ X \xrightarrow{ \Lw^{\cong}(\mcal{W})\cap{}^\perp \mcal{F}} Z \xrightarrow{\Sp^{\cong}(\mcal{F})} Y. $$
But $\Sp^{\cong}(\mcal{F}) \subseteq \mathbf{F} $, $\mathbf{W} = \Lw^{\cong}(\mcal{W})$ and $ {}^\perp \mcal{F} \subseteq {}^\perp (\mcal{F}\cap \mcal{W})=\mathbf{C}$, so we have factored $X\to Y$ as
$$ X \xrightarrow{\mathbf{C} \cap \mathbf{W}} Z \xrightarrow{\mathbf{F}} Y, $$
and we get that  Condition (3) in Definition \ref{d:model_structure} is satisfied.

We now turn to Condition (4) in Definition \ref{d:model_structure}.
From Proposition \ref{p:acyclic cofibration}, Lemma \ref{l:SpMo_is_Mo} and Lemma \ref{c:ret_lift} we get that
$$\mathbf{C}\cap \mathbf{W} ={}^{\perp} \cF={}^{\perp} \Sp^{\cong}(\cF)={}^{\perp} \R(\Sp^{\cong}(\cF))={}^{\perp} \mathbf{F},$$
so in particular $\mathbf{C}\cap \mathbf{W}\perp \mathbf{F},$
or
$\mathbf{F} \subseteq (\mathbf{C}\cap \mathbf{W})^{\perp} .$
Before verifying the other part of Condition (4), we take a pause to show that
$\mathbf{F}$ is closed under composition.
Since
$\mathbf{F}$ is closed under retracts, and we have factorizations $\Mor(\Pro(\cC))=\mathbf{F}\circ(\mathbf{C} \cap \mathbf{W})$, we get by Lemma \ref{l_lift} that $(\mathbf{C}\cap \mathbf{W})^{\perp} \subseteq \mathbf{F}.$
Combining this with what we have just shown, we obtain
$$\mathbf{F} = (\mathbf{C}\cap \mathbf{W})^{\perp},$$
so in particular, $\mathbf{F}$ is closed under composition.

It remains to verify that $\mathbf{C}\perp \mathbf{F}\cap \mathbf{W}$ and to show that
$$\mathbf{F}\cap \mathbf{W} = \R(\Sp^{\cong}(\mcal{F} \cap \mcal{W})).$$
To do this we will need to use again the pro-admissibility of $\cC$.
Consider the category $\Pro_{\mathbf{W}}(\cC)$ with objects the same as $\Pro(\cC)$, but with morphisms only maps in $\mathbf{W}$ (recall that $\mathbf{W}$ is closed under composition).
Applying Proposition \ref{p:factor_gen_big} for the factorization category $(\cC,\mcal{C},\mcal{F} \cap \mcal{W})$ we get that $$\Mor(\Pro(\cC))=\R(\Sp^{\cong}(\mcal{F}\cap \mcal{W}))\circ\mathbf{C}.$$
Since $ \R(\Sp^{\cong}(\mcal{F} \cap \mcal{W})) \subseteq \mathbf{W}$ and $\mathbf{W}$ has the two out of three property, we get that
$$\Mor(\Pro_{\mathbf{W}}(\cC))=\R(\Sp^{\cong}(\mcal{F}\cap \mcal{W}))\circ(\mathbf{C}\cap\mathbf{W}).$$
Recall that we have shown that
$\mathbf{F} = (\mathbf{C}\cap \mathbf{W})^{\perp} .$
Using Lemma \ref{l_lift}, we get that
$$\mathbf{F}\cap\mathbf{W}=(\mathbf{C}\cap \mathbf{W})^{\perp} \cap\mathbf{W} \subseteq (\mathbf{C}\cap \mathbf{W})^{\perp}_{\Pro_{\mathbf{W}}(\cC)}\subseteq \R(\Sp^{\cong}(\mcal{F} \cap \mcal{W})).$$
Combining this with what we have shown before we obtain
$$\mathbf{F}\cap \mathbf{W} = \R(\Sp^{\cong}(\mcal{F} \cap \mcal{W})).$$
Using Lemma \ref{l:SpMo_is_Mo} and Lemma \ref{c:ret_lift} we get that
$$\mathbf{C}  = {}^\perp (\mcal{F}\cap \mcal{W})  = {}^\perp \Sp^{\cong}(\mcal{F}\cap \mcal{W}) = {}^\perp \R(\Sp^{\cong}(\mcal{F}\cap \mcal{W})) = {}^\perp (\mathbf{F}\cap \mathbf{W}).$$
Thus Condition (4) in Definition \ref{d:model_structure} is satisfied and we are done.
\end{proof}

\begin{rem}\label{r:proper}
In this remark we compare Theorem \ref{t:model_big} with the main result in Isaksen's paper \cite{Isa}.
\begin{enumerate}
\item Theorem \ref{t:model_big} was proved by Isaksen for the case that the weak fibration category $(\mcal{C},\mcal{W},\mcal{F})$ comes from a model category $(\mcal{C},\mcal{C}of,\mcal{W},\mcal{F})$. The condition of being homotopically small is then not needed. He also shows that in this case the cofibrations in $\Pro(\mcal{C})$ are given by $\mathbf{C}=\Lw^{\cong}(\mcal{C}of)$ and the acyclic cofibrations in $\Pro(\mcal{C})$ are given by $\mathbf{C}\cap\mathbf{W}=\Lw^{\cong}(\mcal{C}of\cap \cW)$. Note that the results of~\cite{Isa} are stated for a proper model category $\cC$, however, the properness of $\cC$ is only used to show that $\cC$ is pro-admissible (see Example \ref{e:proper}), while the arguments of~\cite{Isa} apply verbatim to any pro-admissible model category.

\item The approach taken by Isaksen is to begin with a model structure on $\mcal{C}$, and to use it to define a model structure also on $\Pro(\mcal{C})$. As we see here, the latter may exist without the former. Namely, $\mcal{C}$ can be a weak fibration category that is not a model category, while on $\Pro(\mcal{C})$ there will still be an induced model structure. The main reason for this phenomenon is that the absence of an initial factorization in $\mcal{C}$ can be solved when working in $\Pro(\mcal{C})$ by simply ``running over" all possible factorizations (see the introduction, and the proof of Proposition \ref{p:factor_gen_0}).
\end{enumerate}
\end{rem}

\section{Weak Right Quillen Functors}\label{s:Quillen}
In this section we discuss a natural notion of a morphism between weak fibration categories, which we call a \emph{weak right Quillen functor}. We discuss when a weak right Quillen functor between homotopically small pro admissible weak fibration categories gives rise to a right Quillen functor between the corresponding model structures on the pro categories.

\begin{define}
Let $F:\mcal{D}\to \mcal{C}$ be a functor between two weak fibration categories. Then $F$ is called a \emph{weak right Quillen functor} if $F$ commutes with finite limits, and preserves fibrations and trivial fibrations.
\end{define}

\begin{rem}
If $F:\mcal{D}\to \mcal{C}$ is a weak right Quillen functor between \emph{model} categories then $F$ is not necessarily a right Quillen functor since $F$ is not assumed to have a left adjoint.
\end{rem}

The main fact we want to prove about weak right Quillen functors is the following:

\begin{prop}\label{p:RQFunc_big}
Let $F:\mcal{D}\to \mcal{C}$ be a weak right Quillen functor between two homotopically small pro admissible weak fibration categories. Then the prolongation of $F$
$$\Pro(F):\Pro(\mcal{D})\to \Pro(\mcal{C})$$
preserves fibrations and trivial fibrations in the model structures defined in Theorem \ref{t:model_big}.
\end{prop}

\begin{proof}
For simplicity we write $F$ instead of $\Pro(F)$. We will show that $F$ preserves fibrations, and the proof that $F$ preserves trivial fibrations is exactly the same.

Let $f:X\to Y$ be a fibration in $\Pro(\mcal{D})$. Then by definition $f\in \R(\Sp(\mcal{F}_{\cD}))$. We need to show that $F(f)\in \R(\Sp(\mcal{F}_{\cC}))$.
The map $f$ is a retract of some map $g\in \Sp(\mcal{F}_{\cD})$. It follows that $F(f)$ is a retract of $F(g)$, so it is enough to show that $F(g)\in \Sp(\mcal{F}_{\cC})$.

The fact that $g:A\to B$ is in $\Sp(\mcal{F}_{\cD})$ means that the indexing category of both $A$ and $B$ is a cofinite directed set $\cT$ and that $g$ is a natural transformation such that the natural map
$$A_t \to B_t \times_{\lim_{s<t} B_s} \lim_{s<t} A_s$$
is in $\cF_{\cD},$ for every $t$ in $\cT$.

The indexing category of both $F(A)$ and $F(B)$ is a cofinite directed set $\cT$ and $F(g)$ is clearly a natural transformation. Let  $t\in \cT$. Applying $F$ to the map above and using the fact that $F$ commutes with finite limits, we see that
$$F(A_t) \to F(B_t \times_{\lim_{s<t} B_s} \lim_{s<t} A_s)\cong F(B_t) \times_{\lim_{s<t} F(B_s)} \lim_{s<t} F(A_s)$$
is in $\cF_{\cC}$.
Thus $F(g)\in \Sp(\mcal{F}_{\cC})$.
\end{proof}

Proposition \ref{p:RQFunc_big} has the following immediate corollary:
\begin{cor}
Let $F:\mcal{D}\to \mcal{C}$ be a weak right Quillen functor between two homotopically small pro admissible weak fibration categories. Assume that the  prolongation of  $F$ to $\Pro(\mcal{D})$ has a left adjoint:
$$L_F:\Pro(\mcal{C})\adj \Pro(\mcal{D}):\Pro(F).$$
Then the above adjunction is a Quillen pair, relative to the model structures defined in Theorem \ref{t:model_big}.
\end{cor}

We would now like to discuss sufficient conditions for the prolongation of  $F$ to have a left adjoint. For this we need to use the notion of a \emph{small functor}. The following is taken from \cite{DaLa} (see also the references there).

\begin{define}
Let $\cC$ be a (not necessarily small) category. A functor $F:\cC\to \Set$ is called \emph{small}, if it satisfies one of the following equivalent conditions:
\begin{enumerate}
\item $F$ is a (small) colimit of representables (that is, there exists a small category $\cI$ and a functor $X:\cI\to \cC^{\op}$ such that for every $c\in\cC$ we have $F(c)=\colim_{i\in \cI}\Hom_{\cC}(X(i),c)$).
\item There exists a small full subcategory $\cC_0\subseteq\cC$ such that $F$ is a left Kan extension of $F|_{\cC_0}$ along the natural inclusion.
\end{enumerate}
\end{define}

\begin{rem}
Let $\cC$ be a small category. Then clearly, every functor $F:\cC\to \Set$ is small.
\end{rem}

\begin{prop}
Let $\cC$ be a locally presentable category (see \cite{AR}). A functor $F:\cC\to \Set$ is small iff it is \emph{accessible} (that is,
iff it preserves $\kappa$-filtered colimits for some regular cardinal $\kappa$).
\end{prop}

Let $\cC$ be a category. We denote by $(\Set^{\cC^{\op}})_{sm}$ the category of small functors from $\cC^{\op}$ to $\Set$, with natural transformations between them. (It can be shown that this is indeed a category, that is, the $\Hom$ sets are small.) This category is always cocomplete. Furthermore, the Yoneda embedding $\cC\to (\Set^{\cC^{\op}})_{sm}$ exhibits $(\Set^{\cC^{\op}})_{sm}$ as the free completion of $\cC$ under colimits. If $\cC$ is complete, then $(\Set^{\cC^{\op}})_{sm}$ is also complete.

We can use $(\Set^{\cC^{\op}})_{sm}$ to construct the free completion of $\cC$ under any type of colimits. Namely, given a class $\mathbb{I}$ of small diagrams, the free completion of $\cC$ under colimits of shape $\mathbb{I}$ is given by the full subcategory of $(\Set^{\cC^{\op}})_{sm}$ obtained by closing the representable functors under $\mathbb{I}$-shaped colimits in $(\Set^{\cC^{\op}})_{sm}$. In particular, $\Ind(\cC)$ can be defined as the full subcategory of $(\Set^{\cC^{\op}})_{sm}$ obtained by closing the representable functors under filtered colimits. If $\cC$ has finite colimits, $\Ind(\cC)$ is the full subcategory of $(\Set^{\cC^{\op}})_{sm}$ spanned by the functors that commute with finite limits.

Dually, if $\cC$ has finite limits, $\Pro(\cC)$ can be defined as the full subcategory of $(\Set^{\cC})_{sm}^{\op}$ spanned by the functors that commute with finite limits. Any such functor can be represented as a limit in $(\Set^{\cC})_{sm}^{\op}$ of representables, over a small cofiltered indexing category. This diagram gives the usual representation of this functor as an object in $\Pro(\cC)$. Given an object in $\Pro(\cC)$ in the usual presentation, the corresponding small diagram is the limit in $(\Set^{\cC})_{sm}^{\op}$ of the corresponding diagram of representables.

We now come to our criterion for the prolongation of a functor to have a left adjoint.

\begin{lem}\label{l:l_adjoint}
Let $F:\mcal{D}\to \mcal{C}$ be a functor.
\begin{enumerate}
\item If $\cC$ and $\cD$ are locally presentable, and $F$ is accessible and commutes with finite limits, then prolongation of  $F$ to $\Pro(\cD)$ has a left adjoint.
\item If $F$ has a left adjoint $G:\cC \to \cD$, then prolongation of  $G$ to $\Pro(\cC)$ is left adjoint to the  prolongation of  $F$ to $\Pro(\cD)$.
\end{enumerate}
\end{lem}

\begin{proof}
\begin{enumerate}
\item Let $X  = \{X_j\}_\cJ \in \Pro(\cC) $. We need to show that the functor
$$\Hom_{\Pro(\cC)}(\{X_j\}_\cJ,\Pro(F)(-)):Pro(\cD)\to \Set$$
is representable. Let $j\in \cJ$. The functor $\Hom_{\cC}(X_j,F(-)):\cD\to \Set$ is small (since it is accessible), and commutes with finite limits. Let $Y^j$ denote the corresponding object in $\Pro(\cD)$. So $Y^j\in\Pro(\cD)$, and for every $d\in\cD$ we have
$$\Hom_{\cC}(X_j,F(d))=\Hom_{\Pro(\cD)}(Y^j,d).$$
Let $j\to j'$ be a morphism in $\cJ$. Then we have a morphism in $\cC$ $X_j\to X_{j'}$. This morphism induces a natural transformation (a morphism in $(\Set^{\cD})_{sm}$) of the form $\Hom_{\cC}(X_{j'},F(-))\to \Hom_{\cC}(X_j,F(-))$, which gives a morphism in $\Pro(\cD)$ of the form $Y^j\to Y^{j'}$. Thus we get a \emph{functor}
$j\to Y^j:\cJ\to \Pro(\cD)$. The category $\Pro(\cD)$ is complete (see Remark \ref{r:complete}), so we can define $Y:=\lim_{j\in \cJ}Y^j\in \Pro(\cD)$.

Let $\{Z_k\}_K\in \Pro(\cD)$. Then we have
$$\Hom_{\Pro(\cC)}(\{X_j\}_\cJ,\{F(Z_k)\}_K)=\lim_{k\in K}\colim_{j\in \cJ}\Hom_{\cC}(X_j,F(Z_k))=$$
$$=\lim_{k\in K}\colim_{j\in \cJ}\Hom_{\Pro(\cD)}(Y^j,Z_k).$$
It was shown in \cite{IsaL} that every simple object in $\Pro(\cD)$ is finitely copresentable (see \cite{AR}). Thus, for every $k\in K$ we have
$$\Hom_{\Pro(\cD)}(\lim_{j\in \cJ}Y^j,Z_k)=\colim_{j\in \cJ}\Hom_{\Pro(\cD)}(Y^j,Z_k).$$
To conclude, we have
$$\Hom_{\Pro(\cC)}(\{X_j\}_\cJ,\{F(Z_k)\}_K)=\lim_{k\in K}\Hom_{\Pro(\cD)}(\lim_{j\in \cJ}Y^j,Z_k)=$$
$$=\Hom_{\Pro(\cD)}(Y,\{Z_k\}_K).$$
\item Assume that $F$ has a left adjoint $G$. We need to prove that $\Pro(G)$ is left adjoint to $\Pro(F)$. Let  $c = \{c_i\}_\cI \in \Pro(\cC)$, and $d  = \{d_j\}_\cJ \in \Pro(\cD) $. We have
$$\Hom_{\Pro(\cD)}(\Pro(G)(c),d)  = \lim\limits_{j\in \cJ} \colim\limits_{i\in \cI} \Hom_\cD(G(c_i),d_j) =  $$
$$ = \lim\limits_{j\in \cJ} \colim\limits_{i\in \cI} \Hom_\cC(c_i,F(d_j)) = \Hom_{\Pro(\cC)}(c,\Pro(F)(d)). $$
\end{enumerate}
\end{proof}

\section{The Homotopy Category of a Weak Fibration Category}\label{s:cofib}
\subsection{The natural functor $i:\Ho(\cC) \to \Ho(\Pro(\cC))$}\label{s:i}
Let $(\mcal{C},\mcal{W})$ be a relative category (see Definition \ref{d:rel}). Recall that the homotopy category $\Ho(\cC)$ is obtained from $\cC$ by formally inverting all maps of  $\cW$.  In other words, $\Ho(\cC)$ has the same objects as $\cC$ and its maps are obtained  from the composable words in the maps of $\cC$ and the formal inverses of the maps of $\cW$, by means of the obvious equivalence relation. Composition of morphisms is defined by concatenation of words (for a more detailed definition see, for example, \cite[Definition 1.2.1]{Hov}).
There is a natural functor $\cC\to\Ho(\cC)$ which is the identity on objects. This functor is initial among functors from $\cC$ that invert all morphisms in $\cW$.

Now let $(\cC,\cW,\cF)$ be a homotopically small pro-admissible weak fibration category. Consider the natural full subcategory inclusion $\cC \to \Pro(\cC)$. This inclusion clearly transfers maps in $\cW$ to maps in $\Lw^{\cong}(\cW)$. Thus we have an induced functor $i:\Ho(\cC) \to \Ho(\Pro(\cC))$. The main purpose of this section is to prove that this last functor is fully faithful, that is, to prove the following:
\begin{prop}\label{p:ho_ff}
Let $(\cC,\cW,\cF)$ be a homotopically small pro-admissible weak fibration category. Then
for every $A,B\in \cC$, the natural functor $i:\Ho(\cC) \to \Ho(\Pro(\cC))$ induces  a  bijection
$$i: \Hom_{\Ho(\cC)}(A,B)\xrightarrow{\sim} \Hom_{\Ho(\Pro(\cC))}(A,B).$$
\end{prop}

The proof of Proposition \ref{p:ho_ff} will occupy the rest of this subsection.

Note that since every object $A \in \cC$ is weakly equivalent to its fibrant replacement in $\cC$ (and thus also in $\Pro(\cC)$), it is enough to prove Proportion \ref{p:ho_ff} for fibrant $A,B$.


Let $(\mcal{C},\mcal{W},\mcal{F})$ be any weak fibration category. We denote by $\cC_f$ the full subcategory of $\cC$ spanned by the fibrant objects. We first prove a proposition that  simplifies the description of the homotopy category $\Ho(\cC)$, using the weak fibration structure. For convenience of notation we define $\cF\cW$ to be the class of acyclic fibrations in $\cC$.

\begin{prop}\label{p:zigzag}
Let $A,B\in\cC$ be fibrant. Then every element of $\Hom_{\Ho(\cC)}(A,B)$ can be represented by  a zigzag of the form
$$A \xleftarrow{\cW\cF} X \to B.$$
\end{prop}

\begin{rem}
Note that we cannot prove Proposition \ref{p:zigzag} by using a similar result known for a Brown category of fibrant objects, since we do not know a-priori that the functor $\Ho(\cC_f)\to \Ho(\cC)$ is fully faithful.
\end{rem}

\begin{proof}
First we can choose a representative for that morphism by a zigzag of the form
$$A \xleftarrow{\cW} X_1 \to Y_1 \xleftarrow{\cW} X_2 \to Y_2 \xleftarrow{\cW} \cdots \xleftarrow{\cW} X_{n} \to B.$$
By choosing fibrant replacements for all the $Y_i$, we can assume that $Y_i$ is fibrant for every $i$.

Our second step is to show that we can replace this last zigzag by a zigzag of the form
$$A \xleftarrow{\cF\cW} X_1 \to Y_1 \xleftarrow{\cF\cW} X_2 \to Y_2 \xleftarrow{\cF\cW} \cdots \xleftarrow{\cF\cW} X_{n} \to B $$
with $Y_i$ fibrant for all $i$.

It is enough to assume that $n=1$, that is, that our zigzag has the form
$A \xleftarrow{\cW} X \to B.$
Consider the induced map $X \to A\times B$ and choose some factorization
$$X \xrightarrow{\cW}  \tilde{X} \xrightarrow{\cF} A \times B.$$
It is now not hard to verify that we get a diagram (here we use the fact that $B$ is fibrant)
$$\xymatrix{
A & \ar[l]^{\cF\cW}\tilde{X}\ar[r] & B. \\
  &    X\ar[ul]^{\cW} \ar[u]^{\cW}\ar[ur]  &
}$$

We are thus left with a zigzag of the form
$$A \xleftarrow{\cF\cW} X_1 \to Y_1 \xleftarrow{\cF\cW} X_2 \to Y_2 \xleftarrow{\cF\cW} \cdots \xleftarrow{\cF\cW} X_{n} \to B $$
with $Y_i$ fibrant for all $i$. Our last step is to show that we can shrink the length of this zigzag until $n=1$.
By using induction it is enough assume that our element is  represented by  a zigzag of the form
$$A \xleftarrow{\cF\cW} X_1 \to Y_1 \xleftarrow{\cF\cW}  X_2 \to  B$$
with $Y_1$ fibrant.

We denote ${X}:= X_1\times_{Y_1} X_2$ and get a commutative diagram
$$
\xymatrix{
A& \ar[l]_{\cF\cW} X\ar[d]^{\cF\cW}\ar[r] & X_2\ar[d]^{\cF\cW} & X_2 \ar[d]^{=}\ar[l]_{=} \ar[r]& B .\\
 & X_1 \ar[ul]^{\cF\cW} \ar[r]               & Y_1                          & X_2 \ar[l]_{\cF\cW} \ar[ur]}
$$
Thus our element is represented by $A \xleftarrow{\cF\cW} X\to B$.
\end{proof}

Let us now assume that $\cC$ is homotopically small and pro-admissible. Then by Theorem \ref{t:model_big} there is an induced model structure on $\Pro(\cC)$.
From now until the end of this section we fix two fibrant objects $A,B \in \cC$.

Our construction of the model structure on $\Pro(\cC)$ enables us to compute explicitly a cofibrant replacement for $A$ in $\Pro(\cC)$. The factorizations into a cofibration followed by a weak equivalence were constructed in Section \ref{s:model} using Proposition \ref{p:factor_gen_big}.

Let $\cF_{A},\barr{\cF_{A}}$ be the categories of factorizations $\cF_{\phi\to A},\barr{\cF_{\phi\to A}}$ for the unique map $\phi\to A$ and the factorization category $(\cC,\cC,\cF\cW)$ (see Definition \ref{d:factorization_cat}). An object in either of these categories is an acyclic fibration
$\tilde{A} \xrightarrow{\cF\cW} A$, and morphisms are commutative diagrams of the form
$$\xymatrix{
\tilde{A} \ar[rr]\ar[dr]_{\cF\cW}& & \barr{A}\ar[dl]^{\cF\cW}, \\
  & A&
}$$
where in $\barr{\cF_A}$ the horizontal map is arbitrary and in ${\cF_A}$ it is an acyclic fibration.
There are obvious forgetful functors
$$\barr{\cF_A}\xrightarrow{\barr{U}}\cC,\:\:\cF_A\xrightarrow{U}\cC.$$
By the proof of Proposition \ref{p:factor_gen_0}, there exists a small cofinite directed set $\cA_A$ and a functor $j:\cA_A\to \cF_A$, that is pre-cofinal relative to $\barr{\cF_A}$.
By Proposition \ref{p:lw_factor_gen} any such $j$ and $\cA_A$ give rise to a cofibrant replacement of $A$ by
$$H_A:\cA_A\xrightarrow{j} \cF_A\xrightarrow{U}\cC.$$

\begin{define}\label{d:R}
Let $g$ be an element in $\Hom_{\Ho(\Pro(\cC))}(A,B)$.

Let $g':H_A\to B$ denote the composition of $g$ and the image of $H_A\to A$ in $\Ho(\Pro(\cC))$
$$\xymatrix{H_A\ar[r]^{\sim}\ar[dr]_{g'} & A\ar[d]^g\\
                    & B.}$$

Since $B$ is fibrant in $\cC$ it is also fibrant in $\Pro(\cC)$.
Since $\Pro(\cC)$ is a model category we have (see for example \cite[Chapter 1]{Hov})
$$\Ho(\Pro(\cC))(H_A,B)\cong\Pro(\cC)(H_A,B)/\sim,$$
where $\sim$ is the homotopy relation in $\Pro(\cC)$.
Thus there exists a map $\gamma:H_A\to B$ in $\Pro(\cC)$ such that $[\gamma]=g'$. It follows that $g:A\to B$  can be represented by a zigzag of the form
$$A \xleftarrow{\mathbf{W}} H_A \xrightarrow{\gamma} B.$$

By the definition of morphisms in $\Pro(\cC)$ there exists an object $a\in \cA_A$ and a morphism $H_A(a)\to B$ representing $\gamma$. By the definition of $H_A$ we have a commutative diagram
$$\xymatrix{
A & H_A \ar[l]_{\mathbf{W}}\ar[r]^{\gamma}\ar[d]  & B. \\
  & H_A(a)\ar[ur]_{}\ar[ul]^{\cF\cW\ni j(a)}&
}$$


We define $R(g)$ to be the element in $\Hom_{\Ho(\cC)}(A,B)$ represented by the zigzag
$A \xleftarrow{\cF\cW\ni j(a)} H_A(a) \xrightarrow{{}} B.$
\end{define}

We now wish to prove that the element $R(g) \in \Hom_{\Ho(\cC)}(A,B)$ is well defined (i.e. it does not depend on the choices of $\gamma$ or $H_A(a) \xrightarrow{{}} B$). For this we need the following:

\begin{define}
Let $B\in\cC_f$. A \emph{path object} for $B$ in $\cC$ is a factorization in $\cC$ of the diagonal map $\Delta_B:B\to B\times B$ into a weak equivalence followed by a fibration: $B\to B^I\to B\times B$.
(Note that we are not assuming any simplicial structure, $B^I$ is just a suggestive notation.)
\end{define}

Let $B\in\cC_f$, and let $B\xrightarrow{j} B^I\xrightarrow{(\pi_0,\pi_1)} B\times B$ be a path object for $B$ in $\cC$. We will sometimes denote this path object simply by $B^I$ suppressing the maps $j,\pi_0,\pi_1$. It is a standard verification  that  $\pi_0,\pi_1$ are acyclic fibrations. It is also a standard verification that in $\Ho(\cC)$ we have $[\pi_0] = [\pi_1]=[j]^{-1}$.

\begin{lem}
The element $R(g) \in \Hom_{\Ho(\cC)}(A,B)$ is well-defined (i.e. it does not depend on the choices of $\gamma$ or $H_A(a) \xrightarrow{{}} B$).

\end{lem}
\begin{proof}
First let us assume that we have chosen some $H_A(a') \xrightarrow{{}} B$ instead of $H_A(a) \xrightarrow{{}} B$. The morphisms $H_A(a) \xrightarrow{{}} B$ and $H_A(a') \xrightarrow{{}} B$ represent the same element in $\colim_{a\in \cA_A}\Hom_{\cC}(H_A(a),B)$. It follows that there exists $a''$ in $\cA_A$ such that $a''\geq a, a'$ and the following diagram commutes:
$$\xymatrix{ &   H_A(a)\ar[dr]^{{}}     & \\
H_A(a'')\ar[ur]^{H_A(a''\to a)}\ar[dr]_{H_A(a''\to a')} & & B. \\
             &   H_A(a')\ar[ur]_{{}}   & }$$

We define $H_A(a'')\to B$ to be the morphism given in the diagram above. We thus have a commutative diagram
$$\xymatrix{
  & H_A(a)\ar[dl]_{\cF\cW\ni j(a)} \ar[dr] &                       \\
A & H_A(a'') \ar[l]^{j(a'')}_{\cF\cW}\ar[u] \ar[d]\ar[r]& B,\\
  & H_A(a')\ar[ul]^{\cF\cW\ni j(a')} \ar[ur]&
}$$
and indeed $A \xleftarrow{\cF\cW\ni j(a)} H_A(a) \xrightarrow{{}} B$ and $A \xleftarrow{\cF\cW\ni j(a')} H_A(a') \xrightarrow{{}} B$ represent the same element in $\Hom_{\Ho(\cC)}(A,B)$.

Now assume we have chosen some other map $\gamma':H_A\to B$ in $\Pro(\cC)$ such that $[\gamma']=g'$.

Let $B\xrightarrow{j} B^I\xrightarrow{(\pi_0,\pi_1)} B\times B$ be a path object for $B$ in $\cC$.
Since $B$ is fibrant, $H_A$ is cofibrant and we have that
$$g'=[\gamma]=[\gamma']\in \Pro(\cC)(H_A,B)/\sim,$$
there exists some  homotopy $H:H_A \to B^I$ such that
$$\gamma = \pi_0\circ H ,\:\:\gamma' = \pi_1\circ H.$$

We can choose some object $a\in \cA_A$ and a morphism $G:H_A(a)\to B^I$ representing $H$. By the definition of $H_A$ we have a commutative diagram
$$\xymatrix{
A & H_A \ar[l]\ar[r]^{H}\ar[d]  & B^I. \\
  & H_A(a)\ar[ur]_{G}\ar[ul]^{\cF\cW\ni j(a)}&
}$$
It is thus enough to show that
$$A \xleftarrow{\cF\cW\ni j(a)} H_A(a) \xrightarrow{\pi_0\circ G} B,\:\: A \xleftarrow{\cF\cW\ni j(a)} H_A(a) \xrightarrow{\pi_1\circ G} B$$
represent the same element in $\Hom_{\Ho(\cC)}(A,B),$ which is true since $[\pi_0] = [\pi_1]$ in $\Ho(\cC)$.
\end{proof}

We thus have a well-defined map
$$R: \Hom_{\Ho(\Pro(\cC))}(A,B)\to \Hom_{\Ho(\cC)}(A,B).$$
In order to finish the proof of Proposition \ref{p:ho_ff} it is enough to show that $R$ is a two-sided inverse to the map induced by the natural functor
$$i:\Hom_{\Ho(\cC)}(A,B)\to \Hom_{\Ho(\Pro(\cC))}(A,B).$$
This is done in the following two lemmas where we show that $R$ is a right inverse to $i$ and that $R$ is surjective.
\begin{lem}
We have $i\circ R=\id$.
\end{lem}
\begin{proof}
Let $g$ be an element in $\Hom_{\Ho(\Pro(\cC))}(A,B)$, represented by a zigzag of the form
$$A \xleftarrow{\mathbf{W}} H_A \xrightarrow{\gamma} B,$$
and let $H_A(a)\to B$ be a morphism representing $\gamma$, as in Definition \ref{d:R}.
Then we have a commutative diagram
$$\xymatrix{
A & H_A \ar[l]_{\mathbf{W}}\ar[r]^{\gamma}\ar[d]  & B. \\
  & H_A(a)\ar[ur]_{}\ar[ul]^{\cF\cW\ni j(a)}&
}$$
By definition $R(g)$ is the element in $\Hom_{\Ho(\cC)}(A,B)$ represented by the following zigzag:
$A \xleftarrow{\cF\cW\ni j(a)} H_A(a) \xrightarrow{{}} B,$
and $i(R(g))$ is the element in $\Hom_{\Ho(\Pro(\cC))}(A,B)$ represented by the same zigzag. The diagram above shows that
$i(R(g))$ and $g$ represent the same element in $\Hom_{\Ho(\Pro(\cC))}(A,B)$.
\end{proof}
\begin{lem}\label{l:R_surj}
The map $R$ is surjective.
\end{lem}
\begin{proof}
Consider an element in $\Hom_{\Ho(\cC)}(A,B)$ represented by  a zigzag of the form $$A \xleftarrow{\cW\cF} X \xrightarrow{f} B$$
(see Proposition \ref{p:zigzag}).
The category $\cF_A$ is semi-cofiltered (see Definition \ref{d:cofiltered}) and $j:\cA_A\to \cF_A$ is a pre-cofinal functor relative to $\barr{\cF_A}$, so by Lemma \ref{l:onto is cofiltered} the functor $j:\cA_A\to\barr{\cF_A}$ is semi-cofinal.

The arrow  $i:X \xrightarrow{\cW\cF} A$ is an object in $\barr{\cF_A}$.
The functor $j:\cA_A\to\barr{\cF_A}$ is semi-cofinal (see Definition \ref{d:cofinal}) so the over-category $j_{/i}$ is nonempty. Let $j(a)\to i$ be an object in $j_{/i}$.
Then we have a commutative diagram
$$\xymatrix{H_A(a)\ar[r]\ar[dr]_{\cF\cW\ni j(a)} & X\ar[d]_i^{\cF\cW}\\
                                    & A.}$$
The above commutative triangle induces the following commutative diagram:
$$\xymatrix{ & H_A\ar[d]\ar[dl]_{\mathbf{W}}\ar[dr]^{\gamma} & \\
A & H_A(a)\ar[l]_{\cF\cW}\ar[r]\ar[d] & B \\
  & X\ar[ur]_{f}\ar[ul]_i^{\cF\cW}&
}$$
Let $g$ be the object in $\Hom_{\Ho(\Pro(\cC))}(A,B)$ represented by the following zigzag:
$$A \xleftarrow{\mathbf{W}} H_A \xrightarrow{\gamma} B.$$
Then by the definition of $R$ (see Definition \ref{d:R}) we have that
$$R(g) =[A \xleftarrow{\cF\cW} H_A(a) \xrightarrow{{}} B]=[A \xleftarrow{\cW\cF} X \xrightarrow{f} B]$$
in $\Hom_{\Ho(\cC)}(A,B)$.
\end{proof}

This finishes our proof of Proposition \ref{p:ho_ff}.

We now present a small application of Proposition \ref{p:ho_ff} that will be used later. Let $\cC$ be a weak fibration category, let $A,B\in\cC_f$, and let $B^I$ be a path object for $B$ in $\cC$.

\begin{define}\label{d:homotopy}
Let $f,g:A\to B$ be two maps in $\cC_f$. We say that $f$ is \emph{strictly homotopic} to $g$ (relative to the path object $B^I$), if there exists a map $H:A\to B^I$ such that $\pi_0 H=f$ and $\pi_{1} H=g$. We denote by $\sim_{B^I}$ the equivalence relation on $\cC_f(A,B)$ generated by strict homotopy. In other words, two maps $f,g:A\to B$ in $\cC_f$ are called \emph{homotopic} (relative to $B^I$), denoted $f\sim_{B^I} g$, if they can be related by a zigzag of strict homotopies (relative to $B^I$). If the choice of $B^I$ is clear from the context, we will write  $\sim $ instead of $\sim_{B^I} $.
\end{define}

The homotopy relation is an equivalence relation on $\cC_f(A,B) = \cC(A,B)$, for every $A,B\in \cC_f$.
We denote the set of equivalence  classes by   ${\cC}(A,B)/\sim_{B^I}$.

Now assume that $\cC$ is homotopically small and pro-admissible, so that we have an induced model structure on $\Pro(\cC)$. Since $H_A\in\Pro(\cC)$ is cofibrant and $B\in\Pro(\cC)$ is fibrant,  we get a coequalizer diagram
$$\Hom_{\Pro(\cC)}(H_A,B^I)  \rightrightarrows \Hom_{\Pro(\cC)}(H_A,B) \to \Hom_{\Ho(\Pro(\cC))}(A,B).$$
By the definition of morphism sets  in $\Pro(\cC)$ and using Proposition \ref{p:ho_ff} and the fact that direct limits commute with each other, we get
$$\Hom_{\Ho(\cC)}(A,B)\cong\Hom_{\Ho(\Pro(\cC))}(A,B) \cong \colim \limits_{a\in \mcal{A}_A^{\op}} \Hom_{\cC}(H_A(a),B)/\sim_{B^I}.$$

\subsection{Functorial path objects}\label{s:Ho}
In trying to understand morphisms in $\Ho(\cC)$ and $\Ho(\Pro(\cC))$ we came across the construction of a path object $B^I$ for a
fibrant object $B\in \cC$. In the proof of Proposition \ref{p:ho_ff} we only needed to choose such an arbitrary path object. However, in many applications there is in fact a functorial construction of such path objects.
In this case morphisms in $\Ho(\cC)$ can be given a nice formula.
Our approach is influenced from both \cite{Bro} and \cite{AM}.
The results obtained in this section will be used later in Section \ref{s:etale_big}, to connect the theory presented here with the approach taken by \cite{AM}.

We begin with a definition:
\begin{define}\label{d:path_object}
Let $(\mcal{C},\mcal{W},\mcal{F})$ be a weak fibration category. Recall that $\cC_f$ denotes the full subcategory of $\cC$ spanned by the fibrant objects.
A functorial path object in $\cC$ is a functor $P:\cC_f\to \cC^{[2]}$, where $[2]$ is the ordinal $\{0,1,2\}$, such that:
\begin{enumerate}
    \item For every $C\in\cC_f$, $P(C)(0)\to P(C)(1)\to P(C)(2)$ is a path object for $C$ in $\cC$.
    \item For every map $f:C\to D$ in $\cC_f$, we have $P(f)(0)=f:C\to D$, and $P(f)(2)=f\times f:C\times C\to D\times D$.
\end{enumerate}
\end{define}

From now until the end of this section we fix a homotopically small pro-admissible weak fibration category $(\mcal{C},\mcal{W},\mcal{F})$. We further assume that $\cC$ has an initial object $\phi$.

Let $A,B$ be objects in $\cC$. By Lemma \ref{l:hom small mor}, we can choose an essentially small dense full sub weak fibration category $\cC_s\subseteq\cC$ containing $A$ and $\phi$. We denote by $\cF^s_{A}$ and $\barr{\cF^s_{A}}$ the categories of factorizations $\cF_{\phi\to A}$ and $\barr{\cF_{\phi\to A}}$ respectively, defined in Definition \ref{d:factorization_cat} for the unique map $\phi\to A$ and the factorization category $(\cC_s,\cC_s,\cF_s\cap\cW_s)$. By the proof of Proposition \ref{p:factor_gen_0}, there exists a small cofinite directed set $\cA_A$, and a pre-cofinal functor $j:\cA_A\to \cF^s_A$.
By Proposition \ref{p:lw_factor_gen} and Lemma \ref{l:dense} any such $j$ and $\cA_A$ give rise to a cofibrant replacement of $A$ by
$$H_A:\cA_A\xrightarrow{j} \cF^s_A\xrightarrow{U}\cC.$$
At the end of Section \ref{s:i} we have shown that if $A,B\in \cC$ are fibrant, then
$$\Hom_{\Ho(\cC)}(A,B)\cong\Hom_{\Ho(\Pro(\cC))}(A,B) \cong\colim \limits_{a\in \mcal{A}_A^{\op}} \Hom_{\cC}(H_A(a),B)/\sim_{B^I} .$$

This formula is already quite nice, but it is still somewhat complicated as $\mcal{A}_A$ is a complicated category. Now, assume there exists
a functorial path object $P$ in $\cC_s$. Then as we will show, this functorial path object induces a much simpler cofiltered category, which we denote by $\pi((\cC_s)_{/A})_{fw}$. (More details explaining this notation will come later, see Definition \ref{d:over}.) We further show that the functor
$$F_{A,B}:\cA_A\to{\Set} $$
given on objects by
$$a\mapsto  \Hom_{\cC}(H_A(a),B)/\sim_{P}$$
decomposes as
$$\mcal{A} \xrightarrow{\rho} \pi((\cC_s)_{/A})_{fw} \xrightarrow{G_{A,B}}{\Set},$$
with $\rho:\mcal{A} \to \pi((\cC_s)_{/A})_{fw}$ cofinal. Thus we get

\begin{prop}\label{p:Ho}
Let $(\mcal{C},\mcal{W},\mcal{F})$ be a homotopically small pro-admissible weak fibration category, and let $A,B$ be fibrant objects in $\cC$. Choose a small dense full sub weak fibration category $\cC_s\subseteq\cC$ containing $A$ and $\phi$, and assume that $\cC_s$ has a functorial path object $P$. Then $P$ induces a cofiltered category $\pi((\cC_s)_{/A})_{fw}$ (see Definition \ref{d:over}) such that there exists an isomorphism
$$\Hom_{\Ho(\cC)}(A,B) \cong \colim \limits_{(A'\to  A)\in \pi((\cC_s)_{/A})_{fw}^{\op}}\Hom_{\cC}(A',B)/\sim_{P}.$$
\end{prop}

\begin{rem}\
\begin{enumerate}
\item Proposition \ref{p:Ho} gives a much simpler formula for the Hom-sets in $\Ho(\cC)$ then the one given at the end of Section \ref{s:i}. The main reason for this is that $\pi((\cC_s)_{/A})_{fw}$ is a much simpler cofiltered category than $\mcal{A}_A^{\op}$. To see this more clearly, note that both categories are constructed from $((\cC_s)_{/A})_{fw}\cong \cF_A^s$, one through passing to the homotopy category, and one through constructing a pre-cofinal functor $j:\cA_A\to \cF_A^s$.
   By the construction of $\cA_A$ given in Proposition \ref{p:factor_gen_0} we have at level zero $\cA^0_A=\Ob(\cF_A^s)$, and the map $j$ is just the identity on objects. So at level zero alone, this map is onto on objects. At the next levels of $\cA_A$ many repetitions of elements from $\Ob(\cF_A^s)$ occur. It can happen that an object of $\cF_A^s$ will have infinitely many preimages in $\cA_A$ under $j$. On the other hand the objects of $\pi((\cC_s)_{/A})_{fw}$ are the same as the objects of $((\cC_s)_{/A})_{fw}$, the difference is that homotopic maps in $((\cC_s)_{/A})_{fw}$ are identified. This particular simple form is used in Example \ref{e:Verdier} to give a simple proof of Verdier's hypercovering theorem \cite{SGA4-I}.
\item Proposition \ref{p:Ho} is reminiscent of, but not identical to, Brown's famous result describing the morphism sets in the homotopy category of a category of fibrant objects (see \cite[Theorem 2.1]{Bro}). First of all, Brown ignores set theoretical issues, so he does not need to restrict to a small dense subcategory $\cC_s$ instead of $\cC$. Second, since Brown does not assume the existence of a functorial path object, he uses a slightly different category than our $\pi((\cC_s)_{/A})_{fw}$. The method of proof is also different. While Brown's proof is based on the Gabriel-Zisman localization theory for a category satisfying a calculus of right fractions (see \cite{GZ}), our proof is based on the induced model structure on $\Pro(\cC)$. Our proof also reveals a somewhat conceptual explanation for the appearance of the colimit in the formula obtained, namely, it comes from the cofibrant replacement $H_A$ to $A$.
\end{enumerate}
\end{rem}

The proof of Proposition \ref{p:Ho} will occupy the rest of this section.

We will only be using the functorial path object $P$.
Thus, for example, when we say that two maps $f,g:C\to D$ in $(\cC_s)_f$ are (strictly) homotopic, we will mean that they are (strictly) homotopic
relative to $P(D)$ (see Definition \ref{d:homotopy}). For every $C\in (\cC_s)_f$ we write
$$C\xrightarrow{j} C^I\xrightarrow{(\pi_0,\pi_1)} C\times C$$
instead of
$$P(C)(0)\to P(C)(1)\to P(C)(2).$$

\begin{lem}
Let $f,g:B\to C$, $u:A\to B$ and $v:C\to D$ be maps in $(\cC_s)_f$, and suppose that $f\sim g$. Then $fu\sim gu$ and $vf\sim vg$.
\end{lem}
\begin{proof}
We can assume that $f$ is strictly homotopic to $g$. Then the proof is a rather straightforward verification. Note, however, that we need to use the functoriality of our path object in the proof of the second equivalence.
\end{proof}

\begin{cor}
There exists a \emph{category} $\pi (\cC_s)_f$ with:
\begin{enumerate}
\item $\Ob(\pi (\cC_s)_f)=\Ob((\cC_s)_f)$.
\item For every $C,D\in \Ob((\cC_s)_f)$, $\pi(\cC_s)_f(C,D):=(\cC_s)_f(C,D)/\sim.$
\item Composition and identities in $\pi(\cC_s)_f$ are induced from those in $(\cC_s)_f$.
\end{enumerate}
\end{cor}
%

Recall that an object $C\in\cC_s$ is called fibrant or contractible if the unique map $C\to *$ is a fibration or a weak equivalence, respectively.

\begin{define}\label{d:pi_C}
\begin{enumerate}
\item Let ${\pi(\cC_s)_{fw}}$ denote the full subcategory of $\pi(\cC_s)_f$ spanned by the (fibrant and) contractible objects.
\item Let $\widehat{\pi(\cC_s)_{fw}}$ denote the category with:
\begin{enumerate}
\item $\Ob(\widehat{\pi(\cC_s)_{fw}}):=\Ob(\pi(\cC_s)_{fw})$.
\item For every $C,D\in \Ob(\widehat{\pi(\cC_s)_{fw}})$, $$\widehat{\pi(\cC_s)_{fw}}(C,D):=\{[h]\in\pi(\cC_s)_{f}(C,D)|h\in \cF\cap\cW\}.$$
\item Composition and identities in $\widehat{\pi(\cC_s)_{fw}}$ are the same as in $\pi(\cC_s)_{f}$.
\end{enumerate}
\end{enumerate}
\end{define}

\begin{lem}\label{l:HC_cofiltered}
The categories $\pi(\cC_s)_{fw}$ and $\widehat{\pi(\cC_s)_{fw}}$ are cofiltered.
\end{lem}
\begin{proof}
We show this for $\widehat{\pi(\cC_s)_{fw}}$, and the proof for $\pi(\cC_s)_{fw}$ is identical.

Let $C,D$ be fibrant and contractible objects in $\cC_s$. We must show that there exists a fibrant and contractible object $E\in \cC_s$, and morphisms $[l]\in\pi(\cC_s)_f(E,C),[k]\in\pi(\cC_s)_f(E,D)$, such that $l,k\in \cF\cap\cW$. We can simply take $E$ to be the following pullback in $(\cC_s)$:
\[
\xymatrix{E\ar[d]^k\ar[r]^l & C\ar[d]\\
D\ar[r] & \ast.}
\]

Let $C,D$ be fibrant and contractible objects in $\cC_s$, and let $[l],[k]\in\pi(\cC_s)_f(C,D)$, such that $l,k\in \cF\cap\cW$. We must show that there exists a fibrant and contractible object $E\in \cC_s$ and a morphism $[t]\in\pi(\cC_s)_f(E,C)$ such that $t\in \cF\cap\cW,[l][t]=[k][t]$.

Note that, since $D$ is fibrant and contractible, $D\times D$ is also fibrant and contractible, since we have a pullback square
\[
\xymatrix{D\times D\ar[d]^{\pi_0}\ar[r]^{\pi_1} & D\ar[d]\\
D\ar[r] & \ast.}
\]
It follows that the diagonal map $D\to D\times D$ is a weak equivalence. By the two out of three property in $\cC_s$, we get that $D^I\xrightarrow{(\pi_0,\pi_1)} D\times D$ is a weak equivalence (and a fibration).
We can thus take $E$ to be the following pullback in $\cC_s$:
\[
\xymatrix{E\ar[d]^H\ar[r]^t & C\ar[d]^{(l,k)}\\
D^I\ar[r]^{(\pi_0,\pi_1)} & D\times D.}
\]
Clearly $H$ is a strict homotopy from $lt$ to $kt$.
\end{proof}

Obviously $\widehat{\pi(\cC_s)_{fw}}$ is a subcategory of $\pi(\cC_s)_{fw}$, that contains all the objects. Let $i:\widehat{\pi(\cC_s)_{fw}}\to \pi(\cC_s)_{fw}$ denote the inclusion functor.

\begin{lem}\label{l:HC_cofinal}
The functor $i:\widehat{\pi(\cC_s)_{fw}}\to \pi(\cC_s)_{fw}$ is cofinal.
\end{lem}
\begin{proof}
Let $Q\in \cC_s$ be a fibrant and contractible object. By Definition \ref{d:cofinal} it is enough to show that the over-category $i_{/Q}$ is nonempty and connected. It is nonempty since it contains $[\id_Q]$. Let $[f]\in \pi(\cC_s)_{fw}(P,Q)$ and $[g]\in \pi(\cC_s)_{fw}(R,Q)$. It is enough to show that there exists $[l]\in \widehat{\pi(\cC_s)_{fw}}(\cT,P)$ and $[k]\in \widehat{\pi(\cC_s)_{fw}}(\cT,R)$ such that $[f][l]=[g][k]$. As in the proof of Lemma \ref{l:HC_cofiltered}, it can be shown that $Q^I\xrightarrow{(\pi_0,\pi_1)} Q\times Q$ is an acyclic fibration. Also, by considering the pullback square
\[
\xymatrix{P\times R\ar[d]^{\pi_0}\ar[r]^{\pi_1} & R\ar[d]\\
P\ar[r] & \ast,}
\]
we see that $\pi_0:P\times R\to P,\pi_1:P\times R\to R$ are acyclic fibrations.
We can thus take $\cT$ to be the following pullback in $\cC_s$:
\[
\xymatrix{\cT\ar[d]^H\ar[r]^{(l,k)} & P\times R\ar[d]^{f\times g}\\
Q^I\ar[r]^{(\pi_0,\pi_1)} & Q\times Q.}
\]
Clearly $H$ is a strict homotopy from $fl$ to $gk$.
\end{proof}

Note that the over-category $(\cC_s)_{/A}$ is a weak fibration category, where a map in $(\cC_s)_{/A}$ is defined to be a fibration or a weak equivalence if it is so under the forgetful functor $U:(\cC_s)_{/A}\to \cC_s$. Since the forgetful functor commutes with pullbacks, the axioms are easily verified.

\begin{define}\label{d:over}
We define a functorial path object in $(\cC_s)_{/A}$. Let $(B \xrightarrow{g} A)\in ((\cC_s)_{/A})_f$. We define $P(B\xrightarrow{g} A)\in ((\cC_s)_{/A})^{[2]}$ by
$$\xymatrix{B\ar[r]^{(g,j)}\ar[dr]^g & A\times_{A^I} B^I \ar[r]^{(\pi_0,\pi_1)}\ar[d] & B\times_A B,\ar[dl]\\
                & A & }$$
where $\pi_i$ is defined to be the composition $A\times_{A^I} B^I\xrightarrow{} B^I \xrightarrow{\pi_i} B$.

Using this functorial path object we can define the categories $\pi((\cC_s)_{/A})_{fw}$ and $\widehat{\pi((\cC_s)_{/A})_{fw}}$ as in Definition \ref{d:pi_C}.
\end{define}
As we have shown, the categories $\pi((\cC_s)_{/A})_{fw}$ and $\widehat{\pi((\cC_s)_{/A})_{fw}}$ are cofiltered, and the natural subcategory inclusion $i:\widehat{\pi((\cC_s)_{/A})_{fw}}\hookrightarrow \pi((\cC_s)_{/A})_{fw}$ is cofinal.

We denote by ${((\cC_s)_{/A})_{fw}}$ the full subcategory of $(\cC_s)_{/A}$ spanned by the fibrant and contractible objects.
Let $\widehat{((\cC_s)_{/A})_{fw}}$ denote the category with the same objects as ${((\cC_s)_{/A})_{fw}}$ and for every $C,D\in \Ob(\widehat{((\cC_s)_{/A})_{fw}})$, $$\widehat{((\cC_s)_{/A})_{fw}}(C,D):=\{h\in((\cC_s)_{/A})(C,D)|h\in \cF\cap\cW\}.$$
Clearly there is a natural map $\widehat{((\cC_s)_{/A})_{fw}}\to\widehat{\pi((\cC_s)_{/A})_{fw}}$. Since this map is onto objects and morphisms, it is clearly pre-cofinal.

Note that the categories $\widehat{((\cC_s)_{/A})_{fw}}$ and ${((\cC_s)_{/A})_{fw}}$ defined above are just the categories $\cF^s_{A}$ and $\barr{\cF^s_{A}}$ respectively, considered after Definition \ref{d:path_object}. Recall also that we have a small cofinite directed set $\cA_A$, and a pre-cofinal functor $j:\cA_A\to \widehat{((\cC_s)_{/A})_{fw}}$. By Lemma \ref{l:compose_onto_big} the composition $r:\cA_A\to \widehat{((\cC_s)_{/A})_{fw}}\to \widehat{\pi((\cC_s)_{/A})_{fw}}$ is pre-cofinal. Since the categories $\cA_A$ and $\widehat{\pi((\cC_s)_{/A})_{fw}}$ are cofiltered, it follows from Lemma \ref{l:onto is cofiltered} that  $r$ is  cofinal.
The inclusion functor $i:\widehat{\pi((\cC_s)_{/A})_{fw}}\to \pi((\cC_s)_{/A})_{fw}$ is also cofinal by Lemma \ref{l:HC_cofinal}, so it follows that the composition $\rho:=ir:\cA_A \to \pi((\cC_s)_{/A})_{fw}$ is cofinal.

Recall that we have a cofibrant replacement of $A$ given by
$$H_A:\cA_A\xrightarrow{j} \widehat{((\cC_s)_{/A})_{fw}}\xrightarrow{U}\cC.$$
Clearly the functor
$$F_{A,B}:\cA_A\to{\Set} $$
given on objects by
$$a\mapsto  \Hom_{\cC_s}(H_A(a),B)/\sim_{B^I}$$
decomposes as
$$\mcal{A} \xrightarrow{\rho} \pi((\cC_s)_{/A})_{fw} \xrightarrow{G_{A,B}} {\Set},$$
where $G_{A,B}$ is defined by
$$G_{A,B}(A'\to  A):=\Hom_{\cC_s}(A',B)/\sim_{B^I}=\Hom_{\pi(\cC_s)_f}(A',B).$$
This finishes the proof of Proposition \ref{p:Ho}.

\section{Simplicial Presheaves as a Weak Fibration Category}\label{s:SPS}

In this section we present the application of the theory presented so far which was our original motivation for its development.

Let $\cC=(\cC,\tau)$ be a small Grothendieck site, and let $\SPS(\cC):= {\cS}^{\cC^{\op}}$ denote the category of simplicial presheaves on $\cC$.
In \cite{Jar}, Jardine defines the notions of combinatorial weak equivalences  and local fibrations in $\SPS(\cC)$. In the same paper Jardine defines a model structure on $\SPS(\cC)$. However, the local fibrations are \textbf{not} the fibrations in this model structure. Jardine (in \cite{Jar}) proves almost all that is needed to show that combinatorial weak equivalences and local fibrations give rise to a weak fibration category structure on $\SPS(\cC)$ (without considering this notion directly). In this section we complete the proof of this fact, and also review some of the definitions and proofs presented in \cite{Jar}, for the sake of completeness. We follow the common convention in the field, and call Jardine's combinatorial weak equivalences local weak equivalences (see \cite{DuIs,Jar1}).

We first present Jardine's original definition for a local weak equivalence.
Let $X$ be a simplicial set and let $1 \leq m$. Consider the set
$$\pi_m(X): =\bigsqcup \limits_{x\in X_0} \pi_m(X, x).$$
There is a canonical map $\pi_m(X) \to X_0$, and this map is functorial in $X$.
Thus, given a simplicial presheaf $X\in \SPS(\cC)$, we can define a map of presheaves $\pi_m(X) \to X_0$, for every $m \geq 1$.

Any simplicial presheaf map $f : X \to Y$ induces a presheaf morphism $\pi_0(X) \to \pi_0(Y)$, and a commutative diagram of presheaves
$$\xymatrix{
\pi_m(X) \ar[r]\ar[d] & \pi_m(Y) \ar[d]\\
X_0 \ar[r] &  Y_0.}
$$
For every $m \geq 0$, write $\widetilde{\pi_m}X$ for the sheaf associated to the presheaf $\pi_mX$. Now we can give the definition of a local weak equivalence.
\begin{define}
A map $f : X \to  Y$ of simplicial presheaves is called a \emph{local weak equivalence} if the following conditions hold:
\begin{enumerate}
\item The map $\widetilde{\pi_0}X \to \widetilde{\pi_0}Y$ is an isomorphism of sheaves.
\item The diagram
$$\xymatrix{
\widetilde{\pi}_m(X) \ar[r]\ar[d] & \widetilde{\pi}_m(Y) \ar[d]\\
\widetilde{X}_0 \ar[r] &  \widetilde{Y}_0}
$$
is a pullback diagram in $\Sh(\cC)$, for every $1\leq m $.
\end{enumerate}
We denote the class of local weak equivalences in $\SPS(\cC)$ by $\cW$.
\end{define}

\begin{rem}
Note that every levelwise weak equivalence in $\SPS(\cC)$ is a local weak equivalence, since
the two conditions are satisfied at the presheaf level, and hence also at the sheaf level. (Note that sheafification commutes with pullbacks.)
\end{rem}

\begin{define}
Let $f:A\to B$ be a map of simplical sets, and let $g:X\to Y$ be a map in $\SPS(\cC)$.
We say that $g$ has the \emph{local} right lifting property with respect to $f$, if for every $U\in \cC$, and every square of the form
$$\xymatrix{
A \ar[r] \ar[d]^f & X(U)\ar[d]^{g_U} \\
B \ar[r] & Y(U)
,}$$
there exists a covering sieve $R$ of $U$, such that for every $V\to U$ in $R$ there is a lift:
$$\xymatrix{
A \ar[r] \ar[d] & X(U) \ar[r]\ar[d] & X(V)\ar[d] \\
B \ar[r] \ar@{..>}[rru] & Y(U) \ar[r]   &  Y(V)
.}$$
In this case we shall denote $f \perp^{l} g$.
\end{define}

We now bring Jardine's definition of a local fibration (see \cite{Jar}).

\begin{define}\label{d:local}
Let $f:X\to Y$ be a map in $\SPS(\cC)$. We say that $f$ is a \emph{local fibration} if $f$ has the local right lifting property with respect to all inclusions of the form $\Lambda^n_k \to \Del^n$ $(n\geq 0,0\leq k\leq n)$.

We denote the class of local fibrations in $\SPS(\cC)$ by $\cF$.
\end{define}

\begin{rem}
Note that every levelwise fibration in $\SPS(\cC)$ is also a local fibration, since then we have the usual lifting property and, in particular, the local one.
\end{rem}

\begin{lem}[{\cite[Proposition 7.2]{DuIs}}]\label{l:FW}
The class of maps in $\cF \cap \cW$ is precisely the class of maps having the local right lifting property with respect to all inclusions  of the form $\partial\Delta^n\to \Delta^n$ $(n\geq 0)$. We call such maps \emph{local acyclic fibrations}.
\end{lem}

\begin{prop}\label{p:SPS_wfc}
$(\SPS(\cC),\cW,\cF)$ is a weak fibration category.
\end{prop}

\begin{proof}
$\SPS(\mcal{C})$ has all limits and colimits, and they are computed objectwise.
Since $\cF$ is defined by a local lifting property, it is easy to see that it is a subcategory that contains all the isomorphisms and is closed under base change. The same is true for $\cF \cap \cW$, by Lemma \ref{l:FW}. The fact that $\mcal{W}$ has the two out of three property and contains all the isomorphisms is also clear. Thus it remains to show the existence of factorizations. Consider a functorial factorization in simplicial sets into a weak equivalence followed by a fibration. Given a map $f:X\to Y$ in $\SPS(\cC)$, we can apply this functorial factorization levelwise, and obtain a factorization of $f$ in $\SPS(\cC)$
$$X \to Z \to Y,$$
where $X\to Z$ is a levelwise weak equivalence, and thus in $\cW$, and  $Z\to Y$ is a  levelwise fibration, and thus in $\cF$.
\end{proof}

Note that the weak fibration category $(\SPS(\cC),\cW,\cF)$ is naturally enriched over $\cS$.
For a simplicial presheaf $X\in \SPS(\cC)$, and a simplicial set $K\in \cS$, we define $K\otimes X ,X^K\in \SPS(\cC)$ levelwise. This makes $\SPS(\cC)$ tensored and cotensored over $\cS$. The following two lemmas are based partly on Corollary 7.4 in \cite{DuIs}, and can be shown by just unwinding the definitions:
\begin{lem}\label{l:lf}
Let $f:X \to Y$ be a map in $\SPS(\cC)$. Then $f$ is a local fibration iff for every map of the form $\Lambda^n_k \to \Del^n$ $(n\geq 0,0\leq k\leq n)$, the induced map
$$X^{\Delta^n}\to Y^{\Delta^n}\times_{Y^{\Lambda^n_k}} X^{\Lambda^n_k}$$ is a a local epimorphism in dimension 0.
\end{lem}

\begin{lem}\label{l:laf}
Let $f:X \to Y$ be a map in $\SPS(\cC)$. Then $f$ is a local acyclic fibration iff for every map of the form $\partial\Delta^n\to \Delta^n$ $(n\geq 0)$, the induced map $$X^{\Delta^n}\to Y^{\Delta^n}\times_{Y^{\partial\Delta^n}} X^{\partial\Delta^n}$$ is a local epimorphism in dimension 0.
\end{lem}

\begin{define}
Let $\lambda$ be a cardinal. We call a presheaf $F \in \PS(\cC)$ $\lambda$-\emph{bounded} if
$|F(c)| \leq \lambda$  for every $c\in \Ob (\cC)$. We denote by $\PS_{\lambda}(\cC)$ the full subcategory of  $\PS(\cC)$ spanned by the $\lambda$-bounded presheaves. Similarly we say that a simplicial presheaf $F \in \SPS(\cC)$ is $\lambda$-\emph{bounded} if $F_n$ is $\lambda$-bounded  for every $n \geq 0$, and  denote by $\SPS_{\lambda}(\cC)$ the full subcategory of $\SPS(\cC)$ spanned by the $\lambda$-bounded simplicial presheaves.
\end{define}

\begin{prop}\label{p:SPS_wfc_l}
Let $\lambda \geq \aleph_0$ be a cardinal. Then $(\SPS_{\lambda}(\cC),\cW,\cF)$ is a weak fibration category.
\end{prop}

\begin{proof}
It is clear that $\SPS_\lambda(\cC)$ is closed under finite limits, so the only thing to verify is that one can factor every morphism in $\SPS_\lambda(\cC)$ inside $\SPS_\lambda(\cC)$. Since the factorization described in the proof of Proposition \ref{p:SPS_wfc} can be constructed by applying the $\Ex^{\infty}$ functor, it is not hard to see that this is indeed the case.
\end{proof}

\subsection{Simplicial sheaves as a weak fibration category}
Let $\SSh(\cC):= \Sh(\cC)^{\Del^{\op}}$ denote the category of simplicial sheaves on $\cC$. Note that $\SSh(\cC)$ is just the full subcategory of $\SPS(\cC)$ spanned by the objects that satisfy the (usual) sheaf condition, since limits in $\SPS(\cC)$ are calculated levelwise.
It is a classical fact (see for example \cite{Jar}) that there is a functor $L:\PS(\cC) \to \PS(\cC)$, such that
$L^2$ is left adjoint to the inclusion $i:\Sh(\cC)\to \PS(\cC)$. The functor $L^2$ is called the sheafification functor. We can take these functors dimensionwise, and obtain a functor $L:\SPS(\cC) \to \SPS(\cC)$ and an adjunction
$$L^2: \SPS(\cC) \rightleftarrows \SSh(\cC): i.$$

\begin{define}
We say that a map in $\SSh(\cC)$ is a local weak equivalence (resp. local fibration) if it is a local weak equivalence  (resp. local fibration) as a map in $\SPS(\cC)$.
\end{define}
By abuse of notation we denote the class of local weak equivalences (respectively local fibrations) in $\SSh(\cC)$ also
by $\cW$ (respectively $\cF$).
\begin{prop}\label{p:SSh_wfc}
$(\SSh(\cC),\cW,\cF)$  is a weak fibration category.
\end{prop}
\begin{proof}
$\Sh(\cC)$ is a topos, and thus has all limits and colimits. It follows that $\SSh(\cC)= \Sh(\cC)^{\Del^{\op}}$ also has all limits and colimits, and they are computed levelwise.
Since the inclusion  $i: \SSh(\cC) \hookrightarrow \SPS(\cC)$ has a left adjoint,
it commutes with pullbacks. It is thus easy to see that $\cF,\cF \cap \cW$ are subcategories that contain all the isomorphisms and are closed under base change. The fact that $\mcal{W}$ has the two out of three property and contains all the isomorphisms is also clear. Thus it remains to show the existence of factorizations. Let $f:X\to Y$ be a map in $\SSh(\cC)$. We already proved that  in $\SPS(\cC)$ we have a factorization
$X \x{\cW}{\lrar} Z \x{\cF}{\lrar} Y$. Now consider the commutative diagram
$$ \xymatrix{
X \ar[d]^{\cong}\ar[r]^{\cW} \ar[rd]^f& Z \ar[d]^{\cF\cW} \ar[r]^{\cF} & Y\ar[d]^{\cong} \\
L^2(X) \ar[r] & L^2(Z)\ar[ru]^g \ar[r]^{\cF} & L^2(Y).
}$$
By \cite[Lemma 1.6]{Jar},  the middle vertical map is in $\cF\cW$,  and by \cite[Corollary 1.8]{Jar}, the map
$L^2(Z)\to L^2(Y)$ is in $\cF$.
Thus we get that $f$ is in $\cW$, and $g$ is in $\cF$.
\end{proof}

The category $\SSh(\cC)$ inherits an $\cS$ enriched structure as a full subcategory of $\SPS(\cC)$. For a simplicial sheaf $X\in \SSh(\cC)$ and a simplicial set $K\in \cS$, we can define $K\otimes X ,X^K$ as in $\SPS(\cC)$, and then take sheafification. This makes $\SSh(\cC)$ tensored and cotensored over $\cS$. It is not hard to check that Lemmas \ref{l:lf} and \ref{l:laf} remain valid, if we replace $\SPS(\cC)$ by $\SSh(\cC)$.

\begin{define}
Let $\lambda$ be a cardinal. We say that a (simplicial) sheaf is $\lambda$-\emph{bounded} if it is $\lambda$-bounded as a (simplicial) presheaf. We shall denote by $\SSh_{\lambda}(\cC)$ the full subcategory of $\SSh(\cC)$ spanned by the $\lambda$-bounded simplicial sheaves.
\end{define}

\begin{lem}\label{l:sheaf_small}
Let $\cC$ be a small site and let $\lambda \geq \max\{2^{|\Mor(\cC)|},\aleph_0\}$ be a cardinal. Then $L$ sends $\lambda$-bounded presheaves to $\lambda$-bounded presheaves.
\end{lem}
\begin{proof}
We have the formula
$$L(F)(U) = \colim_{R \in \cJ(U)^{\op}} \lim_{(V\to U)\in R} F(V),$$
where $U\in\cC$, $\cJ(U)$ is the (cofiltered) poset of covering sieves of $U$, considered as subfunctors of $Hon_{\cC}(-,U)$, and $R$ is a covering sieve of $U$, considered as a full subcategory of $\cC_{/U}$.
\end{proof}

\begin{prop}\label{p:SSh_wfc_l}
Let $\lambda  \geq \max\{2^{|\Mor(\cC)|},\aleph_0\}$ be a cardinal.
Then $(\SSh_{\lambda}(\cC),\cW,\cF)$ is a weak fibration category.
\end{prop}
\begin{proof}
By Lemma \ref{l:sheaf_small} and Proposition \ref{p:SPS_wfc_l}, the factorizations described in Proposition \ref{p:SSh_wfc} can be taken in
$\SSh_{\lambda}(\cC)$.
\end{proof}

\subsection{$\SPS(\cC)$ and $\SSh(\cC)$ are homotopically small}\label{ss:small}
In this subsection we will show that the weak fibration categories $\SPS(\cC)$ and $\SSh(\cC)$ are homotopically small (see Definition \ref{d:sub_weak}).
\begin{define}
Let $f:A\to B$ be a map of simplicial sets, and let $X\to Y\to Z$ be a pair of composable  maps in $\SPS(\cC)$.
We say that the diagram $X\to Y\to Z$ has the \emph{relative local} right lifting property with respect to $f$, if for every $U\in \cC$, and every diagram of the form
$$\xymatrix{
A \ar[r] \ar[d]^f & X(U)\ar[d]\\
B \ar[r]  & Z(U),
}$$
there exists a covering sieve $R$ of $U$, such that for every $V\to U$ in $R$, there is a lift:
$$\xymatrix{
A \ar[r] \ar[dd] & X(U) \ar[r]\ar[d] & X(V)\ar[d] \\
\empty & Y(U)\ar[d] \ar[r]   &  Y(V)\ar[d] \\
B \ar[r] \ar@{..>}[rru] & Z(U) \ar[r]   &  Z(V)
.}$$
In this case we denote
$$f \perp^{l} (X \to Y \to Z).$$
\end{define}

\begin{define}
Let $X\to Y\to Z$ be a pair of composable maps in $\SPS(\cC)$. We say that $X\to Y\to Z$ is a \emph{relative local fibration} if for every horn inclusion $$i_{n,i}:\Lambda^n_i \subset \Delta^n ,\:\: n \geq 1,\: 0 \leq i \leq n,$$
we have
$$i_{n,i} \perp^{l} (X \to Y \to Z).$$
\end{define}

\begin{define}
Let $X\to Y\to Z$ be a pair of composable maps in $\SPS(\cC)$. We say that $X\to Y\to Z$ is a \emph{relative local acyclic fibration} if for every boundary inclusion
$$i_{n}:\partial\Delta^n \subset \Delta^n , \:\:n \geq 0,$$
we have
$$i_{n} \perp^{l} (X \to Y \to Z).$$
\end{define}

\begin{lem}\label{l:ind_small}
Let $Y \in \SPS(\cC)$
and let
$$\xymatrix{
X_0 \ar[d]\ar[r] & X_1\ar[dl] \ar[r] & X_2\ar[dll] \ar[r] & \cdots \\
Y \\
}$$
be an ind-tower  over $Y$ in $\SPS(\cC)$,  such that for every $i\geq 0$, the diagram  $X_i \to X_{i+1} \to Y$ is a relative   local (acyclic) fibration. Then the map $\colim X_i \to Y$ is a local (acyclic) fibration.
\end{lem}
\begin{proof}
We prove the lemma for the case of a fibration, and the case of an acyclic fibration is similar.
It is enough to prove that
$$(\Lambda^n_i \to \Delta^n)\perp^l (\colim X_i \to Y),$$ for every
$ n \geq 1, 0 \leq i \leq n$.
Indeed consider a diagram  of the form
$$\xymatrix{
\Lambda_i^n \ar[r] \ar[d] & \colim X_i(U) \ar[d] \\
\Delta^n \ar[r] & Y(U).
}$$
Since  $\Lambda_i^n$ is a finitely presentable object (see \cite{AR}) in simplicial sets, we can factor  for some $j \geq 0$
$$\xymatrix{
\Lambda_i^n \ar[r] \ar[ddd] & X_j(U) \ar[d] \\
& X_{j+1}(U)\ar[d]\\
& \colim X_j(U)\ar[d] \\
\Delta^n \ar[r] & Y(U).
}$$

Now there exists some covering sieve $R$ for $U$, such that for every $V \to U$ in $R$ we have a lift:
$$\xymatrix{
\Lambda_i^n \ar[r] \ar[ddd] & X_j(U) \ar[d] \ar[r] & X_j(V) \ar[d]\\
& X_{j+1}(U)\ar[d]\ar[r] & X_{j+1}(V)\ar[d]\\
& \colim X_j(U)\ar[d] \ar[r] & \colim X_j(V)\ar[d]\\
\Delta^n \ar@{.>}[uurr] \ar[r] & Y(U)\ar[r] & Y(V).
}$$

Thus we get that required lift.
\end{proof}

\begin{lem}\label{l:next_small}
Let $\cC$ be a small site and $\lambda \geq \max\{|\Mor(\cC)|,\aleph_0\}$ a cardinal. Then every diagram
$$\xymatrix{
&Y_0\ar[d]\ar[dr]& \\
X \ar[r]\ar[ur]& Y \ar[r] & Z
}$$
in $\SPS(\cC)$ such that
\begin{enumerate}
\item $X,Y_0$ and $Z$ are $\lambda$-bounded,
\item the map $Y \to Z$ is a local (acyclic) fibration,
\end{enumerate}
can be extended into a diagram
$$
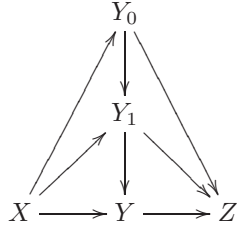
\xymatrix{
&Y_0\ar[d]\ar[ddr]& \\
& Y_1\ar[d]\ar[dr]& \\
X \ar[r]\ar[uur]\ar[ur] & Y \ar[r] & Z
}$$
in $\SPS(\cC)$ such that
\begin{enumerate}
\item The composition $Y_0\to Y_1\to Y$ is the map $Y_0\to Y$ in the original diagram.
\item $Y_1$ is $\lambda$-bounded,
\item the diagram  $Y_0\to Y_1 \to Z$ is a relative local (acyclic) fibration.
\end{enumerate}
\end{lem}
\begin{proof}
We prove the lemma for the case of a local acyclic  fibration, and the case of a local fibration is similar.
Let $U \in \cC$ be an object. Consider the set of all possible diagrams of the form

$$\xymatrix{
\partial \Delta^n \ar[r] \ar[d] & Y_0(U)\ar[d]\\
\Delta^n \ar[r] & Z(U),
}$$
for $n \geq 0$.
We denote this set by $D_U(Y_0 \to Z)$. It is easy to verify that since $Y_0$ and $Z$ are $\lambda$-bounded and $\lambda \geq\aleph_0$ we have
$$|D_U(Y_0\to Z)| \leq \lambda.$$
Since $Y \to Z$ is a local acyclic fibration, we have for every $d\in D_U:=D_U(Y_0\to Z)$ a covering sieve $R_d$ of $U$ such that for every $r:V\to U$ in $R_d$ we can complete that diagram $d$ to a diagram
$$\xymatrix{
\partial \Delta^n \ar[r] \ar[dd] & Y_0(U)\ar[d]\ar[r]& Y_0(V)\ar[d]\\
                & Y(U)\ar[d]\ar[r]& Y(V)\ar[d]\\
\Delta^n \ar[r]\ar[urr] & Z(U)\ar[r] & Z(V).
}$$
Thus, we get for every $U \in \cC$, $d \in D_U$ and $r \in R_d$ a diagram
$$\xymatrix{
\partial \Delta^n \times \h_V \ar[r] \ar[d] & Y_0\ar[d]\\
\Delta^n \times \h_V \ar[r]\ar[d] & Y \ar[d]\\
\Delta^n \times \h_U \ar[r]  &  Z,
}$$
where $\h_U$ denotes the representing presheaf of $U$.
We denote the set of all possible $U \in \cC$, $d \in D_U$ and $r \in R_d$ by $S$.
Note that since $|\Mor(\cC)| \leq \lambda$ we have that $\Ob( \cC)|\leq \lambda$ and $|R_d|\leq \lambda$ for every $d \in D_U$. Thus
$|S| \leq \lambda$.  Now consider the coproduct  parameterized by $S$:

$$\xymatrix{
\coprod_{S}\partial \Delta^n \times \h_V \ar[r] \ar[d] & Y_0\ar[d]\\
\coprod_{S}\Delta^n \times \h_V \ar[r] & Y.
}$$
 We shall take $Y_1$ to be the pushout of the upper left corner of the diagram above.
 Note that by definition it is clear that $Y_1$ fits in the diagram
$$\xymatrix{
&Y_0\ar[d]\ar[ddr]& \\
& Y_1\ar[d]\ar[dr]& \\
X \ar[r]\ar[uur]\ar[ur] & Y \ar[r] & Z.
}$$
To prove that $Y_1$ is $\lambda$-bounded, note that all the $\h_V$ are clearly $\lambda$-bounded and that $S$ is $\lambda$-bounded. It remains to show that $Y_0 \to Y_1 \to Z$ is a relative local acyclic fibration.
For this, let $U\in \cC$ be an object and let $d \in D_U(Y_0 \to Z)$ be a diagram
$$\xymatrix{
\partial \Delta^n \ar[r] \ar[d] & Y_0(U)\ar[d]\\
\Delta^n \ar[r] & Z(U).
}$$

Now $R_d$ is a covering sieve of $U$. For every $r:V\to U$ in $R_d$ we have a diagram
$$\xymatrix{
\partial \Delta^n\times \h_V \ar[r] \ar[ddd] & \partial \Delta^n\times \h_U \ar[r]\ar[ddd] & Y_0\ar[d]\\
 & & Y_1\ar[d]\\
 & & Y\ar[d]\\
\Delta^n\times \h_V \ar[r] \ar@{.>}[uurr] \ar[urr]&  \Delta^n\times \h_U \ar[r] & Z,\\}$$
where the dotted map to $Y_1$ is the one that comes from the $s=(U,d,r)\in S$ coordinate of the coproduct in the construction of $Y_1$.
By adjunction we get a diagram of the form
$$\xymatrix{
\partial \Delta^n \ar[r] \ar[dd] & Y_0(U)\ar[d]\ar[r]& Y_0(V)\ar[d]\\
                & Y_1(U)\ar[d]\ar[r]& Y_1(V)\ar[d]\\
\Delta^n \ar[r]\ar[urr] & Z(U)\ar[r] & Z(V),
}$$
which give us the required lift.
\end{proof}

%

\begin{prop}\label{p:SPS_dense}
Let $\cC$ be a small site and $\lambda \geq \max\{|\Mor(\cC)|,\aleph_0\}$ a cardinal. Then the  full sub weak fibration category
$\SPS_\lambda(\cC)\subseteq \SPS(\cC)$ is dense (see Definition \ref{d:sub_weak}).
\end{prop}
\begin{proof}
Let $X\to H \xrightarrow{\mathcal{F}} Y$ be a diagram in $\SPS(\cC)$ such that $X\to Y$ is in $\SPS_{\lambda}(\cC)$.
We need to complete the above diagram to a diagram
$$\xymatrix{
X \ar[dr]\ar[r]  & H'  \ar[r]^{\cF} \ar[d] & Y \\
\empty &  H\ar[ur]_{\cF} & \empty }$$
such that $H'\in \SPS_{\lambda}(\cC)$.
Let us denote $H_0 :=X$. Note that we have a diagram
$$\xymatrix{
X \ar[dr]\ar[r]^{=}  & H_0  \ar[r]^{} \ar[d] & Y.  \\
\empty &  H\ar[ur]_{\cF} & \empty }$$
By Lemma \ref{l:next_small} we get a diagram of the form
$$\xymatrix{
&H_0\ar[d]\ar[ddr]& \\
& H_1\ar[d]\ar[dr]& \\
X \ar[r]\ar[uur]\ar[ur] & H \ar[r] & Y,
}$$
with $H_1$ $\lambda$-bounded and $H_0 \to H_1 \to Y$ a relative local fibration. By applying Lemma \ref{l:next_small} over and over again, we obtain, for every $n\geq 0$, a diagram of the form
$$\xymatrix{
&H_0\ar[d]\ar[ddddr]& \\
& H_1\ar[dddr]\ar[d]& \\
& \vdots\ar[d] &\\
&  H_n\ar[dr]\ar[d]& \\
X\ar[uuuur] \ar[uuur]\ar[ur]\ar[r] & H \ar[r] & Y,
}$$
such that for every $0\leq i\leq n-1$ we have that $H_i \to H_{i+1} \to Y$ is a relative local fibration and $H_i$ is $\lambda$-bounded.
Thus, by Lemma \ref{l:ind_small}, we can take $H': = \colim H_i$.

The proof for a local acyclic fibration is similar.
\end{proof}

\begin{prop}\label{p:SSh_dense}
Let $\cC$ be a small site and $\lambda \geq \max\{2^{|\Mor(\cC)|},\aleph_0\}$ a cardinal. Then the  full sub weak fibration category
$\SSh_\lambda(\cC)\subseteq \SSh(\cC)$ is dense (see Definition \ref{d:sub_weak}).
\end{prop}
\begin{proof}

Let $X\to H \xrightarrow{\mathcal{F}} Y$ be a diagram in $\SSh(\cC)$ such that $X\to Y$ is in $\SSh_{\lambda}(\cC)$.
We need to complete the above diagram to a diagram
$$\xymatrix{
X \ar[dr]\ar[r]  & H'  \ar[r]^{\cF} \ar[d] & Y  \\
\empty &  H\ar[ur]_{\cF} & \empty }$$
such that $H'\in \SSh_{\lambda}(\cC)$.

By Proposition \ref{p:SPS_dense} we can get a diagram
 $$\xymatrix{
X \ar[dr]\ar[r]  & H'  \ar[r]^{\cF} \ar[d] & Y  \\
\empty &  H\ar[ur]_{\cF} & \empty }$$
such that $H'\in \SPS_{\lambda}(\cC)$. Since $Y ,H\in \SSh(\cC)$ it is clear that we have a diagram
$$\xymatrix{ & H'\ar[d]\ar[dr]^{\cF} & \\
X\ar[ur] \ar[dr]\ar[r]  & L^2(H')  \ar[r] \ar[d] & Y,  \\
\empty &  H\ar[ur]_{\cF} & \empty }$$
where  $L^2(H')\in \SSh_{\lambda}(\cC)$ by Lemma \ref{l:sheaf_small}.
It remains to show that the map $L^2(H') \to Y$ is in $\cF$.

Consider the diagram
$$ \xymatrix{
 H' \ar[d] \ar[r]^{\cF} & Y\ar[d]^{\cong} \\
 L^2(H')\ar[ru]\ar[r] & L^2(Y).
}$$
Since $H' \to Y$  is in $\cF$, by  \cite[Corollary 1.8]{Jar}, the map
$L^2(H')\to L^2(Y)$ is in $\cF$ and thus so is the isomorphic map $L^2(H') \to Y$.

The proof for the case of a local acyclic fibration is similar, using the fact that the left vertical map in the last diagram is in $\cF\cW$ by \cite[Lemma 1.6]{Jar}.
\end{proof}

Using Lemma \ref{cardinal_condition} we get immediately from Propositions \ref{p:SPS_dense} and \ref{p:SSh_dense} the following:

\begin{cor}
Let $\cC$ be a small site. Then the weak fibration categories $\SPS(\cC)$ and $\SSh(\cC)$ are homotopically small (see Definition \ref{d:sub_weak}).
\end{cor}

\subsection{The new model structures}\label{s:new_big}
As shown in \cite{Jar1}, Theorems 2 and 5, there exist \emph{proper} model category structures on the categories $\SPS(\cC)$ and $\SSh(\cC)$ in which the weak equivalences are the local weak equivalences. Thus, as relative categories
$(\SPS(\cC),\cW)$ and $(\SSh(\cC),\cW)$ are pro-admissible (see Example \ref{e:proper}).
We have shown in this section that $\SPS(\cC)$ and $\SSh(\cC)$ can also be given (other) weak fibration structures, with the same class of weak equivalences. Furthermore, we have shown in Section \ref{ss:small} that these weak fibration categories are homotopically small. It follows from Theorem \ref{t:model_big} that there are induced model structures on $\Pro(\SPS(\cC))$ and $\Pro(\SSh(\cC))$. Since $\SPS(\cC)$ and $\SSh(\cC)$ are complete and cocomplete, it follows that the same is true for $\Pro(\SPS(\cC))$ and $\Pro(\SSh(\cC))$ (see the remark following Theorem \ref{t:model_big}). Thus $\Pro(\SPS(\cC))$ and $\Pro(\SSh(\cC))$ are actually model categories (but the factorizations may not be functorial).

Consider the inclusion functor
$$i: \SSh(\cC) \hookrightarrow \SPS(\cC).$$
Since $i$ has a left adjoint (namely $L^2$) it commutes with all small limits. Furthermore, $i$ clearly preserves local fibrations and local acyclic fibrations. Thus $i$ is a weak right Quillen functor, and it induces a Quillen adjunction (see Proposition \ref{p:RQFunc_big} and Lemma \ref{l:l_adjoint} (2))
$$ \Pro(L^2) : \Pro(\SPS(\cC))  \leftrightarrows  \Pro(\SSh(\cC)) :\Pro(i). $$
We claim that this Quillen adjunction is a Quillen equivalence. This follows easily from the fact that both $L^2$ and $i$ preserve local weak equivalences, and the unit and counit of the adjunction $L^2 \dashv i$ are also weak equivalences (see \cite[Lemma 1.6]{Jar}).

Consider the sheafification functor
$$L^2: \SPS(\cC) \lrar \SSh(\cC).$$
It is a well-known fact that $L^2$ commutes with finite limits (see for example \cite{Jar}). By \cite[Corollary 1.8]{Jar}, $L^2$ preserves local fibrations, and by \cite[Lemma 1.6]{Jar}, $L^2$ preserves local acyclic fibrations. Thus $L^2$ is a weak right Quillen functor. The functor $L^2$ preserves all small colimits (being a left adjoint to $i$), so in particular $L^2$ is accessible. Thus, by Lemma \ref{l:l_adjoint} (1) and Proposition \ref{p:RQFunc_big}, we have a Quillen adjunction:
$$ L_{L^2} : \Pro(\SSh(\cC))  \leftrightarrows  \Pro(\SPS(\cC)) :\Pro(L^2). $$
This Quillen adjunction can also be shown to be a Quillen equivalence.

\begin{example}\label{e:Verdier}
We now present an application of the model structure constructed in this section. Namely, we give a simple proof of Verdier's hypercovering theorem \cite{SGA4-I}.

We first note that there is a natural functorial path object on $\SPS(\cC)$ given by the simplicial structure. For every object $C\in \SPS(\cC)_f$ we define $P(C)(0)\to P(C)(1)\to P(C)(2)$ to be the path object
$$C\cong C^{\Delta^{0}}\xrightarrow{}  C^{\Delta^{1}}\xrightarrow{}C^{(\Delta^{\{0\}}\coprod\Delta^{\{1\}})}\cong C\times C.$$
It is not hard to verify that the first map above is a local weak equivalence and the second map is a local fibration.
This functorial path object gives rise to the category $\pi \SPS(\cC)_f$. Clearly, for every $A,B\in \SPS(\cC)_f$ we have
$$\Hom_{\pi \SPS(\cC)_f}(A,B)={\SPS(\cC)_f}(A,B)/\sim=\pi_0( \Map_{\SPS(\cC)}(A, B)).$$

Now let $F:\cC^{\op}\to \Ab$ be a presheaf of abelian groups on $\cC$, and let $n\geq 0$. Let $K(-,n):\Ab\to \cS$ be a levelwise fibrant model for the Eilenberg-MacLane functor (see for example \cite{GJ}). Composing this functor with $F$ we get a simplicial presheaf $K(F,n):\cC^{\op}\to \cS$. As explained, for example, in \cite{Bro}, there is a natural isomorphism
$$H^n(\cC,\widetilde{F})\cong \Hom_{\Ho(\SPS(\cC))}(*,K(F,n)),$$
where $H^n(\cC,\widetilde{F})$ is the n'th sheaf cohomology group of the site $\cC$, with coefficients in the (sheaf associated to the) presheaf $F$.

We define $\lambda:=\max\{|\Mor(\cC)|,\aleph_0\}$. Then according to Proposition \ref{p:SPS_dense} the  full sub weak fibration category
$\SPS_\lambda(\cC)\subseteq \SPS(\cC)$ is dense. The category $\SPS_\lambda(\cC)$ is essentially small, ant it clearly contains $*$ and $\phi$.
Thus, using Proposition \ref{p:Ho}, we get that there are canonical isomorphisms
$$H^n(\cC,\widetilde{F})\cong \Hom_{\Ho(\SPS(\cC))}(*,K(F,n)) \cong \colim \limits_{U\in \pi_0(\SPS_\lambda(\cC))_{fw}^{\op}}\Hom_{\pi_0(\SPS(\cC))}(U,K(F,n))$$
$$\cong \colim\limits_{{U\in \pi_0(\SPS_{\lambda}(\cC))_{fw}^{\op}}}\pi_0( \Map_{\SPS(\cC)}(U, K(F,n)))\cong \colim\limits_{{U\in \pi_0(\SPS_{\lambda}(\cC))_{fw}^{\op}}}H^n_{Cech}(\cC,F,U),$$
where the last isomorphism is a classical observation. This is exactly Verdier's theorem, saying that the sheaf cohomology of a site can be computed as the colimit over all hypercoverings in the site of the \v{C}ech cohomologies.
\end{example}

\begin{rem}
In order to get Verdier's theorem, we should restrict this last colimit
only to hypercoverings, that is, to those locally fibrant locally contractible simplicial presheaves, which are levelwise
representable in the \'etale site of $X$. However, since the hypercoverings are cofinal among all the locally fibrant locally contractible simplicial presheaves (\cite[Lemma 2.2]{Jar3}), the resulting colimit is isomorphic.
\end{rem}

\subsection{Comparison with the Isaksen-Jardine model structure}\label{s:compare}
In this subsection we compare our ``projective" model structure on pro-simplicial presheaves of Section \ref{s:new_big}, with the ``injective" model structure on the same category, that can be deduced from \cite{Isa}, when applied to \cite{Jar}. Namely, we show that the identity functors constitute a Quillen equivalence between these two model structures.

As shown in \cite{Jar}, there exists a model structure on the category $\SPS(\cC)$, in which the cofibrations are the levelwise cofibrations, and the weak equivalences are the local weak equivalences. Furthermore, this model structure is proper (see \cite{Jar1}, Theorem 2). It follows from \cite[Theorem 4.15]{Isa} that there exists a model structure on $\Pro(\SPS(\mcal{C}))$, which we will denote by $\Pro(\SPS(\mcal{C}))_I$, such that:
\begin{enumerate}
\item The weak equivalences are $\mathbf{W}_I := \Lw^{\cong}(\cW)$, where $\cW$ is the class of local weak equivalences.
\item The fibrations are $\mathbf{F}_I := \R(\Sp^{\cong}(\mcal{F}_J))$, where $\cF_J$ is the class of fibrations in the Jardine structure on  $\SPS(\cC)$.
\item The cofibrations are $\mathbf{C}_I := {}^{\perp} \Sp^{\cong}(\mcal{F}_J\cap \cW)= {}^{\perp} (\mcal{F}_J\cap \cW)$.
\end{enumerate}
This model structure on $\Pro(\SPS(\mcal{C}))$ was considered by Jardine in \cite{Jar2}.
We call $\Pro(\SPS(\mcal{C}))_I$ the \emph{injective} model structure on $\Pro(\SPS(\mcal{C}))$, since every levelwise cofibration in $\SPS(\mcal{C})$ is a cofibration in $\Pro(\SPS(\mcal{C}))_I$ (between simple objects).

Consider the model structure constructed on $\Pro(\SPS(\mcal{C}))$ in Section \ref{s:new_big}, which we will denote by $\Pro(\SPS(\mcal{C}))_P$. We have that:
\begin{enumerate}
\item The weak equivalences are $\mathbf{W}_P := \Lw^{\cong}(\cW)$, where $\cW$ is the class of local weak equivalences.
\item The fibrations are $\mathbf{F}_P := \R(\Sp^{\cong}(\cF))$, where $\cF$ is the class of local fibrations.
\item The cofibrations are $\mathbf{C}_P := {}^{\perp} \Sp^{\cong}(\cF\cap \cW)= {}^{\perp} (\cF\cap \cW)$.
\end{enumerate}
We call $\Pro(\SPS(\mcal{C}))_P$ the \emph{projective} model structure on $\Pro(\SPS(\mcal{C}))$, since every local fibration (and, in particular, every levelwise fibration) in $\SPS(\mcal{C})$ is a fibration in $\Pro(\SPS(\mcal{C}))_P$ (between simple objects).

Let $f$ be a fibration in the Jardine model structure on $\SPS(\cC)$. Since the Jardine model structure is a left Bousfield localization of the injective model structure on $\SPS(\cC)$ (see, for example, \cite[Section A.3.3]{Lur}), $f$ is also a fibration in the injective model structure on $\SPS(\cC)$. It follows that $f$ is a levelwise fibration in $\SPS(\cC)$, and in particular a local fibration in $\SPS(\cC)$. Thus, $\cF_J\subseteq\cF$. It follows that
$$\mathbf{C}_P = {}^{\perp} (\cF\cap \cW)\subseteq {}^{\perp} (\mcal{F}_J\cap \cW)=\mathbf{C}_I .$$
From this inclusion we conclude trivially that
$$\id:\Pro(\SPS(\mcal{C}))_P\adj \Pro(\SPS(\mcal{C}))_I:\id$$
is a Quillen equivalence between the projective and injective model structures on $\Pro(\SPS(\mcal{C}))$.

\section{The \'Etale Homotopy Type as a Derived Functor}\label{s:etale_big}

Given an algebraic variety $X$, Artin and Mazur defined in \cite{AM} the notion of the \'etale homotopy type of $X$, by applying the connected components functor to the hypercoverings in the \'etale site of $X$. This gives rise to an object in the category $\Pro(\Ho(\cS))$, where $\cS$ is the category of simplicial sets. Artin and Mazur's construction can be easily generalized to any locally connected site $\cC$. However, for many applications it is essential to lift Artin and Mazur's construction from $\Pro(\Ho(\cS))$ to $\Pro(\cS)$. This was achieved by Friedlander in \cite{Fri}, by replacing hypercoverings with rigid hypercoverings. In this section we shall give an alternative solution, by using the model structure described in Section \ref{s:SPS}. This new approach will give a nice description of the \'etale homotopy type as the result of applying a derived functor, and will also have the advantage of working with usual hypercoverings rather than the more involved rigid hypercoverings (see Definition \ref{d:etale_big} and Proposition \ref{p:AM} below). Another advantage is that our construction works over any site.

\begin{prop}\label{p:geo_wrq_big}
Let $\cT = \Sh(\cC)$, ${\cR} = \Sh(\cD)$ be two topoi, and  let
$$f^*:{\cR} \rightleftarrows \cT : f_*$$
be a geometric morphism. Then $f^*$ induces a weak \textbf{right} Quillen functor
$$f^*:{\cR}^{\Del^{\op}}\to \cT^{\Del^{\op}},$$
relative to the local weak fibration structure on simplicial sheaves, described in Section \ref{s:SPS}.
\end{prop}

\begin{proof}
$f^*:{\cR}\to \cT$ preserves finite limits by definition of a geometric morphism, so $f^*:{\cR}^{\Del^{\op}}\to \cT^{\Del^{\op}}$ also preserves finite limits.
Further, since $f^*:{\cR}\to \cT$ preserves local epimorphisms, it follows from Lemmas \ref{l:lf} and \ref{l:laf} that $f^*:{\cR}^{\Del^{\op}}\to \cT^{\Del^{\op}}$ preserves local fibrations and local acyclic fibrations.
\end{proof}

\begin{define}\label{d:etale_big}
Let $\cT$ be a topos. Consider the unique geometric morphism
$$\Gamma^*:\Set \rightleftarrows \cT : \Gamma_*.$$
Here, $\Gamma_*$ is the global sections functor and $\Gamma^*$ is the constant sheaf functor. By Proposition \ref{p:geo_wrq_big}, we have an induced weak right Quillen functor $\Gamma^*:\Set^{\Del^{\op}}\to \cT^{\Del^{\op}}$. The functor $\Gamma^*$ preserves all small colimits (being a left adjoint), so in particular $\Gamma^*$ is accessible. Thus, by Lemma \ref{l:l_adjoint} (1) and Proposition \ref{p:RQFunc_big}, we get a Quillen adjunction:
$$L_{\Gamma^*}:\Pro(\cT^{\Del^{\op}}) \rightleftarrows \Pro(\Set^{\Del^{\op}}):\Pro(\Gamma^*).$$

We define the \emph{topological realization of $\cT$} to be
$$|\cT| := \mathbb{L}L_{\Gamma^*}(*_\cT)\in \Ho(\Pro(\Set^{\Del^{\op}}))=\Ho(\Pro(\cS)),$$
where $*_\cT$ is a terminal object of $\cT^{\Del^{\op}}.$

If $\cC$ is a small Grothendieck site, we define the \emph{topological realization of $\cC$} to be $|\cC|:=|\Sh(\cC)|.$
\end{define}

A case of special interest is when $\cT$ is locally connected, i.e. when $\Gamma^*:\Set\to \cT$ has a \emph{left} adjoint $\Gamma_!:\cT\to \Set$. In geometric situations, the functor $\Gamma_!$ is induced by the functor which sends a scheme to its set of connected scheme-theoretic components. Thus we denote $\pi_0 := \Gamma_!$. By Lemma \ref{l:l_adjoint} (2) and the uniqueness of the left adjoint we get, when $\cT$ is locally connected,
$$L_{\Gamma^*}\cong \Pro(\pi_0).$$
It follows that
$$|\cT|= \mathbb{L}\Pro(\pi_0)(*_\cT).$$
This formula allows us to give a quite concrete description of
$|\cT|$. Recall that in order to compute a left derived functor, one should apply the original functor to a cofibrant replacement. Thus, we should apply $\Pro(\pi_0)$ to a cofibrant replacement of $*_\cT$ in $Pro(\cT^{\Del^{\op}})$.

Recall that in the discussion following Definition \ref{d:path_object}, we presented an explicit construction for a cofibrant replacement to an arbitrary object in a homotopically small pro-admissible weak fibration category. We define $\lambda:=\max\{2^{|\Mor(\cC)|},\aleph_0\}$. Then according to Proposition \ref{p:SSh_dense} the  full sub weak fibration category
$$\cT^{\Del^{\op}}_\lambda :=\SSh_\lambda(\cC)\subseteq \SSh(\cC)=\cT$$
is dense. The category $\cT^{\Del^{\op}}_\lambda$ is essentially small and it clearly contains $*$ and $\phi$. Applying our construction to the terminal object of $\cT^{\Del^{\op}}$ we obtain its cofibrant replacement as a functor
$$H=H_*:\cA_*\xrightarrow{j}\widehat{(\cT^{\Del^{\op}}_\lambda)_{fw}}\subseteq \cT^{\Del^{\op}}.$$
Thus $|\cT|$ is just the composite
$$|\cT|:\cA_*\xrightarrow{j}\widehat{(\cT^{\Del^{\op}}_\lambda)_{fw}}\subseteq \cT^{\Del^{\op}}\xrightarrow{\pi_0}\cS.$$

\begin{example}
Let $\cC$ be a small category, equipped with the trivial topology. Then $\cT=\Sh(\cC)=\Set^{\cC^{\op}}$ is the category of functors $\cC^{\op}\to \Set$. Then $\Gamma^*:\Set\to \Set^{\cC^{\op}}$ is just the diagonal functor. The topos $\cT$ is thus locally connected, since $\Gamma^*$ has a left adjoint which is the colimit functor $\colim=\Gamma_!=\pi_0:\Set^{\cC^{\op}}\to \Set$. By definition
$$|\cC|=|\cT|= \mathbb{L}\Pro(\colim)(*_\cT).$$
The above-defined weak fibration structure on $\cT^{\Delta^{\op}}=\cS^{\cC^{\op}}$ is just the projective model structure on $\cS^{\cC^{\op}}$. Let $E(\cC^{\op})\to *$ be a cofibrant replacement to the terminal object in the projective structure on $\cS^{\cC^{\op}}$. By \cite{Isa} we have that the cofibrations in the projective model structure on $\Pro(\cS^{\cC^{\op}})$ are just $\Lw^{\cong}(\cC of)$, where $\cC of$ are the cofibrations in $\cS^{\cC^{\op}}$. Thus $E(\cC^{\op})\to *$ is also a cofibrant replacement to the terminal object in the projective structure on $\Pro(\cS^{\cC^{\op}})$, and we get that
$$|\cC|=|\cT|= \mathbb{L}\Pro(\colim)(*_\cT)=\Pro(\colim)(E(\cC^{\op}))=$$
$$=\colim_{\cC^{\op}}E(\cC^{\op})=\hocolim_{\cC^{\op}}*\simeq N(\cC^{\op})\simeq N(\cC).$$
\end{example}

\begin{prop}\label{p:AM}
Let $X$ be a locally notherian scheme and $X_{\acute{e}t}$ its \'etale topos.
The natural functor
$$\gamma : \cS \to \Ho(\cS)$$
induces a functor
$$\gamma : \Pro(\cS) \to \Pro(\Ho(\cS)).$$
Then $\gamma(|X_{\acute{e}t}|)$ is isomorphic to the \'etale homotopy type of $X$, defined in \cite{AM}.
\end{prop}
\begin{proof}
Define $\cT:=X_{\acute{e}t}$. First note that $\cT$ is locally connected, so the discussion above concerning locally connected topoi applies.

There is a natural functorial path object on $\cT^{\Del^{\op}}_\lambda$ given by the simplicial structure. For every object $C\in (\cT^{\Del^{\op}}_\lambda)_f$ we define $P(C)(0)\to P(C)(1)\to P(C)(2)$ to be the path object
$$C\cong C^{\Delta^{0}}\xrightarrow{}  C^{\Delta^{1}}\xrightarrow{}C^{(\Delta^{\{0\}}\coprod\Delta^{\{1\}})}\cong C\times C$$
(note that there is no need to take sheafification).
It is not hard to verify that the first map above is a local weak equivalence and the second map is a local fibration.
This functorial path object gives rise to the categories $\pi (\cT^{\Del^{\op}}_\lambda)_f$ and $\pi (\cT^{\Del^{\op}}_\lambda)_{fw}$ as explained in Section \ref{s:Ho}. Clearly, for every $A,B\in (\cT^{\Del^{\op}}_\lambda)_f$ we have
$$\Hom_{\pi (\cT^{\Del^{\op}}_\lambda)_f}(A,B)={(\cT^{\Del^{\op}}_\lambda)_f}(A,B)/\sim=\pi_0( \Map_{\cT^{\Del^{\op}}_\lambda}(A, B)).$$
As we have shown in Section \ref{s:Ho}, $\pi (\cT^{\Del^{\op}}_\lambda)_{fw}$ is cofiltered, and the composition functor
$$\cA_*\xrightarrow{j}\widehat{(\cT^{\Del^{\op}}_\lambda)_{fw}}\to \pi\widehat{(\cT^{\Del^{\op}}_\lambda)_{fw}}\to \pi (\cT^{\Del^{\op}}_\lambda)_{fw}$$
is cofinal. From the commutative diagram
$$\xymatrix{
\cA_*\ar[r]^j\ar@/^2pc/[rrr]^{|\cT|} & \widehat{(\cT^{\Del^{\op}}_\lambda)_{fw}}\ar[d]\ar[r]^{}  & (\cT^{\Del^{\op}}_\lambda)_{fw} \ar[d] ^{\gamma}\ar[r]^{\pi_0} & \cS\ar[d]^{\gamma}& \\
        & \widehat{\pi{(\cT^{\Del^{\op}}_\lambda)_{fw}}}\ar[r]^{} & \pi (\cT^{\Del^{\op}}_\lambda)_{fw}\ar[r]^{\pi_0} & \Ho(\cS)}$$
we see that the pro-object $\gamma|\cT|: \cA_* \to \Ho(\cS)$ factors through $\pi (\cT^{\Del^{\op}}_\lambda)_{fw}$.
Since $\cA_* \to \pi (\cT^{\Del^{\op}}_\lambda)_{fw}$ is  cofinal, it follows from Lemma \ref{l:cofinal} that the pro-object $\gamma|\cT|: \cA_* \to \Ho(\cS)$ is isomorphic, in $\Pro(\Ho(\cS))$, to the pro-object $\pi (\cT^{\Del^{\op}}_\lambda)_{fw}\xrightarrow{\pi_0} \Ho(\cS)$, which is very close to the \'etale homotopy type of $X$ defined in \cite{AM}.

In order to get Artin and Mazur's construction, we should restrict this pro-object only to hypercoverings, i.e. to those locally fibrant locally contractible simplicial sheaves, which are levelwise
a coproduct of representables in the \'etale site of $X$. However, since the hypercoverings are cofinal among all the locally fibrant locally contractible simplicial sheaves (\cite[Lemma 2.2]{Jar3}), the resulting object in  $\Pro(\Ho(\cS))$ is isomorphic.
\end{proof}

\begin{rem}\
\begin{enumerate}
\item Artin and Mazur ignore set theoretical issues, so they do not need to restrict to the small dense subcategory $\cT^{\Del^{\op}}_\lambda$ instead of $\cT^{\Del^{\op}}$.
\item As we have mentioned, Artin and Mazur's construction can be generalized to any locally connected topos $\cT$. Proposition \ref{p:AM} remains valid also in this more general situation, and the proof is exactly the same.
\item Note that in \cite{AM}, Artin and Mazur work in  some localization of  $\Pro(\Ho(\cS))$ (namely, the $\natural$-localization). This localization
also has a model theoretic counterpart, as a localization of our model structure on pro-simplicial sheaves. This will be described in detail in a future paper.
\end{enumerate}
\end{rem}

\subsection{The relative homotopy type}\label{s:relative}
The notion of a relative \'etale homotopy type was considered in \cite{HaSc} as a useful construction for the study
of rational points. However, similarly to Artin and Mazur's \'etale homotopy type, the relative \'etale homotopy type was not given
within a suitable model category. In this section we lift this construction in a suitable way.

\begin{define}\label{d:relative_big}
Let $\cT = \Sh(\cC)$, ${\cR} = \Sh(\cD)$ be two topoi, and  let
$$f^*:{\cR} \rightleftarrows \cT : f_*$$
be a geometric morphism. By Proposition \ref{p:geo_wrq_big}, we have an induced weak right Quillen functor $f^*:{\cR}^{\Del^{\op}}\to \cT^{\Del^{\op}}$. The functor $f^*$ preserves all small colimits (being a left adjoint), so in particular $f^*$ is accessible. Thus, by Lemma \ref{l:l_adjoint} (1) and Proposition \ref{p:RQFunc_big}, we get a Quillen adjunction
$$L_{f^*}:\Pro(\cT^{\Del^{\op}}) \rightleftarrows \Pro({\cR}^{\Del^{\op}}):\Pro(f^*).$$

We define the \emph{relative topological realization of $\cT$ over ${\cR}$} to be
$$|\cT|_{\cR} := \mathbb{L}L_{f^*}(*_\cT)\in \Ho(\Pro({\cR}^{\Del^{\op}})),$$
where $*_\cT$ is a terminal object of $\cT^{\Del^{\op}}$.

If the above geometric morphism corresponds to the morphism of sites $\cC\to\cD$, we also define the \emph{relative topological realization of $\cC$ over $\cD$} to be
$|\cC|_{\cD}:=|\cT|_{\cR}.$
\end{define}

As in the case ${\cR}=\Set$, if the geometric morphism $f^*:{\cR} \rightleftarrows \cT : f_*$ is \emph{essential}, i.e. if $f^*:{\cR}\to \cT$ has a \emph{left} adjoint $f_!:\cT\to {\cR}$, we have
$$L_{f^*}\cong \Pro(f_!),$$
and this allows us to give a very concrete description of $|\cT|_{\cR}$.
There is also an analogue of Proposition  \ref{p:AM}, if we are only interested in the image of $|\cT|_{\cR}$ in $\Pro(\Ho({\cR}^{\Del^{\op}}))$.

\begin{rem}
The additional generality of working with not necessarily essential geometric morphisms is useful. Perhaps the simplest example is given by the geometric morphism
$$f_*: \Set^H  \to \Set^G  $$
induced by the inclusion $H\subseteq G$, where  $G$ is a pro-finite group and $H$ is a closed subgroup of infinite index. (Consider, for example, $G$ to be an absolute Galois group of $\mathbb{Q}$ and $H$ a decomposition group of some prime $p$.) If $G$ and $H$ were discrete, the desired left adjoint to $f^*$ would be the functor
$$X \mapsto X \times_H G.$$
However, in the pro-finite case $f^*$ need not have a left adjoint. This kind of setting is used by the second author and V. Stojanoska in order to give a Poitou-Tate duality for spectra, and will appear in an upcoming paper.
\end{rem}

It is not hard to check that we get a functor
$$|\bullet|_{\cR}: \cT opoi/{\cR} \to \Ho(\Pro({\cR}^{\Del^{\op}})),$$
where $\cT opoi$ is the category of topoi, and geometric morphisms between them (which is equivalent to the category of small sites, and morphisms of sites between them).

It is easy to verify that for every topos ${\cR}$ we have $|{\cR}|_{\cR} \simeq *_{\cR}$.
Thus, by the functoriality of $|\bullet|_{\cR}$, we have a map
$$h:\cT({\cR})\to [*_{\cR},|\cT|_{\cR}]_{\Pro({\cR}^{\Del^{\op}})},$$
where $\cT({\cR})$ is the set of geometric morphisms $s_*:{\cR} \to \cT$ which are sections of the map $f_*:\cT\to {\cR}$.
The codomain of $h$ above has an obstruction theory and a Bousfield-Kan type spectral sequence, so the map $h$ can be used to study sections of maps of topoi.
For example, if the codomain of $h$ can be shown to be empty, then we know that $\cT({\cR})$ is empty, or in other words that $f$ has no section.
A case of special interest is when $f_*$ is the morphism of \'etale topoi induced by a scheme morphism $X\to spec(K)$. Then, a section of $f$ is just a $K$-rational point of $X$. We elaborate more on this in the next subsection.

\subsection{Rational points}\label{s:rational}
The work presented in this paper originated from the motivation of finding a suitable model structure in which the general machinery of abstract homotopy theory can be used to define and study obstructions to the existence of rational points. Such obstructions were studied without the framework of a model structure by Y. Harpaz and the second author in  \cite{HaSc} and by Ambrus P\'al in \cite{Pal}.
In \cite{HaSc}, Harpaz and the second author defined a notion of a relative \'etale homotopy type of a variety $X/K$ over a field $K$. This construction was then used to study rational $K$-points on $X$, by using some notion of homotopy fixed points.
However, similar to the construction of Artin and Mazur in \cite{AM}, the construction in \cite{HaSc} is homotopical rather then topological, namely, it gives an object in $\Pro(\Ho((\Spec K)_{\acute{e}t}^{\Del^{\op}}))$ rather than $\Pro((\Spec K)_{\acute{e}t}^{\Del^{\op}})$. Furthermore, the above notions are defined  by  ad-hoc constructions, and are not given conceptual definitions in a suitable model category. The construction of the relative topological realization presented here gives an object in $\Pro((\Spec K)_{\acute{e}t}^{\Del^{\op}})$, and allows us to define the above  notions using the language of model categories.

Indeed, let $K$ be a field and let $X/K$ be a $K$-variety. Note that
$$(\Spec K)_{\acute{e}t}^{\Del^{\op}}=\cS^{\Gamma_K}$$
is just the weak fibration category of simplical sets with a continuous $\Gamma_K$ action considered in Example \ref{e:profinite} in the introduction, where $\Gamma_K$ is the absolute Galois group of $K$.
We can define
$$Top_K(X):= |X_{\acute{e}t}|_{(\Spec K)_{\acute{e}t}}\in \Ho\Pro(\cS^{\Gamma_K}),$$
and we get a map
$$h:X(K) \to [*,Top_K(X)]_{\Pro(\cS^{\Gamma_K})}.$$

The map $h$ is closely related to the map $h:X(K) \to X(hK)$, presented in \cite{HaSc} and \cite{Pal}, and can be used to study rational points.
For example, if we define
$$X(hK):= [*,Top_K(X)]_{\Pro(\cS^{\Gamma_K})},$$
then the emptiness of $X(hK)$ is an obstruction to the existence of a $K$-rational point on $X$.
By using the Postnikov filtration on $Top_K(X)$ one can obtain a series of obstructions
$$o_n \in H^{n+1}_{cont-Gal}(K,\pi_{n}(Top_K(X)))$$
to the non-emptiness of $X(hK)$. The results are ``higher Grothendieck obstructions" that generalize the Grothendick section  obstruction which employs the \'etale fundamental group. Furthermore, having a ``topological" object and a model structure allows one to use the general machinery of model categories in order to give simpler and more conceptual proofs to the results in \cite{HaSc}. This also enables one to generalize the homotopical obstruction theory of \cite{HaSc}, from fields to arbitrary base schemas. This approach will be discussed in future papers.

As a first illustration we formulate the main result of \cite{HaSc} using the language and notation presented here. Suppose that $K$ is a number field and let $\mathbb{A}_K$ denote its adele ring. Then an adelic point in $X$ is just a map $\Spec\mathbb{A}_K \to X$ of schemes over $\Spec K$. The set of adelic points is denoted $X(\mathbb{A}_K)$. Applying the functor $|\bullet|_{(\Spec K)_{\acute{e}t}}$ we get a map
$$X(\mathbb{A}_K)\to [|(\Spec \mathbb{A}_K)_{\acute{e}t}|_{(\Spec K)_{\acute{e}t}},Top_K(X)]_{\Pro(\cS^{\Gamma_K})}=:X(h\mathbb{A}_K).$$
We clearly have a commutative diagram of sets
$$\xymatrix{X(K)\ar[r]\ar[d] & X(hK)\ar[d] \\
X(\mathbb{A}_K)\ar[r] & X(h\mathbb{A}_K).}$$
We can now formulate the main result in \cite{HaSc}.

\begin{thm}[Harpaz-Schlank]
If $X$ is a smooth geometrically connected variety over $K$ then the image of the natural map
$$X(hK)\times_{X(h\mathbb{A}_K)}X(\mathbb{A}_K)\to X(\mathbb{A}_K)$$
is exactly the \'etale Brauer obstruction to the existence of a $K$-rational point in $X$ (defined in \cite{Sko}).
\end{thm}

\subsection{Embedding problems}
The approach described above to obstructing sections can be used in other contexts as well. For example, let $K$ be any field and let $\Gamma_K$ be the absolute Galois group of $K$. Consider a diagram of pro-finite groups
$$\xymatrix{
&      &   & \Gamma_K\ar@{->>}[d]\ar^l@{-->>}[dl] & \\
1\ar[r] & G\ar[r] & L \ar@{->>}[r] & M\ar[r] & 1.
}$$
For finite $L$, in the context of Galois theory, a surjective lift $l$ in the diagram above is called a ``solution to an embedding problem".
We would like to give an overview of how one can use the relative topological realization of topoi to obstruct solutions to embedding problems.
Indeed, let $$\Gamma := \Gamma_K \times_M L.$$
Then we obtain a short exact sequence of pro-finite groups
$$\xymatrix{
1\ar[r] & G\ar[r] & \Gamma \ar@{->>}[r]^{f} & \Gamma_K\ar[r] & 1.
}$$
The map $f$ has a section iff there exists a lift $l$  (not necessarily  surjective) in the diagram above.
Now we have an induced geometric morphism of topoi of sets with a continuous action
$$f_{*}: \Set^{\Gamma} \to \Set^{\Gamma_K}.$$
If we denote $\cT:=\Set^{\Gamma}$ and ${\cR}:=\Set^{\Gamma_K}$, the existence of  a lift $l$ gives a section to $f_{*}$ and thus an element in $[*, |\cT|_{\cR}]_{\Pro(\cS^{\Gamma_K})}$ (called a $\Gamma_K$-homotopy fixed point in $|\cT|_{\cR}$).
To compute  $|\cT|_{\cR}$, note that $*^{cof} \in \Pro(\cS^{\Gamma})$ is a pro-diagram of Kan-contractible simplical  sets with  continuous $\Gamma$-action.  Since $L$ is a finite quotient of $\Gamma$ we can restrict to a cofinal diagram and get that all spaces in the diagram have a map to $E(L)$. Thus they all have a free $G$-action.  Now, it is easy to see that $f_{!}$ (the left adjoint to $f_{*}$) exists in this case, and it is just the functor of taking $G$-orbits.
Thus, all $\Gamma_K$ spaces in the diagram of $|\cT|_{\cR}$ are homotopic to $B(G)$ and we get that $|\cT|_{\cR}$ is equivalent to the simple object
$E(L /G)$, which is a form of $B(G)$ with  a $\Gamma_K$-action.
Let
$$X \mapsto \mathbb{Z}^1 X:\cS^{\Gamma_K}\to \cS^{\Gamma_K}$$
be the functor obtained by levelwise taking a set $A$ to the set of formal sums of elements of $A$ of total degree 1.
It is a classical fact that $\pi_n(\mathbb{Z}^1 X) = \tilde{H}_n(X)$ and that the natural map $X \to \mathbb{Z}^1 X$ realizes the Hurewtiz map.
A lift $l$ gives rise to a section for $f_*$, which in turn gives rise to a homotopy fixed point in $|\cT|_{\cR}$ and thus also in
$\mathbb{Z}^1|\cT|_{\cR}$. By using the Postnikov filtration we get a  Bousfield-Kan type obstruction theory. That is,
there exists a sequence of obstructions $o_1,o_2,\dots$ to the existence of a lift  $l$ such that
$$o_i \in H^{i+1}(K,H_i(G,\mathbb{Z})).$$

The second author used this obstruction theory to show that some non-Abelian groups cannot be  the Galois group of an unramified extension of certain number fields (to appear in a future paper).


Department of Mathematics, University of Muenster, Nordrhein-Westfalen, Germany.
\emph{E-mail address}:
\texttt{ilanbarnea770@gmail.com}

Department of Mathematics, Massachusetts Institute of Technology, Massachusetts, USA.
\emph{E-mail address}:
\texttt{schlank@math.mit.edu}

\end{document}